\newif \ifJ
\ifJ
\else
\documentclass[8pt]{article}
\usepackage[a4paper, total={5.5in, 8in}]{geometry}
\usepackage[numbers]{natbib}
\usepackage{amsmath,amsfonts,amsthm,amssymb}
\usepackage{mathrsfs}
\usepackage[breaklinks,bookmarks=false]{hyperref}
\hypersetup{colorlinks, linkcolor=blue, citecolor=blue,
urlcolor=blue, plainpages=false, pdfwindowui=false,
pdfstartview={FitH}}
\usepackage{cleveref,color,enumitem,tensor,mathtools}
\usepackage{hyperref}
\usepackage[only,llbracket,rrbracket,llparenthesis,rrparenthesis]{stmaryrd} 
\usepackage[textsize=tiny,color=blue!10!white]{todonotes}
\usepackage{bm}
\usepackage[medium,compact]{titlesec}
\titleformat*{\section}{\large\bfseries}
\titleformat*{\subsection}{\bfseries}
\RequirePackage{subfig}
\RequirePackage{etex}
\RequirePackage{tikz}
\RequirePackage{graphicx}
\RequirePackage{pgfplots}
\RequirePackage{epstopdf}
\usepgfplotslibrary{groupplots}
\pgfplotsset{compat=newest}
\usepackage{bigints}
\pgfplotsset{
    every axis plot/.append style = {font = \scriptsize}
  }
\usepackage{cleveref}
\usepackage{float}
\numberwithin{equation}{section}
\newtheorem{theorem}{Theorem}[section]{\bfseries}{\it}
\newtheorem{proposition}[theorem]{Proposition}{\bfseries}{\it}
\newtheorem{lemma}[theorem]{Lemma}{\bfseries}{\it}
\newtheorem{corollary}[theorem]{Corollary}{\bfseries}{\it}
{\bfseries}{\it}
\newtheorem{example}{Example}[section]{\bfseries}{\rmfamily}
{\bfseries}{\it}
\theoremstyle{definition}
\newtheorem{remark}{Remark}[section]{\bfseries}{\rmfamily}

\fi

\newcommand{\R}{\mathbb{R}}
\newcommand{\N}{\mathbb{N}}

\DeclareMathOperator{\diam}{diam}
\DeclareMathOperator{\argmax}{arg\,max}
\DeclareMathOperator{\Div}{div}

\renewcommand{\dim}{d}
\newcommand{\Res}{\mathcal{R}}
\newcommand{\abs}[1]{\lvert#1\rvert}
\newcommand{\norm}[1]{\left\lVert#1\right\rVert}
\newcommand{\Mfty}{M_\infty}
\newcommand{\Rdd}{\R^{\dim\times\dim}}

\newcommand{\Rddsymp}{\R^{\dim\times\dim}_{\mathrm{sym},+}}
\newcommand{\T}{\mathcal{T}}

\newcommand{\tends}{\rightarrow}
\newcommand{\dx}{\mathrm{d}x}
\newcommand{\ds}{\mathrm{d}s}

\newcommand{\etastabi}{\eta_{\mathrm{stab},i}}
\newcommand{\jump}[1]{\left\llbracket #1 \right\rrbracket}

\newcommand{\oscKqi}{\mathrm{osc}_{K,\kappa,i}}
\newcommand{\oscKqH}{\mathrm{osc}_{K,\kappa,1}}
\newcommand{\oscKqF}{\mathrm{osc}_{K,\kappa,2}}

\newcommand{\calV}{\mathcal{V}}
\newcommand{\calVi}{\mathcal{V}^I}
\newcommand{\calF}{\mathcal{F}}
\newcommand{\calFi}{\mathcal{F}^I}
\newcommand{\calFiK}{\mathcal{F}^I_K}
\newcommand{\calFKi}{\calFiK}
\newcommand{\calFinK}{\mathcal{F}^{I,N}_{K}}
\newcommand{\calFiz}{\mathcal{F}^I_z}
\newcommand{\calE}{\mathcal{E}}
\newcommand{\calEKI}{\mathcal{E}_K^I}
\newcommand{\TK}{\T_K}
\newcommand{\tTK}{\widetilde{\T}_K}
\newcommand{\Tz}{\T_z}
\newcommand{\TF}{\T_F}

\newcommand{\Resi}{R_i}
\newcommand{\RH}{R_1}
\newcommand{\RF}{R_2}
\newcommand{\etaKi}{\eta_{K,i}}
\newcommand{\etaKH}{\eta_{K,1}}
\newcommand{\etaKF}{\eta_{K,2}}
\newcommand{\etaResi}{\eta_{\mathrm{res},i}}
\newcommand{\rKi}{r_{K,i}}
\newcommand{\rKH}{r_{K,1}}
\newcommand{\rKF}{r_{K,2}}
\newcommand{\ji}{j_{F,i}}
\newcommand{\jH}{j_{F,1}}
\newcommand{\jF}{j_{F,2}}

\newcommand{\omz}{\omega_z}

\newcommand{\omK}{\omega_K}
\newcommand{\edgeweight}{\gamma_{\T,E}}

\newcommand{\IT}{I_\T}

\newcommand{\GammaIn}{\Gamma_{\mathrm{in}}}
\newcommand{\GammaD}{\Gamma_{\mathrm{D}}}
\newcommand{\GammaN}{\Gamma_{\mathrm{N}}}

\newcommand{\TITLESTRING}{A posteriori error bounds for finite element approximations of steady-state mean field games}

\newcommand{\ABSTRACTSTRING}{We analyze \emph{a posteriori} error bounds for stabilized finite element discretizations of second-order steady-state mean field games.
We prove the local equivalence between the $H^1$-norm of the error and the dual norm of the residual. We then derive reliable and efficient estimators for a broad class of stabilized first-order finite element methods.
We also show that in the case of affine-preserving stabilizations, the estimator can be further simplified to the standard residual estimator.
Numerical experiments illustrate the computational gains in efficiency and accuracy from the estimators in the context of adaptive methods.}

\begin{document}

\ifJ

\else

\title{\TITLESTRING}
\author{Yohance A. P. Osborne\footnotemark[1]~, Iain Smears\footnotemark[2]~ and Harry Wells\footnotemark[2]}

\maketitle

\renewcommand{\thefootnote}{\fnsymbol{footnote}}

\footnotetext[1]{Department of Mathematical Sciences, Durham University, Stockton Road, DH1 3LE Durham, United Kingdom (\texttt{yohance.a.osborne@durham.ac.uk}).}

\footnotetext[2]{Department of Mathematics, University College London, Gower
Street, WC1E 6BT London, United Kingdom (\texttt{i.smears@ucl.ac.uk}, \texttt{harry.wells@ucl.ac.uk}).}

\begin{abstract}
\ABSTRACTSTRING
\end{abstract}

\fi

%

\section{Introduction}\label{sec-1-introduction}
Mean field games (MFG), which were introduced by{\ifJ\else~Lasry and Lions\fi}~\cite{lasry2006jeuxI,lasry2006jeuxII,lasry2007mean} and independently
by{\ifJ\else~Huang, Caines, and Malham\'{e}\fi}~\cite{huang2006large}, model Nash equilibria of rational games of stochastic
optimal control in which there are infinitely many players. The equilibria are characterized by a system of partial differential equations (PDE) that consist of a Hamilton--Jacobi--Bellman (HJB) equation coupled with a Kolmogorov--Fokker--Planck (KFP) equation.
Comprehensive reviews of MFG, including analysis and applications, are well-documented in the literature{\ifJ,~c.f.~\cite{GueantLasryLions2003,GomesSaude2014,achdou2021mean}\else~\cite{achdou2021mean,GomesSaude2014,GueantLasryLions2003}\fi}.
We consider as a model problem a quasilinear elliptic stationary MFG system of the form
\begin{equation}\label{sys}
\begin{aligned}
- \nu\Delta u+H(x,\nabla u)&=F[m](x)  &&\text{in }\Omega,
\\
-\nu\Delta m -\Div\left(m\frac{\partial H}{\partial p}\left(x,\nabla u\right)\right)&= G(x)  &&\text{in }\Omega,
\end{aligned}
\end{equation}
along with homogeneous Dirichlet boundary conditions $u=0$ and $m=0$ on $\partial \Omega$, where $\Omega\subset \mathbb{R}^d$ denotes the state space of the game, which is assumed to be a bounded, connected, polytopal open set with Lipschitz boundary $\partial \Omega$.
Here $u$ denotes the value function of the underlying optimal control problem faced by the players, and $m$ denotes the density of players in the state space of the game. 
This model problem can be interpreted as a steady-state solution for a game where the players exit the game upon reaching the boundary, which occurs for instance in applications to traffic flow and evacuation planning. 

The numerical analysis of MFG is an active area of research, with early approaches including works on finite difference methods~{\ifJ\cite[]{achdou2010mean,achdou2013mean}\else\cite{achdou2013mean,achdou2010mean}\fi} and semi-Lagrangian methods~\cite[]{CarliniSilva14,CarliniSilve15}.
More recently, finite element methods (FEM) have been proposed and analysed for rather general MFG with nondifferentiable Hamiltonians in~\cite{osborne2024analysis,osborne2024erratum} for the steady-state case, and in~\cite{osborne2023finite}, for the time-dependent case.
The first optimal rates of convergence of piecewise-affine FEM were shown in~\cite{osborne2024near}, for MFG with regular Hamiltonians.
In these works, the FEM are stabilized to enforce a discrete maximum principle.
See also~\cite{berry2024approximation} for an analysis of convergence rates based on a different approach. 
So far, existing works have treated only the \emph{a priori} error analysis of numerical methods for MFG, i.e.\ studying the convergence in the small-mesh limit and error bounds in terms of the unknown solution and the discretization parameters.
In this paper, we focus on \emph{a~posteriori} error analysis for MFG, which uses the computed numerical solution in order to obtain computable bounds on the error. We refer the reader to the textbook~\cite{verfurth2013posteriori} for an introduction to \emph{a posteriori} error analysis and its applications.
A key motivation is that \emph{a~posteriori} error estimators are essential ingredients in adaptive FEM for guiding the mesh refinements.

To give a summary of our main results, we focus here on MFG systems of the form~\eqref{sys} where the Hamiltonian $H$ is $C^{1,1}$ with respect to $p$, where $F$ is Lipschitz continuous and strongly monotone on a suitable space, and where the datum $G$ is nonnegative in the sense of distributions; see Section~\ref{sec:Notation&Setting} below for details. 
Under these assumptions, the weak formulation of the MFG system is well-posed with existence and uniqueness of the solution, see~Remark~\ref{rem-continuous-pb-basic-properties-C1H} below for further details.
Concerning the numerical methods, our analysis covers a very broad class of stabilizations, which includes, but is not limited to, the ones from \cite{osborne2024analysis,osborne2023finite,osborne2024near}. 
Indeed, we treat abstract stabilizations that are only required to be linear with respect to discrete test functions and enforce nonnegativity of the density approximation $m_\T$, see Section~\ref{sec:stablized-fem} below. This allows the analysis to cover many of the common stabilization techniques in the literature that enforce a discrete maximum principle, such as upwinding {\ifJ\cite[]{tabata1977,baba1981conservation}\else\cite[]{baba1981conservation,tabata1977}\fi}, flux-limiting \cite[]{burman2002nonlinear}, edge-averaging \cite[]{xu1999monotone}, edge-based nonlinear diffusion \cite[]{barrenechea2017edge}, and algebraic flux correction \cite[]{barrenechea2018unified}.
In particular, our analysis allows for nonlinear stabilizations.
For a given computational mesh $\T$ and numerical approximation $(u_\T,m_\T)\in \left[V(\T)\right]^2$, where $V(\T)$ is the $H^1_0$-conforming P1 FEM space associated to $\T$, we derive an estimator $\eta(u_\T,m_\T)$ that combines standard residual estimators with some additional terms coming from the stabilization in the FEM.
Our main results, given in Theorems~\ref{thm:error_upper_bound} and~\ref{theorem:eta-lower-bound}, and Corollary~\ref{corollary:eta-lower-bound} below, are the first \emph{a posteriori} error bounds for MFG, showing that
\begin{align}
&\norm{u-u_\T}_{H^1(\Omega)} + \norm{m-m_\T}_{H^1(\Omega)} \lesssim \eta(u_\T,m_\T),
\label{eq:intro:main-error-bound}
\\ 
 &    \eta(u_\T,m_\T)  \lesssim \norm{m -m_\mathcal{T}}_{H^1(\Omega)} + \norm{u -u_\T}_{H^1(\Omega)}  + \mathrm{oscillation},
\label{eq:intro:main_lower_bound}
\end{align}
where the bound~\eqref{eq:intro:main-error-bound} is subject to the additional condition that $\eta(u_\T,m_\T)\leq \widetilde{\mathcal{R}}_*$ for some problem-dependent constant $\widetilde{\mathcal{R}}_*>0$.
This smallness condition is due to the nonlinearity and noncoercivity of the problem.
The oscillation terms that are alluded to in~\eqref{eq:intro:main_lower_bound} are made precise in Section~\ref{sec:general-estimates} below; roughly speaking, these terms relate to the approximability properties of the problem data appearing in the HJB and KFP equations.
The bounds~\eqref{eq:intro:main-error-bound} and~\eqref{eq:intro:main_lower_bound} show that the estimator is \emph{reliable} and \emph{efficient}, so that, up to data oscillation terms and generic constants, it is equivalent to the error.
Our analysis is founded on original analysis of the MFG system~\eqref{sys} in the form of a local equivalence result between the $H^1$-norm of the error and the $H^{-1}$-norm of the residual, see Theorem~\ref{thm:error-res-equiv} below.

In this paper, we also contribute towards the wider analysis of error estimators for stabilized FEM.
It is well-known that, even for linear PDE, adding stabilization to a FEM leads to a loss of Galerkin orthogonality of the discrete approximations, which in turn leads to challenges in attempting to localize the estimators. This issue is of wider interest beyond MFG.
In~\cite{TobiskaVerfurth2015}, \emph{a posteriori} error estimators of residual-type were obtained for several examples of stabilizations, where they showed that in some cases, one can avoid additional terms in the estimators coming from the stabilization thus leading to a final estimator that is the same as for the (unstabilized) Galerkin method.
Note however that the stabilizations treated in~\cite{TobiskaVerfurth2015} typically do not lead to discrete maximum principles and are thus not considered here.
Nevertheless, this motivates our analysis in Section~\ref{sec:affine_preserving_stab}, where we turn from the very general class of stabilizations of Section~\ref{sec:general-estimates} to a more restricted class with additional structure, the most important one being \emph{patchwise affine preservation}, see~\ref{H:stabilization_lin_preserving} in Section~\ref{sec:affine_preserving_stab} below.
Note that affine-preserving stabilizations have been considered in the literature {\ifJ\cite[]{barrenechea2017edge,barrenechea2018unified,barrenechea2024finite}\else\cite[]{barrenechea2017edge,barrenechea2024finite,barrenechea2018unified}\fi}, many of which are nonlinear stabilizations.
Our key result is that the stabilization-related terms that enter~$\eta(u_\T,m_\T)$ are bounded by the standard jump estimators, see Theorem~\ref{thm:stab_jump_bound} below.
Therefore, the stabilization-related terms can be removed, and the standard residual estimators are reliable and efficient, with the further advantage of being localizable.
The benefits include simplifications in the computation of the error estimators and in the implementation of adaptive mesh-refinements.

We present numerical experiments that support our theoretical results and also test the performance of the estimators when relaxing some of the assumptions used in the analysis, including several examples of challenging MFG for which the exact solution is not known analytically.
The experiments demonstrate the improvements in computational efficiency and accuracy when using the estimators in adaptive computations for problems on nonconvex domains with solutions with limited regularity.

The outline of this paper is as follows. In Section \ref{sec:Notation&Setting}, we outline the basic notation of this paper along with the regularity requirements on the MFG system for our estimates. 
In Section \ref{sec:continous-equivalence}, we present the equivalence result between residual and error in the continous setting. 
The class of stabilized FEMs for the MFG system are formulated in Section \ref{sec:stablized-fem}. The \textit{a posteriori} error bounds are proven for general stabilizations in Sections~\ref{sec:general-estimates}. 
Section~\ref{sec:affine_preserving_stab} treats the analysis of estimators for affine-preserving stabilizations.
Numerical experiments are presented in Section \ref{sec:numerics}.

\section{Notation and Setting}\label{sec:Notation&Setting}

\paragraph{Basic notation.} We denote $\mathbb{N}\coloneqq \{1,2,3,\cdots\}$.
For a Lebesgue measurable set $\omega\subset \R^\dim$, $\dim\in\N$, we let $|\omega|_\dim$ denote its Lebesgue measure.
The inner products on $L^2(\omega)$ and $L^2(\omega;\R^\dim)$ are denoted commonly by~$(\cdot,\cdot)_{\omega}$, with induced norm denoted by $\norm{\cdot}_\omega$; there will be no risk of confusion as the two cases will be distinguished by the arguments.
We equip the space  $L^{\infty}(\omega;\mathbb{R}^{d\times d})$ of bounded matrix-valued functions on $\omega$ with the essential supremum norm $\|\cdot\|_{L^{\infty}(\omega;\mathbb{R}^{d\times d})}$ that is induced by the Frobenius norm on $\mathbb{R}^{d\times d}$.

Let $\Omega$ be a bounded, connected, open subset of $\mathbb{R}^d$, $d\in\N$, with Lipschitz boundary $\partial \Omega$. We will also require that $\Omega$ be polytopal in Section~\ref{sec:stablized-fem} below.
Let $H^1(\Omega)$ denote the Sobolev space which consists of functions in  $L^2(\Omega)$ that have first-order weak partial derivatives also in $L^2(\Omega)$. We let $H^1_0(\Omega)$ denote the closure of $C_{0}^{\infty}(\Omega)$, the space of smooth compactly supported functions on $\Omega$, in the norm on $H^1(\Omega)$, and we let $H^{-1}(\Omega)$ denote the space of continuous linear functionals on $H_0^1(\Omega)$. 
The canonical norm on $H^{-1}(\Omega)$ is defined by
\begin{equation}
\norm{\Phi}_{H^{-1}(\Omega)}\coloneqq \sup_{\substack{v\in H^1_0(\Omega)\\\norm{\nabla v}_\Omega =1}} \langle \Phi, v\rangle_{H^{-1}\times H^1_0} \quad \forall\Phi \in H^{-1}(\Omega),
\end{equation}
where $\langle\cdot,\cdot\rangle_{H^{-1}\times H^1_0}$ denotes the duality pairing between $H^{-1}(\Omega)$ and $H^1_0(\Omega)$.

\paragraph{Notation for inequalities.}
In order to avoid the proliferation of generic constants, in the following, we write $a\lesssim b$ for real numbers $a$ and $b$ if $a\leq C b$ for some constant $C$ that may depend on the problem data and the shape-regularity parameter of the computational mesh $\T$ appearing below, but is otherwise independent of the mesh $\T$, the mesh-size, and the numerical solution. Throughout this work we will regularly specify the particular dependencies of the hidden constants.

\paragraph{Problem data.}

We make the following assumptions on the data appearing in \eqref{sys}. 
The diffusion $\nu>0$ is assumed to be constant. We assume that the source term $G\in H^{-1}(\Omega)$ is of the form $G=g_0-\nabla {\cdot}g_1$ with $g_0\in L^{q/2}(\Omega)$ and $g_1\in L^q(\Omega;\mathbb{R}^d)$ for some $q>d$, and that $G\in H^{-1}(\Omega)$ is nonnegative in the distributional sense, i.e.\ 
$\langle G, \phi\rangle_{H^{-1}\times H_0^1} \geq 0$ for all functions $\phi\in H_0^1(\Omega)$ that are nonnegative a.e.\ in~$\Omega$. 
We suppose that the coupling operator $F$ is a Lipschitz continuous operator from some given Hilbert space $\mathcal{H}$ into $H^{-1}(\Omega)$:  
\begin{equation} 
\|F[w]-F[v]\|_{H^{-1}(\Omega)}\leq L_F\|w-v\|_{\mathcal{H}} \quad\forall w,\,v\in  \mathcal{H},\label{F2}  
\end{equation}
for some constant $L_F\geq 0$.
We assume here that the space $\mathcal{H}$ is such that $H^1_0(\Omega)$ is continuously and compactly embedded into $\mathcal{H}$. For example, $\mathcal{H}=L^2(\Omega)$ is a typical situation that is included in the analysis.
We also assume that $F$ is \emph{strongly monotone} with respect to the given Hilbert space $\mathcal{H}$, at least when $F$ is restricted to the smaller space $H^1_0(\Omega)$. In other words, we assume that there exists a constant $c_F>0$ such that
\begin{equation}\label{strong-mono-F-C1H}
c_F \norm{w-v}_{\mathcal{H}}^2 \leq \langle {F}[w]-{F}[v], w-v\rangle_{H^{-1}\times H_0^1}\quad \forall w,v\in H_0^1(\Omega).
\end{equation}
Note that it is natural to restrict the arguments of $F$ in~\eqref{strong-mono-F-C1H} to the smaller space~$H^1_0(\Omega)$ owing to the fact that $F$ has images in~$H^{-1}(\Omega)$.
Several examples of operators $F$ that satisfy the above conditions are given in \cite[Example~1]{osborne2024analysis}. 
Observe that $F$ is allowed to be a \emph{nonlocal} operator, thus allowing for applications with global interactions between players of the game via the cost term. 
We also remark from the onset that whereas we allow distributional $F$ and $G$ in the analysis of the continuous problem in Section~\ref{sec:continous-equivalence}, we will make stronger assumptions in Sections~\ref{sec:general-estimates} and~\ref{sec:affine_preserving_stab} when dealing with the discrete estimators, namely that $F$ maps into $L^2(\Omega)$ and that $G\in L^2(\Omega)$, to avoid some unnecessary technicalities.
 
In applications, the Hamiltonian $H$ is usually specified in terms of the controlled drift and running cost of the associated optimal control problem in the MFG. Therefore, we assume that
\begin{equation}\label{Hamiltonian}
H(x,p)\coloneqq\sup_{\alpha\in\mathcal{A}}\left(b(x,\alpha)\cdot p-f(x,\alpha)\right)\quad \forall (x,p)\in\overline{\Omega}\times\mathbb{R}^d,
\end{equation} 
where the set $\mathcal{A}$ is the set of controls and is assumed to be a compact metric space, and where both the control-dependent drift $b:\overline{\Omega}\times\mathcal{A}\to\mathbb{R}^d$ and the control-dependent component of the running cost $f:\overline{\Omega}\times\mathcal{A}\to\mathbb{R}$ are assumed to be uniformly continuous.
It follows from the above hypotheses that~$H$ satisfies the following bounds 
\begin{subequations}\label{bounds}
\begin{alignat}{2} 
|H(x,p)|&\leq C_H\left(|p|+1\right) & \quad &\forall  (x,p)\in\overline{\Omega}\times\mathbb{R}^d,\\ 
|H(x,p)-H(x,q)|&\leq L_H|p-q|&&\forall(x,p,q)\in\overline{\Omega}\times\mathbb{R}^d\times\mathbb{R}^d,\label{bounds:lipschitz}
\end{alignat} 
\end{subequations}
with
$C_H\coloneqq \max\left\{\|b\|_{C(\overline{\Omega}\times\mathcal{A};\mathbb{R}^d)},\|f\|_{C(\overline{\Omega}\times\mathcal{A})}\right\}$ and $ L_H\coloneqq\|b\|_{C(\overline{\Omega}\times\mathcal{A};\mathbb{R}^d)}.$ Note that $H$ is also convex w.r.t.\ $p$ uniformly in $x\in\overline{\Omega}$.
We assume that $H$ is continuously differentiable w.r.t.\ $p$ with $\frac{\partial H}{\partial p}\colon \Omega\times\mathbb{R}^d\rightarrow \mathbb{R}^d$ continuous. 
{The bound~\eqref{bounds:lipschitz} implies that $\frac{\partial H}{\partial p}$ is in addition uniformly bounded with}
\begin{equation}\label{h-deriv-linf-bound}
\begin{aligned}
\left|\frac{\partial H}{\partial p}(x,p)\right|\leq L_H &&& \forall (x,p)\in \Omega\times\mathbb{R}^d.
\end{aligned}
\end{equation}
We also suppose that
\begin{equation}\label{eq:Hp_bound}
\begin{aligned}
\left|\frac{\partial H}{\partial p}(x,p) - \frac{\partial H}{\partial p}(x,q)\right|\leq L_{H_p}|p-q| &&& \forall (x,p,q)\in\Omega\times\mathbb{R}^d\times\mathbb{R}^d,
\end{aligned}
\end{equation}
for some constant $L_{H_p}\geq 0$, so $\frac{\partial H}{\partial p}$ is Lipschitz continuous w.r.t.\ $p$, uniformly in $x\in\Omega$.
We shall also require some regularity of $\frac{\partial H}{\partial p}$ w.r.t its $x$-dependence, in particular, we suppose that $x\mapsto \frac{\partial H}{\partial p}(x,p)$ is in $H^1(\Omega;\R^\dim)$ for each $p\in \R^\dim$.
It is clear that the mappings $H^1(\Omega)\ni v\mapsto H(\cdot,\nabla v)\in L^2(\Omega)$ and $H^1(\Omega)\ni v\mapsto \frac{\partial H}{\partial p}(\cdot,\nabla v)\in L^{2}(\Omega;\mathbb{R}^d)$ are Lipschitz continuous, and also that $\frac{\partial H}{\partial p}(\cdot,\nabla v)\in L^\infty(\Omega;\R^\dim)$ for all $v\in H^1(\Omega)$.
We will often abbreviate these compositions by writing instead $H[\nabla v]\coloneqq H(\cdot,\nabla v)$ and $\frac{\partial H}{\partial p}[\nabla v]\coloneqq\frac{\partial H}{\partial p}(\cdot,\nabla v)$ a.e.\ in $\Omega$, for each $v\in H^1(\Omega)$.

In the analysis we will employ the following bound on the linearizations of the Hamiltonians when composed with gradients of functions in Sobolev spaces, c.f.~\cite[Lemma 2.1]{osborne2024near}.
\begin{lemma}\label{lemma-semismooth-tech-result}
For any $\epsilon>0$, there exists a $R>0$, depending only on $\epsilon$, $\Omega$, $L_H$ and $L_{H_p}$, such that 
\begin{equation}\label{semismooth-bound-general}
\left\|H[\nabla v] - H[\nabla w] - \frac{\partial H}{\partial p}[\nabla w]\cdot \nabla (v-w)\right\|_{H^{-1}(\Omega)} \leq \epsilon\|v - w\|_{H^1(\Omega)},
\end{equation}
whenever $v,w\in H^1(\Omega)$ satisfy $\|v- w\|_{H^1(\Omega)}\leq R$.
\end{lemma}
In fact, see also~\cite[Theorem~13]{SmearsSuli2014} which gives a major generalization of Lemma~\ref{lemma-semismooth-tech-result}, in particular by showing the semismoothness in function spaces of HJB operators in the case of Lipschitz but possibly nondifferentiable Hamiltonians.

\section{Continuous problem and equivalence between error and residual}\label{sec:continous-equivalence}
Our first contribution toward deriving computable a posteriori error bounds is to prove the equivalence between norms of the error and of the residuals, detailed in Theorem \ref{thm:error-res-equiv} below.
This equivalence is nontrivial due to the lack of coercivity of the underlying PDE, the nonlinear coupling, and possible nondifferentiability of the nonlinear coupling $F[m]$.
To establish this equivalence, we first define the weak form of \eqref{sys}, which is to find $(u,m)\in H_0^1(\Omega)\times H_0^1(\Omega)$ which satisfies 
\begin{subequations}\label{weakform-C1H}
\begin{gather}
\int_{\Omega}\nu\nabla u{\cdot}\nabla \psi+H[\nabla u]\psi \, \mathrm{d}x=\langle F[m],\psi\rangle_{H^{-1}\times H_0^1}, \label{weakform1-C1H}\\ 
\int_{\Omega}\nu\nabla m{\cdot}\nabla \phi+m\frac{\partial H}{\partial p}[\nabla u]{\cdot}\nabla \phi \,\mathrm{d}x=\langle  G,\phi\rangle_{H^{-1}\times H_0^1},\label{weakform2-C1H} 
\end{gather} 
\end{subequations}
for all $\psi,\phi\in H_0^1(\Omega)$. 

\begin{remark}[Existence and uniqueness of the solution]\label{rem-continuous-pb-basic-properties-C1H}
Under the above assumptions on the data, the existence and uniqueness of a solution $(u,m)$ of~\eqref{weakform-C1H} is known from~\cite[Theorems~3.3 \& 3.4]{osborne2024analysis}, see also~\cite{osborne2024erratum}. Note however that if some of these assumptions are removed, then uniqueness of the solution of~\eqref{weakform-C1H} does not generally hold, see e.g.~\cite[Example~3.1]{osborne2025regularization} for an example with multiple solutions where $F$ is allowed to be nonmonotone.
Moreover, under the above assumptions, we have the \emph{a priori} bounds $\|m\|_{H^1(\Omega)} \lesssim \|G\|_{H^{-1}(\Omega)}$ and $\|u\|_{H^1(\Omega)} \lesssim 1+\|G\|_{H^{-1}(\Omega)}+\|f\|_{C({\overline{\Omega}\times\mathcal{A}})}$. The density function $m$ is in addition nonnegative a.e.\ in $\Omega$.
Furthermore, recalling the hypothesis that $G=g_0-\nabla {\cdot}g_1$ with $g_0\in L^{q/2}(\Omega)$ and $g_1\in L^q(\Omega;\mathbb{R}^d)$ for some $q>d$.
It then follows that $m\in L^\infty(\Omega)$ with
\begin{equation}\label{m-ess-bound}
\|m\|_{L^\infty(\Omega)}\leq {\Mfty} \coloneqq C\left(\|G\|_{H^{-1}(\Omega)}+\|g_0\|_{L^{q/2}(\Omega)}+\|g_1\|_{L^q(\Omega;\mathbb{R}^d)}\right),
\end{equation} 
for some constant $C$ depending only on $d,\Omega, L_H,\nu$, $q$ and $L_F$, see~\cite[Theorem 8.15]{gilbarg2015elliptic}.
We will use~\eqref{m-ess-bound} in the analysis that follows.
\end{remark}

\subsection{Local equivalence between error and residual}
Our main result in this section is the local equivalence between the norms of the error, when comparing the exact solution with some general $H^1_0$-conforming approximation, and the dual-norm of the residual. 
In particular, we define the following residual operators: let $\RH,\RF:{H_0^1(\Omega)}\times H_0^1(\Omega)\to H^{-1}(\Omega)$ be given by
\begin{subequations}\label{FEM-Res}
\begin{align}
\langle \RH(\overline{u},\overline{m}),\psi\rangle_{H^{-1}\times H^1_0}&\coloneqq \langle F[\overline{m}],\psi\rangle_{H^{-1}\times H_0^1} - \int_{\Omega}\nu\nabla \overline{u}{\cdot}\nabla \psi+ H[\nabla \overline{u}]\psi \, \mathrm{d}x ,\label{FEM-Res-HJB}
\\
\langle \RF(\overline{u},\overline{m}),\phi\rangle_{H^{-1}\times H^1_0}&\coloneqq \langle G,\phi\rangle_{H^{-1}\times H_0^1} - \int_{\Omega}\nu\nabla \overline{m}{\cdot}\nabla \phi+ \overline{m}\frac{\partial H}{\partial p}[\nabla \overline{u}]{\cdot}\nabla \phi
\, \mathrm{d}x ,\label{FEM-Res-KFP}
\end{align}
\end{subequations}
for all $\overline{u},\overline{m},\psi,\phi\in H_0^1(\Omega)$.
The total residual is then defined by
\begin{equation}\label{eq:FEM-Res_total}
\Res(\overline{u},\overline{m}) \coloneqq  \left(\norm{\RH(\overline{u},\overline{m})}_{H^{-1}(\Omega)}^2+\norm{\RF(\overline{u},\overline{m})}_{H^{-1}(\Omega)}^2\right)^{\frac{1}{2}}.
\end{equation}
The following result on the local equivalence between errors and residuals is the cornerstone of the \emph{a posteriori} error analysis.
\begin{theorem}[Local equivalence between error and residual]\label{thm:error-res-equiv}
There exists a constant $\mathcal{R}_*>0$ such that, for any pair $(\overline{u},\overline{m})\in H^1_0(\Omega)\times H^1_0(\Omega)$ that satisfies 
\begin{equation}\label{eq:error_res_equiv_condition}
\overline{m}\geq 0 \text{ a.e.\ in }\Omega, \quad 
\Res(\overline{u},\overline{m})\leq \mathcal{R}_*,
\end{equation}
then we have the equivalence between errors and residuals
\begin{equation}\label{eq:error_res_equivalence}
\|u-\overline{u}\|_{H^1(\Omega)}+\|m-\overline{m}\|_{H^1(\Omega)}\eqsim\Res(\overline{u},\overline{m}).
\end{equation}
\end{theorem}

We prove Theorem~\ref{thm:error-res-equiv} in the following subsections.

\begin{remark}
It is well-known from {\ifJ\cite{lasry2007mean}\else the seminal work of Lasry \& Lions~\cite{lasry2007mean}\fi} that the nonnegativity of the density plays an important role in the proof of uniqueness of the solution of the MFG system in the setting of strictly monotone couplings. It is therefore not surprising that Theorem~\ref{thm:error-res-equiv} also requires nonnegativity of the density approximations in the first condition in~\eqref{eq:error_res_equiv_condition}.
\end{remark}

\subsection{Stability of HJB and FP equations}
In this section we establish stability results associated with the continuous HJB and KFP equations of the continuous problem \eqref{weakform-C1H}.
The following result expresses the continuous dependence of the residuals on the errors.
\begin{lemma}\label{lemma-E-general-est}
For all $\overline{u},\overline{m}\in {H_0^1(\Omega)}$, we have
\begin{subequations}\label{eq:dual-norm-residuals}
\begin{align} 
    \|\RH(\overline{u},\overline{m})\|_{H^{-1}(\Omega)}&\lesssim \|m-\overline{m}\|_{\mathcal{H}}+\|u-\overline{u}\|_{H^1(\Omega)},\label{eq:dual_norm_res_R1}
    \\ \|\RF(\overline{u},\overline{m})\|_{H^{-1}(\Omega)}&\lesssim \|m-\overline{m}\|_{H^1(\Omega)}+{\|u-\overline{u}\|_{H^1(\Omega)}},\label{eq:dual_norm_res_R2}  
\end{align}
\end{subequations}
where the hidden constants depend only on $\Omega$, $L_F$, $L_{H_p}$, $L_H$, $\nu$, {and $\Mfty$.} 
\end{lemma}

\begin{proof}
Since $(u,m)$ solves~\eqref{weakform-C1H}, we have for all $\psi\in {H_0^1(\Omega)}$
\begin{multline}\label{res-est-1}
\langle \RH(\overline{u},\overline{m}),\psi\rangle_{H^{-1}\times H_0^1}=\langle F[\overline{m}]-F[m],\psi\rangle_{H^{-1}\times H_0^1}
\\+\int_{\Omega}\nu\nabla (u-\overline{u}){\cdot}\nabla \psi + (H[\nabla u]-H[\nabla \overline{u}])\psi\,\mathrm{d}x.
\end{multline}
The bound~\eqref{eq:dual_norm_res_R1} then follows from the Lipschitz continuity of $F$ in~\eqref{F2} and the Lipschitz continuity of $H$ in~\eqref{bounds:lipschitz}.
Next, we have
\begin{multline}\label{res-est-2}
\langle \RF(\overline{u},\overline{m}),\phi\rangle_{H^{-1}\times {H_0^1}}=\int_{\Omega}\nu\nabla (m-\overline{m}){\cdot}\nabla \phi+(m-\overline{m})\frac{\partial H}{\partial p}[\nabla \overline{u}]{\cdot}\nabla \phi\,\mathrm{d}x
\\+\int_{\Omega}m\left(\frac{\partial H}{\partial p}[\nabla u] - \frac{\partial H}{\partial p}[\nabla \overline{u}]\right){\cdot}\nabla \phi\,\mathrm{d}x ,
\end{multline}
for all $\phi\in {H_0^1(\Omega)}$.
Then, using the bound $\norm{\frac{\partial H}{\partial p}[\nabla \overline{u}]}_{L^\infty(\Omega,\R^\dim)}\leq L_H$ which follows from~\eqref{h-deriv-linf-bound}, we see that
\begin{equation}\label{eq:R2_bound_2}
\sup_{\substack{\phi\in {H_0^1(\Omega)}\\ \|\nabla\phi\|_\Omega =1}}
\left[\int_\Omega (m-\overline{m})\frac{\partial H}{\partial p}[\nabla \overline{u}]{\cdot}\nabla \phi\,\mathrm{d}x \right] \leq L_H \norm{m-\overline{m}}_\Omega .
\end{equation}
The Lipschitz continuity of $\frac{\partial H}{\partial p}$ in~\eqref{eq:Hp_bound} and the $L^\infty$-bound on $m$ in~\eqref{m-ess-bound} imply that 
\begin{equation}\label{eq:R2_bound_3}
\sup_{\substack{\phi\in {H_0^1(\Omega)}\\ \|\nabla\phi\|_\Omega =1}} \left[ \int_{\Omega}m\left(\frac{\partial H}{\partial p}[\nabla u] - \frac{\partial H}{\partial p}[\nabla \overline{u}]\right)\cdot\nabla \phi\,\mathrm{d}x 
\right]
\leq L_{H_p} \Mfty\norm{u-\overline{u}}_{H^1(\Omega)}.
\end{equation}
Combining~\eqref{res-est-2} with~\eqref{eq:R2_bound_2} and~\eqref{eq:R2_bound_3}, we then obtain~\eqref{eq:dual_norm_res_R2}.
\end{proof}
{We now turn to the (local) stability of the continuous HJB equation~\eqref{weakform1-C1H}.}
\begin{lemma}[Stability of continous HJB equation]\label{lemma-w-H1-bound}
There exists a constant $R>0$ such that 
\begin{equation}\label{val-func-approx-err}
\|u-\overline{u}\|_{H^1(\Omega)}\lesssim \|\RH(\overline{u},\overline{m})\|_{H^{-1}(\Omega)}+\|m-\overline{m}\|_{\mathcal{H}},
\end{equation}
for any $\overline{m}\in {H_0^1(\Omega)}$, and for any $\overline{u}\in {H_0^1(\Omega)}$ satisfying $\|\overline{u}-u\|_{H^1(\Omega)}\leq R$. 
The hidden constant in~\eqref{val-func-approx-err} depends only on $\Omega$, $\dim$, $\nu$, $L_H$ and $L_F$.
\end{lemma}
\begin{proof}
Since $\norm{\frac{\partial H}{\partial p}[\nabla \overline{u}]}_{L^\infty(\Omega;\R^\dim)}\leq L_H$, it follows from~\cite[Lemma~4.5]{osborne2024analysis} that
\begin{equation}
\norm{u-\overline{u}}_{H^1(\Omega)}
\lesssim \sup_{\substack{\psi\in {H_0^1(\Omega)}\\ \|\nabla\psi\|_\Omega=1}}\left[\int_{\Omega}\nu\nabla (u-\overline{u}){\cdot}\nabla \psi + \frac{\partial H}{\partial p}[\nabla \overline{u}]{\cdot}\nabla (u-\overline{u}) \psi\,\mathrm{d}x\right],
\end{equation}
where the hidden constant depends only on $\Omega$, $\dim$, $\nu$ and $L_H$.
Using the continuous HJB equation of \eqref{weakform1-C1H}, we find that 
\begin{multline}
\norm{u-\overline{u}}_{H^1(\Omega)} \lesssim 
\sup_{\substack{\psi\in {H_0^1(\Omega)}\\ \|\nabla\psi\|_\Omega=1}}\left[
\langle \RH(\overline{u},\overline{m}),\psi\rangle_{H^{-1}\times {H_0^1}}
+ \langle F[m]-F[\overline{m}],\psi\rangle_{H^{-1}\times H_0^1} 
\right. 
\\\left. -\int_{\Omega}\left(H[\nabla u]-H[\nabla \overline{u}]-\frac{\partial H}{\partial p}[\nabla \overline{u}]{\cdot}\nabla (u-\overline{u})\right)\psi\,\mathrm{d}x  \right].
\end{multline}
Hence, the Lipschitz continuity of $F$, c.f.\ \eqref{F2}, and the triangle inequality imply that    
\begin{multline}\label{eq:val-func-approx-err_1}
\norm{u-\overline{u}}_{H^1(\Omega)} \lesssim \|\RH(\overline{u},\overline{m})\|_{H^{-1}(\Omega)} + L_F\norm{m-\overline{m}}_{\mathcal{H}}
\\ +  \left\lVert H[\nabla u] - H[\nabla \overline{u}] - \frac{\partial H}{\partial p}[\nabla \overline{u}]{\cdot}\nabla (u-\overline{u})\right\rVert_{H^{-1}(\Omega)}.
\end{multline}
Next, we apply Lemma~\ref{lemma-semismooth-tech-result} with $\epsilon$ chosen sufficiently small in order to absorb the last term on the right-hand side of~\eqref{eq:val-func-approx-err_1} into the left-hand side, and thereby deduce that there exists $R>0$ such that~\eqref{val-func-approx-err} holds whenever $\norm{u-\overline{u}}_{H^1(\Omega)}\leq R$.
\end{proof}

We now consider the stability of the continuous KFP equation~\eqref{weakform2-C1H}.
\begin{lemma}[Stability of continuous KFP equation]\label{lemma-zk-H1-bound}
For all $\overline{u},\overline{m}\in {H_0^1(\Omega)}$, we have
\begin{equation}\label{eq:mk_H1_bound}
\norm{m-\overline{m}}_{H^1(\Omega)}\lesssim  \norm{\RF(\overline{u},\overline{m})}_{H^{-1}(\Omega)} +\|u-\overline{u}\|_{H^1(\Omega)},
\end{equation}
where the constant depends only on $\Omega$, $\dim$, $\nu$, $L_H$, $L_{H_p}$ and on $\Mfty$.
\end{lemma}
\begin{proof}
Since $\norm{\frac{\partial H}{\partial p}[\nabla \overline{u}]}_{L^\infty(\Omega;\R^\dim)}\leq L_H$, it follows from~\cite[Lemma~4.5]{osborne2024analysis} that
\begin{multline}
    \norm{m - \overline{m}}_{H^1(\Omega)} \lesssim \sup_{\substack{\phi\in {H_0^1(\Omega)}\\ \|\nabla\phi\|_\Omega=1}}\left[  \int_{\Omega} \nu\nabla(m - \overline{m}){\cdot}\nabla \phi + (m - \overline{m}){\frac{\partial H}{\partial p}[\nabla \overline{u}}]{\cdot}\nabla \phi \,\mathrm{d}x\right]
    \\ = \sup_{\substack{\phi\in {H_0^1(\Omega)}\\ \|\nabla\phi\|_\Omega    =1}}
    \left[\langle  \RF(\overline{u},\overline{m}),\phi\rangle_{H^{-1}\times H^1_0}
    - \int_\Omega m \left(\frac{\partial H}{\partial p}[\nabla u] - \frac{\partial H}{\partial p}[\nabla \overline{u}] \right)\cdot\nabla \phi \,\mathrm{d}x\right],
\end{multline}
where the constant depends only on $\Omega$, $\dim$, $\nu$ and $L_H$.
Note that we have used~\eqref{res-est-2} in passing to the second line above.
Next, we see from~\eqref{eq:R2_bound_3} that
\begin{equation}
\norm{m - \overline{m}}_{H^1(\Omega)} \lesssim \norm{\RF(\overline{u},\overline{m})}_{H^{-1}(\Omega)} +  \Mfty L_{H_p}\norm{u-\overline{u}}_{H^1(\Omega)},
\end{equation}
thus showing~\eqref{eq:mk_H1_bound}.
\end{proof}

The next result concerns the $\mathcal{H}$-norm stability in the density component for the continuous MFG system~\eqref{weakform-C1H}, when restricted to the set of nonnegative functions in~$H_0^1(\Omega)$.
This bound can be seen as a quantitative analogue to some of the inequalities used in the well-known proofs of uniqueness of solutions due to~\cite{lasry2007mean}.

\begin{lemma}\label{lemma-key-L2-bound}
For any $\overline{u}, \overline{m}\in {H_0^1(\Omega)}$ such that $\overline{m}\geq 0$ a.e.\ in $\Omega$, we have
\begin{equation}\label{key-L2-bound}
c_F\|m-\overline{m}\|_{\mathcal{H}}^2\leq \langle \RH(\overline{u},\overline{m}),\overline{m}-m\rangle_{H^{-1}\times {H_0^1}} - \langle \RF(\overline{u},\overline{m}),\overline{u}-u\rangle_{H^{-1}\times H_0^1}.
\end{equation}
\end{lemma}

\begin{proof}
Let $\overline{u}, \overline{m}\in {H_0^1(\Omega)}$ with $\overline{m}\geq 0$ a.e.\ in $\Omega$ be fixed but arbitrary.
Using the continuous HJB equation~\eqref{weakform1-C1H} and the continuous KFP equation~\eqref{weakform2-C1H}, we find after some elementary computation that
\begin{multline}\label{eq:L2mono_bound_1}
\langle \RH(\overline{u},\overline{m}),\overline{m}-m\rangle_{H^{-1}\times {H_0^1}} - \langle \RF(\overline{u},\overline{m}),\overline{u}-u\rangle_{H^{-1}\times H_0^1}
\\= \langle F[\overline{m}] - F[m],\overline{m}-m\rangle_{H^{-1}\times H_0^1} + (m, R_H[\nabla\overline{u},\nabla u])_\Omega + (\overline{m},R_H[\nabla u,\nabla \overline{u}])_\Omega,
\end{multline}
where $R_H[\nabla v,\nabla w]\coloneqq H[\nabla v]-H[\nabla w]-\frac{\partial H}{\partial p}[\nabla w]{\cdot} \nabla(v-w)$ for any $v,\,w \in H^1(\Omega)$,
and where we recall that $(\cdot,\cdot)_{\Omega}$ denotes the inner product on $L^2(\Omega)$.
Since $H$ is convex and differentiable w.r.t.\ $p$, we have $R_H[\nabla v,\nabla w]\geq 0$ a.e.\ in $\Omega$, for any $v,\,w$. Also, we have $m\geq 0$ in $\Omega$ by the comparison principle. Therefore, the terms $(m, R_H[\nabla\overline{u},\nabla u])_\Omega$ and $(\overline{m},R_H[\nabla u,\nabla \overline{u}])_\Omega$ in~\eqref{eq:L2mono_bound_1} are both nonnegative, which implies that
\begin{equation}\label{eq:L2mono_bound_2}
 \begin{aligned}
\langle F[\overline{m}] - F[m],\overline{m} - m\rangle_{H^{-1}\times H_0^1} \leq  \langle \RH(\overline{u},\overline{m}),\overline{m}-m\rangle_{H^{-1}\times {H_0^1}} - \langle \RF(\overline{u},\overline{m}),\overline{u}-u\rangle_{H^{-1}\times H_0^1}.
\end{aligned}
\end{equation}
The desired result~\eqref{key-L2-bound} is then obtained from~\eqref{eq:L2mono_bound_2} and from the strong monotonicity of $F$ in~\eqref{strong-mono-F-C1H}.
\end{proof}

\subsection{Proof of Theorem~\ref{thm:error-res-equiv}}
In this section we prove the following result which shows the continuous dependence of the error on the residual of the problem.

\begin{lemma}\label{lem:residual_stability}
For each $\delta>0$, there exists a $\mathcal{R}_*>0$ such that
\begin{equation}
\mathcal{R}(\overline{u},\overline{m})\leq \mathcal{R}_* \implies \norm{u-\overline{u}}_{H^1(\Omega)}+\norm{m-\overline{m}}_{H^1(\Omega)}<\delta,
\end{equation}
for all $(\overline{u},\overline{m})\in H^1_0(\Omega)\times H^1_0(\Omega)$.
\end{lemma}
\begin{proof}
The main idea of the proof is to show that if the claim were false, then this would contradict the uniqueness of the solution of~\eqref{sys}.
Under the hypothesis that the claim is false, there exists a $\delta>0$ and a sequence of pairs $(u_k,m_k)\in H^1_0(\Omega)\times H^1_0(\Omega)$, $k\in \N$, with $\mathcal{R}(u_k,m_k)\tends 0$ as $k\tends \infty$, but $\norm{u-u_k}_{H^1(\Omega)}+\norm{m-m_k}_{H^1(\Omega)}\geq \delta$.

First, we claim that the sequence is uniformly bounded in the $H^1$-norm, i.e.\ that
\begin{equation}\label{eq:weaklimits_uniform}
\sup_{k\in\N }\left[\norm{m_k}_{H^1(\Omega)}+\norm{u_k}_{H^1(\Omega)}\right] <\infty.
\end{equation}
Indeed, since $\norm{\frac{\partial H}{\partial p}[\nabla u_k]}_{L^\infty(\Omega;\R^\dim)}\leq L_H$ for all $k\in\N$, we first apply \cite[Lemma~4.5]{osborne2024analysis} and the hypothesis that~$\lim_{k\tends\infty}\norm{\RF(u_k,m_k)}_{H^{-1}(\Omega)}=0$ to obtain
\begin{equation}\label{eq:weaklimits_uniform_1}
\begin{split}
\limsup_{k\tends\infty}\norm{m_k}_{H^1(\Omega)}
&\lesssim \limsup_{k\tends\infty}\sup_{\substack{\phi\in H^1_0(\Omega) \\ \norm{\nabla \phi}_\Omega=1 }} \left[\int_\Omega \nu\nabla m_k{\cdot}\nabla \phi + m_k \frac{\partial H}{\partial p}[\nabla u_k]{\cdot}\nabla \phi \,\mathrm{d}x \right]
\\ 
&= \limsup_{k\tends\infty}\norm{G-\RF(u_k,m_k)}_{H^{-1}(\Omega)}
 \leq \norm{G}_{H^{-1}(\Omega)}.
\end{split}
\end{equation}
Next, recalling~\eqref{Hamiltonian} and the fact that $\mathcal{A}$ is compact, with the functions $b$ and $f$ continuous on $\overline{\Omega}\times\mathcal{A}$, we use the measurable selection theorem~\cite[Theorem~10]{SmearsSuli2014} to find that, for each $k\in \N$, there exists a Lebesgue measurable map $\alpha_k:\Omega\tends \mathcal{A}$ such that
\[
\alpha_k(x) \in \argmax_{\alpha\in\mathcal{A}}\left[b(x,\alpha){\cdot}\nabla u_k(x)-f(x,\alpha)\right] \quad \text{for a.e. }x\in\Omega.
\]
Defining $b_k\coloneqq b(\cdot,\alpha_k(\cdot))$ and $f_k\coloneqq f(\cdot,\alpha_k(\cdot))$, we then have $H[\nabla u_k]=b_k\cdot \nabla u_k - f_k$ in $\Omega$. This implies that each $u_k\in H^1_0(\Omega)$ satisfies
\begin{equation}\label{eq:weaklimits_uniform_2}
\int_\Omega \nu\nabla u_k{\cdot}\nabla \psi + b_k{\cdot}\nabla u_k \psi \,\mathrm{d}x =  \langle F[m_k],\psi \rangle_{H^{-1}\times H^1_0} - \langle \RH(u_k,m_k),\psi\rangle
+ \int_\Omega f_k \psi \,\mathrm{d}x,
\end{equation}
for all $\psi \in H^1_0(\Omega)$.
Therefore, \cite[Lemma~4.5]{osborne2024analysis} applied to~\eqref{eq:weaklimits_uniform_2} shows that
\begin{equation}
\limsup_{k\tends\infty}\norm{u_k}_{H^1(\Omega)}\lesssim 1+ \norm{f}_{C(\overline{\Omega}\times\mathcal{A})}+\norm{G}_{H^{-1}(\Omega)},
\end{equation}
where we have used here the hypothesis that $\norm{\RH(u_k,m_k)}_{H^{-1}(\Omega)}\tends 0$ as $k\tends \infty$, and also the Lipschitz continuity of $F$, c.f.\ \eqref{F2}, the embedding of $H^1_0(\Omega)$ into $\mathcal{H}$, and the above bound~\eqref{eq:weaklimits_uniform_1} to bound $\limsup_{k\tends\infty}\norm{F[m_k]}_{H^{-1}(\Omega)}\lesssim 1+\norm{G}_{H^{-1}(\Omega)}$.
This completes the proof of~\eqref{eq:weaklimits_uniform}.

The uniform boundedness of the sequence~\eqref{eq:weaklimits_uniform} and the Rellich--Kondrachov theorem then imply that there exists a $(\widetilde{u},\widetilde{m})\in H^1_0(\Omega)\times H^1_0(\Omega)$ and a subsequence, to which we pass without change of notation, such that $m_k\rightharpoonup \widetilde{m}$ in $H^1_0(\Omega)$ and $u_k\rightharpoonup \widetilde{u}$ in $H^1_0(\Omega)$ as $k\tends \infty$, where the convergence of the sequences is additionally strong in $L^2(\Omega)$.
Furthermore, after possibly passing to a further subsequence, we can assume $m_k\tends \widetilde{m}$ in $\mathcal{H}$ without loss of generality, since $H^1_0(\Omega)$ is compactly embedded in $\mathcal{H}$ by hypothesis.
This implies that $F[m_k]\tends F[\widetilde{m}]$ in $H^{-1}(\Omega)$ by continuity of $F$ from $\mathcal{H}$ to $H^{-1}(\Omega)$.
After possibly passing to a further subsequence, we can also assume without loss of generality that $h_k\coloneqq H[\nabla u_k] \rightharpoonup \tilde{h} $ in $L^2(\Omega)$ as $k\tends\infty$, for some $\tilde{h}\in L^2(\Omega)$.
Then, it follows from $\mathcal{R}(u_k,m_k)\tends 0$ as $k\tends \infty$ that
 \begin{align}\label{eq:weaklimit_1}
 0=\lim_{k\tends \infty}\langle \RH(u_k,m_k),\psi\rangle_{H^{-1}\times H^1_0} = \langle F[\widetilde{m}],\psi\rangle_{H^{-1}\times H^1_0} - \int_\Omega \nu\nabla \widetilde{u}{\cdot} \nabla \psi + \tilde{h} \psi \,\mathrm{d}x,
 \end{align}
 for all $\psi \in H^1_0(\Omega)$.
Using weak-times-strong convergence and $\RH(u_k,m_k)\tends 0$ as $k\tends \infty$, we see that
\begin{equation}\label{eq:weaklimit_2}
\begin{aligned}
\lim_{k\tends\infty}\nu \norm{\nabla u_k}_
{\Omega}^2 &= \lim_{k\tends \infty}
\left[ \langle F[m_k],u_k\rangle_{H^{-1}\times H^1_0}-\int_\Omega h_k u_k \,\mathrm{d}x - \langle \RH(u_k,m_k),u_k\rangle_{H^{-1}\times H^1_0} \right] 
\\ &= \langle F[\widetilde{m}],\widetilde{u}\rangle_{H^{-1}\times H^1_0}-\int_\Omega\tilde{h}\widetilde{u} \,\mathrm{d}x = \nu \norm{\nabla \widetilde{u}}_\Omega^2,
\end{aligned}
\end{equation}
where the last equality follows from~\eqref{eq:weaklimit_1} when testing with $\widetilde{u}$.
This implies the strong convergence $u_k\tends u_k$ in the $H^1$-norm as $k\tends \infty$.
In turn, this implies that $\tilde{h} = H[\nabla \widetilde{u}]$ in $L^2(\Omega)$ by Lipschitz continuity of $H$. Combined with~\eqref{eq:weaklimit_1}, this implies that $\RH(\widetilde{u},\widetilde{m})=0$.
 Moreover, we find that $\frac{\partial H}{\partial p}[\nabla u_k]\tends \frac{\partial H}{\partial p}[\nabla \widetilde{u}]$ in $L^2(\Omega;\R^{\dim})$ as $k\tends \infty$ by Lipschitz continuity of $\frac{\partial H}{\partial p}$. 
Therefore, we obtain
 \begin{align}\label{eq:weaklimit_3}
0 = \lim_{k\tends \infty}\langle \RF(u_k,m_k),\phi \rangle_{H^{-1}\times H^1_0}
= \langle G,\phi\rangle - \int_\Omega \nu \nabla \widetilde{m}{\cdot}\nabla \phi + \widetilde{m}\frac{\partial H}{\partial p}[\nabla \widetilde{u}]{\cdot}\nabla \phi \,\mathrm{d}x,
\end{align}
for all $\phi \in H^1_0(\Omega)$. 
Hence $\RF(\widetilde{u},\widetilde{m})=0$, and thus $\Res(\widetilde{u},\widetilde{m})=0$, i.e.\ the pair $(\widetilde{u},\widetilde{m})$ is a weak solution of~\eqref{sys}. 
Using the uniform $L^\infty$-norm bound $\norm{\frac{\partial H}{\partial p}[\nabla u_k]}_{L^\infty(\Omega;\R^d)}\leq L_H$, we also deduce the weak convergence $\frac{\partial H}{\partial p}[\nabla u_k] {\cdot} \nabla m_k \rightharpoonup \frac{\partial H}{\partial p}[\nabla \widetilde{u}]{\cdot} \nabla \widetilde{m} $ in $L^2(\Omega)$, and thus weak-times-strong convergence gives
\begin{equation}
\begin{aligned}
  \lim_{k\tends \infty}\nu \norm{\nabla m_k}_{\Omega}^2 &= \lim_{k\tends \infty}\left[ \langle G-\RF(u_k,m_k),m_k\rangle_{H^{-1}\times H^1_0} - \int_\Omega m_k \frac{\partial H}{\partial p}[\nabla u_k]{\cdot} \nabla m_k \,\mathrm{d}x  \right] 
  \\ & = \langle G,\widetilde{m}\rangle_{H^{-1}\times H^1_0} - \int_\Omega \widetilde{m}\frac{\partial H}{\partial p}[\nabla \widetilde{u}]{\cdot} \nabla \widetilde{m}\,\mathrm{d}x
= \nu \norm{\nabla \widetilde{m}}_\Omega^2,
\end{aligned}
\end{equation}
where the last identity follows from~\eqref{eq:weaklimit_3} when testing with $\widetilde{m}$.
This implies the strong convergence $m_k\tends m$ in the $H^1$-norm as $k\tends \infty$.
Thus, we conclude that $(\widetilde{u},\widetilde{m})$ is a weak solution of \eqref{sys}, yet
\[
\norm{m-\widetilde{m}}_{H^1(\Omega)}+\norm{u-\widetilde{u}}_{H^1(\Omega)}=\lim_{k\tends \infty}\left[\norm{m-m_k}_{H^1(\Omega)}+\norm{u-u_k}_{H^1(\Omega)}\right]\geq \delta >0.
\]
This contradicts the uniqueness of the solution of~\eqref{sys}, see Remark~\ref{rem-continuous-pb-basic-properties-C1H}.
\end{proof} 

\begin{proof}[Proof of Theorem \ref{thm:error-res-equiv}]
First, note that the continuity bounds of Lemma~\ref{lemma-E-general-est} imply the lower bound 
$\Res(\overline{u},\overline{m})\lesssim \norm{m-\overline{m}}_{H^1(\Omega)}+\norm{u-\overline{u}}_{H^1(\Omega)}$,  where we are using here the continuous embedding of $H^1_0(\Omega)$ into $\mathcal{H}$ so that $\norm{m-\overline{m}}_{\mathcal{H}}\lesssim \norm{m-\overline{m}}_{H^1(\Omega)}$.
Therefore, it remains only to show the upper bound, namely that
\begin{equation}\label{eq:main_thm_proof_4}
\norm{m-\overline{m}}_{H^1(\Omega)}+\norm{u-\overline{u}}_{H^1(\Omega)}
\lesssim \Res(\overline{u},\overline{m}).
\end{equation}
Let $R>0$ be the constant given in~Lemma~\ref{lemma-w-H1-bound}.
Then, Lemma~\ref{lem:residual_stability} implies that there exists a $\mathcal{R}_*>0$ such that $\norm{u-\overline{u}}_{H^1(\Omega)}\leq R$ for any $(\overline{u},\overline{m})\in H^1_0(\Omega)\times H^1_0(\Omega)$ that satisfy $\mathcal{R}(\overline{u},\overline{m})<\mathcal{R}_*$.
Suppose now that we are given $(\overline{u},\overline{m})\in H_0^1(\Omega)\times H^1_0(\Omega)$ with $\overline{m}\geq 0$ a.e.\ in $\Omega$ and $\mathcal{R}(\overline{u},\overline{m})<\mathcal{R}_*$.
We use Lemma~\ref{lemma-zk-H1-bound} to get
\begin{equation}\label{eq:main_thm_proof_1}
\norm{m-\overline{m}}_{H^1(\Omega)}+\norm{u-\overline{u}}_{H^1(\Omega)}
\lesssim \norm{\RF(\overline{u},\overline{m})}_{H^{-1}(\Omega)} +\|u-\overline{u}\|_{H^1(\Omega)}.
\end{equation}
Since $\norm{u-\overline{u}}_{H^1(\Omega)}\leq R$, we then apply Lemma~\ref{lemma-w-H1-bound} to the last term on the right-hand side of~\eqref{eq:main_thm_proof_1} above to find that
\begin{equation}\label{eq:main_thm_proof_2}
\begin{aligned}
\norm{m-\overline{m}}_{H^1(\Omega)}+\norm{u-\overline{u}}_{H^1(\Omega)}
&\lesssim \norm{\RF(\overline{u},\overline{m})}_{H^{-1}(\Omega)} + \|\RH(\overline{u},\overline{m})\|_{H^{-1}(\Omega)}+\|m-\overline{m}\|_{\mathcal{H}}
\\ & \lesssim \Res(\overline{u},\overline{m})+\|m-\overline{m}\|_{\mathcal{H}}.
\end{aligned}
\end{equation}
Then, since by hypothesis $\overline{m} \in H_0^1(\Omega)$ with $\overline{m}\geq 0$ a.e.\ in $\Omega$, we use the bound~\eqref{key-L2-bound} (after taking square-roots) from Lemma~\ref{lemma-key-L2-bound}, which implies that
\begin{multline}\label{eq:main_thm_proof_5}
\norm{m-\overline{m}}_{\mathcal{H}}\lesssim \left(\langle \RH(\overline{u},\overline{m}),\overline{m}-m\rangle_{H^{-1}\times H_0^1} - \langle \RF(\overline{u},\overline{m}),\overline{u}-u\rangle_{H^{-1}\times H_0^1}\right)^{\frac{1}{2}}
\\ \lesssim \Res(\overline{u},\overline{m})^{\frac{1}{2}}\left(\norm{m-\overline{m}}_{H^1(\Omega)}+\norm{u-\overline{u}}_{H^1(\Omega)}\right)^{\frac{1}{2}},
\end{multline}
Therefore, combining~\eqref{eq:main_thm_proof_2} with \eqref{eq:main_thm_proof_5} gives
\begin{multline}\label{eq:main_thm_proof_3}
\norm{m-\overline{m}}_{H^1(\Omega)}+\norm{u-\overline{u}}_{H^1(\Omega)}
\\ \lesssim \Res(\overline{u},\overline{m}) + \Res(\overline{u},\overline{m})^{\frac{1}{2}}\left(\norm{m-\overline{m}}_{H^1(\Omega)}+\norm{u-\overline{u}}_{H^1(\Omega)}\right)^\frac{1}{2}
\\ \lesssim (1+\epsilon^{-1})\Res(\overline{u},\overline{m}) + \epsilon \left(\norm{m-\overline{m}}_{H^1(\Omega)}+\norm{u-\overline{u}}_{H^1(\Omega)}\right),
\end{multline}
for any $\epsilon>0$, where the last inequality follows from Young's inequality with a parameter.
After choosing $\epsilon$ sufficiently small compared to the hidden constant in~\eqref{eq:main_thm_proof_3}, we can then absorb the last term on the right-hand side of~\eqref{eq:main_thm_proof_3} into the left-hand side, which finally gives~\eqref{eq:main_thm_proof_4} and thus completes the proof.
\end{proof}

\section{Stabilized finite element discretizations} \label{sec:stablized-fem}
\subsection{Setting and notation}
\paragraph{The mesh.}
Let $\mathcal{T}$ denote a conforming simplicial mesh of the domain $\Omega$. We adopt the convention that each element $K\in\T$ is closed.
For each element $K\in\T$, we let $h_K$ denote the diameter of $K$. We define the mesh-size function $h_{\T}\in L^\infty(\Omega)$ by $h_{\T}|_K\coloneqq h_K$ for each element $K\in\T$. In the following, the notation for inequalities $a \lesssim b$ will allow the hidden constant to depend on the shape-regularity parameter of $\T$, defined as $\theta_\T\coloneqq \max_{K\in\T}\frac{h_K}{\rho_K}$, where $\rho_K$ denotes the diameter of the largest inscribed balls in $K$. However, the hidden constants will be otherwise independent of the size of the mesh elements.
Two distinct elements of $\T$ are called neighbours if they have non-empty intersection. 
For each $K\in\T$, we let $\TK$ denote the set of neighbouring elements of $K$.

Let $\calV$ denote the set of all vertices of $\T$ and let $\calVi=\calV\cap \Omega$ denote the set of interior vertices of $\T$.
Two distinct vertices are called neighbours if they belong to a common element of $\T$. 
For each vertex $z\in \calV$, we let $\Tz$ denote the set of elements of $\T$ that contain $z$, and we let $\omz\coloneqq\bigcup \{ K^\prime \colon K^\prime \in \Tz\}$ denote its associated vertex patch. Let $h_{\omz}\coloneqq \diam \omz$ denote the diameter of $\omz$.

Let $\calF$ denote the set of all faces of $\T$, and let $\calFi$ denote the subset of all interior faces of $\calF$, i.e.\ all faces that are not contained in $\partial\Omega$.
For each face $F\in\calF$, let $h_F\coloneqq \diam F$ denote the diameter of $F$.
For each element $K\in\T$, we let $\calFiK$ denote the set of interior faces $F\in\calFi$ that are contained in $K$.
For each face $F\in\calF$, we let $\TF$ denote the set of elements that contain $F$, and we let $\omega_F \coloneqq \bigcup_{K\in\TF}K$ denote its associated patch.
For each element $K\in\T$, we define the set $\tTK$ of face-sharing neighbouring elements and the associated patch $\omK$
\begin{equation}
\begin{aligned}
\tTK \coloneqq \bigcup_{F\in\calFKi}\TF, &&& \omK\coloneqq \bigcup_{K^\prime\in\tTK} K^\prime = \bigcup_{F\in\calFKi}\omega_F.
\end{aligned}
\end{equation}
Note that $K\in\tTK$ for all $K\in\T$.
For each vertex $z\in\calV$, we let $\calFiz$ denote the set of interior faces that contain $z$.

Let $\calE$ denote the set of edges of the mesh $\T$, i.e.\ the set of all closed line segments formed by all pairs of neighbouring vertices. Note that for $\dim=2$, edges and faces of the mesh coincide. For each edge $E\in\calE$, we let the set of elements of $\T$ containing~$E$ be denoted by $\T_E\coloneqq \{K\in\T:E\subset K\}$. An edge $E\in\calE$ is called an internal edge if there exists at least one vertex $z_*\in \calVi$ such that $z_*\in E$. We denote the set of all internal edges by $\mathcal{E}^I$. For each $z\in \mathcal{V}$, we let $\mathcal{E}_{z}\coloneqq\{E\in\mathcal{E}:z\in E\}$ denote the set of edges containing the vertex $z\in \mathcal{V}$. Given a simplex $K\in\mathcal{T}$ we define the set
$\calEKI\coloneqq \{E\in \mathcal{E}^I\colon E\subset K\}$ which is the collection of edges of $K$ that are internal edges.

\paragraph{Finite element space.}
The numerical method is computed using a linear finite element space $V(\mathcal{T}) \subset H^1_0(\Omega)$ defined by
\begin{align}
    V(\mathcal{T}) \coloneqq \left\{ v_\mathcal{T} \in H^1_0(\Omega)\colon v_\mathcal{T} \in \mathcal{P}_1(K) \quad \forall K \in \mathcal{T} \right\},
\end{align}
where $\mathcal{P}_1(K)$ denotes the space of real-valued affine functions on $K$, i.e.\ polynomials of degree at most $1$.
Let $\{\psi_z\}_{z\in\calVi}$ denote the standard Lagrange nodal basis for $V(\T)$, which is characterized by the conditions $\psi_z(z)=1$ and $\psi_z(z^\prime)=0$ for all $z^\prime\in \calV\setminus\{z\}$, for all $z\in\calVi$.

\paragraph{Quasi-interpolation operator.}
Let $\IT:H^1_0(\Omega)\rightarrow V(\T)$ denote a quasi-interpolation operator that satisfies the following properties: $\IT v_\T = v_\T$ for all $v_\T\in V(\T)$, and
\begin{equation}\label{eq:quasi_approx_bound_elements}
\left(\sum_{K\in\T} h_K^{2(k-1)}\norm{\nabla^k(v-\IT v)}_K^2 \right)^\frac{1}{2} \lesssim \norm{v}_{H^1(\Omega)} \quad\forall v \in H^1_0(\Omega),\; k\in\{0,1\},
\end{equation}
where the hidden constant depends only on the dimension $d$ and the shape-regularity of~$\T$.
Note that~\eqref{eq:quasi_approx_bound_elements} implies that $\norm{v-\IT v}_{\Omega}\lesssim h_\T \norm{v}_{H^1(\Omega)}$ and $\norm{\IT v}_{H^1(\Omega)}\lesssim \norm{v}_{H^1(\Omega)}$ for all $v\in H^1_0(\Omega)$.
For instance, one may consider the Scott--Zhang quasi-interpolation operator~\cite{ScottZhang1990} where~\eqref{eq:quasi_approx_bound_elements} is given by~\cite[Theorem~4.1]{ScottZhang1990}.
Note that the trace inequality for element faces \cite[Lemma~B.1]{MonkSuli1999} then implies that
\begin{equation}\label{eq:quasi_approx_bound_faces}
\left( \sum_{F\in \mathcal{F}} h_F^{-1}\norm{v-\IT v}_F^2 \right)^{\frac{1}{2}} \lesssim \norm{v}_{H^1(\Omega)}\quad \forall v \in H^1_0(\Omega).
\end{equation}

\subsection{Family of numerical schemes}
In order to satisfy the nonnegativity of the density approximation, which is one of the hypotheses of Theorem~\ref{thm:error-res-equiv}, it is often necessary to include some stabilization in the construction of the numerical discretization. 
To allow for a wide range of possible choices of stabilizations, we develop the \emph{a posteriori} analysis for an abstract stabilized FEM discretization of the general form: find $(u_{\T},m_{\T})\in V(\mathcal{T})\times V(\mathcal{T})$ such that
\begin{subequations}\label{eq:num_scheme}
    \begin{align}
    &\int_{\Omega} \nu \nabla u_\mathcal{T} \cdot \nabla v_\mathcal{T} + H[\nabla u_{\T}] v_{\mathcal{T}} \mathrm{d}x + S_1(u_{\T},m_{\T};v_{\mathcal{T}}) = \langle F[m_{\T}],v_{\mathcal{T}}\rangle_{H^{-1}\times H^1_0},\label{eq:num_scheme_hjb}\\
    &\int_{\Omega} \nu \nabla m_{\T}\cdot \nabla w_{\mathcal{T}} + m_\mathcal{T} \frac{\partial H}{\partial p}[\nabla u_\mathcal{T}]\cdot \nabla w_{\mathcal{T}} \mathrm{d}x + S_2(u_{\T},m_{\T};w_{\mathcal{T}}) = \langle G, w_{\mathcal{T}} \rangle_{H^{-1}\times H^1_0},\label{eq:num_scheme_fp}
\end{align}
\end{subequations}
for all discrete test functions $v_\T,\;w_\T\in V(\T)$.
In~\eqref{eq:num_scheme} above, the terms $S_i(\cdot,\cdot;\cdot):[V(\mathcal{T})]^3\rightarrow \R$, $i\in\{1,2\}$, represent general stabilization terms, see below. 

\paragraph{Assumptions on the stabilizations.}
In order to develop the \emph{a posteriori} analysis for as wide a range of stabilizations as possible, we consider abstract stabilizations that satisfy some structural conditions. See however Section~\ref{subsec:examples} below for some concrete examples.
We will assume throughout this work that the forms $S_i(\cdot,\cdot;\cdot)$ are linear in the third argument, i.e.\ for each $i\in\{1,2\}$, we have $S_i(u_{\T},m_{\T};\lambda v_\mathcal{T}+w_\mathcal{T})= \lambda S_i(u_{\T},m_{\T};v_{\mathcal{T}})+S_i(u_{\T},m_{\T};w_{\mathcal{T}})$ for all $(u_{\T},m_{\T},v_{\mathcal{T}},w_{\mathcal{T}})\in \left[V(\mathcal{T})\right]^4$, for all $\lambda \in \R$. 
Note that this allows the stabilizations to be nonlinear with respect to $u_\T$ and $m_\T$.
The main assumption for the \emph{a posteriori} error analysis is that the stabilizations must be chosen to ensure the nonnegativity of the density approximation. 
In other words, we require that
\begin{enumerate}[label={(H\arabic*)},resume]
\item any solution $(u_{\T},m_{\T})$ of the numerical scheme~\eqref{eq:num_scheme} satisfies $m_{\T}\geq 0$ in $\Omega$.
\label{H:stabilization_main}
\end{enumerate}
Recall that nonnegativity of the density approximation is the natural discrete counterpart of the nonnegativity of the exact density $m$, see Remark~\ref{rem-continuous-pb-basic-properties-C1H}.
We present various examples of possible stabilizations in~Section~\ref{subsec:examples} below.
In Section~\ref{sec:general-estimates}, we show general \emph{a posteriori} error bounds assuming only that the stabilizations satisfy linearity in the third argument and the assumption~\ref{H:stabilization_main}. 

\begin{remark}
In practice, stabilizations are often designed to satisfy stronger conditions with the purpose of ensuring, amongst other things, existence and uniqueness of a numerical solution $(u_\T,m_\T)\in \left[V(\T)\right]^2$ and good \emph{a priori} convergence properties, c.f.~\cite{osborne2024analysis,osborne2024erratum,osborne2023finite,osborne2024near}.
However, in this work, we aim for rather minimal conditions on the stabilization to enable the \emph{a posteriori} error analysis.
The assumptions made above on the stabilization, including~\ref{H:stabilization_main} do not appear to be generally sufficient for proving the existence or uniqueness of a numerical solution. 
Therefore, in the following, we will assume the existence of a numerical solution $(u_\T,m_\T) \in \left[V(\T)\right]^2$. The uniqueness of the numerical solution is however not required for the \emph{a posteriori} analysis.
\end{remark}

\subsection{Examples}\label{subsec:examples}

In order to illustrate some particular choices of the stabilizations $S_i(\cdot,\cdot;\cdot)$ from~\eqref{eq:num_scheme} that we have in mind, in this section we summarize the stabilization methods that were considered in the earlier works \cite{osborne2024analysis,osborne2023finite,osborne2024near} on MFG.
In particular, these works considered linear stabilizations that can be written in the following general form
\begin{equation}\label{eq:linear_stabilization}
\begin{aligned}
S_1(u_\T,m_\T;v_\T)\coloneqq \int_\Omega D_\T\nabla u_\T\cdot \nabla v_\T\dx,
&&& S_2(u_\T,m_\T;v_\T)=\int_\Omega D_\T\nabla m_\T\cdot \nabla v_\T\dx,
\end{aligned}
\end{equation}
for all $(u_\T,m_\T,v_\T)\in [V(\T)]^3$, for some suitably chosen matrix $D_\T\in L^{\infty}\left(\Omega;\Rdd\right)$ with $D_\T(x)\in \Rddsymp$ for a.e.\ $x\in \Omega$, where $\Rddsymp$ denotes the set of symmetric positive semidefinite matrices in $\Rdd$.
Note that the stabilizations in~\eqref{eq:linear_stabilization} are linear in all arguments, and moreover the stabilization $S_1$ for the HJB equation is independent of the density approximation, and likewise the stabilization $S_2$ for the KFP equation is independent of the value function.
In~\cite{osborne2023finite,osborne2024near}, the stabilization matrix $D_\T$ was defined elementwise to be of the form
\begin{equation}\label{edge-tensor-formula}
D_\T|_K\coloneqq \sum_{E\in\calEKI}\edgeweight\, t_E t_E^T\quad \forall K\in\mathcal{T}, 
\end{equation}
where $t_E$ denotes a fixed choice of unit tangent vector to $E$ (the choice of orientation does not matter); and where $\edgeweight\geq 0$ is a suitably chosen weight. 
In particular, it is shown in~\cite[Theorem~4.2]{osborne2023finite} that if the mesh satisfies the Xu--Zikatanov condition{\ifJ, c.f.~\cite{xu1999monotone},\else~\cite{xu1999monotone}\fi} and if the weights are suitably chosen so that $\edgeweight \eqsim L_H \diam E$ for each edge $E$, where $L_H$ is from~\eqref{bounds:lipschitz}, then the resulting stabilization enforces a discrete maximum principle/comparison principle for the discretized problem, so that Assumption~\ref{H:stabilization_main} is satisfied, and moreover the method remains consistent to first-order.
The earlier work~\cite{osborne2024analysis} considered the simpler choice where $D_\T|_K\coloneqq \gamma_{\T,K} \mathbb{I}_\dim $, where $\mathbb{I}_\dim$ denotes the $\dim\times\dim$ identity matrix and where $\gamma_{\T,K}$ is a chosen constant parameter for each $K\in \T$. 
It is known if that if the mesh is strictly acute, and if the parameter $\gamma_{\T,K}$ is chosen to be sufficiently large, c.f.~\cite[Eq.~(5.3)]{osborne2024analysis}, then a discrete maximum principle/comparison principle is again also available so that Assumption~\ref{H:stabilization_main} also holds, see~\cite[Lemma~6.1]{osborne2024analysis}.
The stabilizations considered above are by no means the only ones that are admissible for the purposes of \emph{a posteriori} error analysis, see also for instance~\cite{barrenechea2017edge,barrenechea2018unified,barrenechea2024finite}.

\section{A posteriori error bounds for general stabilizations}\label{sec:general-estimates}

\subsection{Decomposition of the residuals}
The starting point for the derivation of \emph{a posteriori} error estimators for the approximations are some general abstract bounds on the residuals. 
We initially adopt a similar approach as in~\cite{TobiskaVerfurth2015}.
Recall that the residual operators $\Resi$ are defined in~\eqref{FEM-Res} above.
Supposing that~$(u_\T,m_\T)\in \left[V(\T)\right]^2$ solves~\eqref{eq:num_scheme}, 
we start by splitting the residuals $\Resi(u_\T,m_\T)$ between one component that will lead to standard residual estimators, and another component that comprises the effect of stabilization.
To do so, for each $i\in\{1,2\}$ and each $v\in H^1_0(\Omega)$, we have, for any $(u_\T,m_\T)\in \left[V(\T)\right]^2$,
\begin{equation}\label{eq:residual_split_1}
\langle \Resi(u_\T,m_\T),v\rangle = \langle \Resi(u_\T,m_\T),v-\IT v \rangle_{H^{-1}\times H^1_0} + \langle \Resi(m_{\T},u_\mathcal{T}), \IT v\rangle_{H^{-1}\times H^1_0},
\end{equation}
where $\IT$ was the quasi-interpolation operator considered above.
Note that the definition of $\Resi$ implies that if $(u_\T,m_\T)$ is a solution of the numerical scheme~\eqref{eq:num_scheme}, then the final term on the right-hand side of~\eqref{eq:residual_split_1} satisfies
\begin{equation}\label{eq:residual_split_2}
\begin{aligned}
\langle \Resi(m_{\T},u_\mathcal{T}), \IT v\rangle_{H^{-1}\times H^1_0} = S_i(u_\T,m_\T;\IT v) &&& \forall i\in\{1,2\}.
\end{aligned}
\end{equation}
Therefore, we have the following decomposition of the residual into two key components: for any solution~$(u_\T,m_\T)\in \left[V(\T)\right]^2$ of \eqref{eq:num_scheme},
\begin{multline}\label{eq:residual_split_3}
\norm{\Resi(u_\T,m_\T)}_{H^{-1}(\Omega)}\leq \sup_{v\in H^1_0(\Omega)\setminus\{0\}}\frac{\langle \Resi(u_\T,m_\T),v-\IT v \rangle_{H^{-1}\times H^1_0}}{\norm{\nabla v}_\Omega} 
\\ + \sup_{v\in H^1_0(\Omega)\setminus\{0\}}\frac{S_i(u_\T,m_\T;\IT v)}{\norm{\nabla v}_\Omega}.
\end{multline}
We will see shortly below that the first term on the right-hand side of~\eqref{eq:residual_split_3} will lead towards the well-known residual estimators, whereas the second term will lead to the stabilization estimators.

\subsection{Residual and stabilization estimators}
\paragraph{Residual estimators.}
The derivation of computable estimators for the first term on the right-hand side of~\eqref{eq:residual_split_3} is quite standard by now, and there is a range of possible approaches.
For the sake of simplicity and illustration, we present here the classical residual estimators following a standard approach, see for instance the textbook~\cite{verfurth2013posteriori}.
In the current setting, the classical definition of residual estimators require that $F[m_\T] \in L^2(\Omega)$, so we make the assumption  that $F\colon \mathcal{H}\tends L^2(\Omega)$ throughout the following analysis. We also assume that $G\in L^2(\Omega)$.
However, there are approaches to constructing \emph{a posteriori} estimators for distributional right-hand sides, see e.g.~\cite{ErnSmearsVohralik2017,KreuzerVeeser2021}.

For a given solution $(u_\T,m_\T)\in \left[V(\T)\right]^2$ of~\eqref{eq:num_scheme}, we start by defining the volume residual terms $\rKi\in L^2(K)$, $K\in\T$, by
\begin{subequations}\label{eq:residual_estimators_volume}
\begin{align}
\rKH &\coloneqq F[m_\T]|_K + \nu \Delta (u_\T|_K) -  H[\nabla u_\T]|_K , \\
\rKF &\coloneqq  G|_K + \nu \Delta (m_\T|_K) + \Div\left.\left(m_\T\frac{\partial H}{\partial p}[\nabla u_\T] \right)\right|_K.
\end{align}
\end{subequations}
Note that the dependence of $\rKi$ on $(u_T,m_\T)$ is left implicit to alleviate the notation.
Also, since functions in $V(\T)$ are piecewise linear, the contributions from the piecewise Laplacians $\Delta (u_\T|_K) $ and $\Delta (m_\T|_K)$ vanish identically; nevertheless, we write them explicitly in~\eqref{eq:residual_estimators_volume} to highlight the conceptual meaning of the volume residual terms.
Observe also that since $\nabla u_\T$ is piecewise constant with respect to the mesh $\T$, it follows from the hypotheses on $H$ that, for each $K\in\T$, the map $x\mapsto \frac{\partial H}{\partial p}[\nabla u_\T](x)$ is $H^1$-regular on $K$. This ensures that $\rKF \in L^2(K)$ for each $K\in\T$.

Next, for each interior face $F\in \calFi$, we define the jump residual terms $\ji \in L^2(F)$ by
\begin{subequations}\label{eq:residual_estimators_jumps}
\begin{align}
\jH &\coloneqq \nu \jump{ \nabla u_\T \cdot n_F },\\
\jF &\coloneqq \nu \jump{\nabla m_\T \cdot n_F } +   m_\T|_F\jump{ \frac{\partial H}{\partial p}[\nabla u_\T]\cdot n_F }.
\end{align}
\end{subequations}
Again, we leave the dependence of the jump residual on $(u_\T,m_\T)$ implicit.
In the notation above, the vector~$n_F\in \R^\dim$ denotes an arbitrary choice of unit normal to the face $F$, and the jump operator $\jump{\cdot}$ on $F\in \calFi$ is defined by $\jump{w}\coloneqq \tau_F^{-}w - \tau_F^{+}w$ for all $w\in BV(\Omega)$, where $\tau_F^{-}$, respectively $\tau_F^+$, denotes the trace operator from the side of $F$ for which $n_F$ is \emph{outward} pointing, respectively \emph{inward} pointing. Note that with this definition, the jump residual terms defined in~\eqref{eq:residual_estimators_jumps} are invariant with respect to the choice of normal vector~$n_F$, since reversing the orientation of $n_F$ also reverses the sign of the jump operator.
Note that since $\nabla u_\T$ is piecewise constant, it follows from the hypotheses on the problem data that $x\mapsto \frac{\partial H}{\partial p}[\nabla u_\T](x)$ is $H^1$-regular on each element $K\in \T$, and thus has well-defined one-sided traces in $L^2(F)$ on each face $F$. This ensures that $\jF$ is well-defined and is in $L^2(F)$ for each $F\in \calFi$.
Note also that in~\eqref{eq:residual_estimators_jumps} the term $m_\T|_F$ appears outside the jump operator owing to the continuity of $m_\T$.

With the residual and jump terms defined above, we define the elementwise residual estimator $\etaKi$, $K\in\T$, $i\in \{1,2\}$, by
\begin{equation}\label{eq:element_estimators}
\left[\etaKi\right]^2 \coloneqq h_K^2\norm{\rKi}_K^2 + \sum_{F\in \calFiK} h_F \norm{\ji}_F^2.
\end{equation}
The \emph{residual estimators}, denoted by $\etaResi$, $i\in \{1,2\}$, are then defined by
\begin{equation}\label{eq:residual-estimator}
\etaResi \coloneqq \left(\sum_{K\in\T}[\etaKi]^2\right)^{\frac{1}{2}}, \quad i\in\{1,2\}.
\end{equation}

\begin{remark}[Alternative estimators]\label{remark:alternative-estimators}
Note that one could also consider the numerous alternatives to the classical residual estimators, many of which have certain theoretical and/or practical advantages, especially in the case of distributional data~\cite{ErnSmearsVohralik2017,KreuzerVeeser2021}.
In practice, it is also desirable to employ \emph{a posteriori} estimators that control the linearization and algebraic errors committed in solving the discretized problems, see~\cite{ErnVohralik2013}; however a detailed analysis for this goes beyond the scope of this work.
\end{remark}

The next theorem states the well-known upper bound for the residuals with the above estimators.
\begin{theorem}[Upper bound on the residuals]\label{thm:residual_bound}
Suppose that $F\colon \mathcal{H}\rightarrow L^2(\Omega)$. 
Then, for any $(u_\T,m_\T)\in \left[V(\T)\right]^2$, we have
\begin{equation}\label{eq:residual_upper_bound}
\sup_{v\in H^1_0(\Omega)\setminus \{0\}}\frac{\langle \Resi(u_\T,m_\T),v-\IT v \rangle_{H^{-1}\times H^1_0}}{\norm{\nabla v}_\Omega} \lesssim \etaResi  \quad \forall i\in \{1,2\}.
\end{equation}
\end{theorem}
\begin{proof}
The essential techniques of the proof are already well-known, see for instance \cite[Theorem~3.58, p.~135]{verfurth2013posteriori}, so we only sketch the main steps.
For each $i\in\{1,2\}$, observe that the residual term $\langle \Resi(u_\T,m_\T), v-\IT v\rangle $ vanishes for any test function $v$ that belongs to the discrete space $V(\T)$. 
In other words, the space $V(\T)$ belongs to the kernel of the linear functional $(I-\IT^*)\Resi(u_\T,m_\T)\in H^{-1}(\Omega)$ where $\IT^* \colon H^{-1}(\Omega)\subset V(\T)^*\tends H^{-1}(\Omega)$ denotes the formal adjoint of $\IT$.
Furthermore, for any $v\in H^1_0(\Omega)$, using integration by parts over the elements of $\T$, we obtain the identity
\begin{equation}\label{eq:residual_representation}
\langle \Resi(u_\T,m_\T), v-\IT v\rangle_{H^{-1}\times H^1_0} = \sum_{K\in \T} \int_K \rKi(v-\IT v) \dx - \sum_{F\in\calFi}\int_F \ji (v-\IT v) \ds.
\end{equation}
We then obtain~\eqref{eq:residual_upper_bound} by applying the Cauchy--Schwarz inequality to~\eqref{eq:residual_representation}, and applying the approximation properties of $\IT$ from~\eqref{eq:quasi_approx_bound_elements} and~\eqref{eq:quasi_approx_bound_faces}.
\end{proof}

\paragraph{Stabilization estimators.}
We now turn to bounding the final term in~\eqref{eq:residual_split_3}. 
Our approach here differs from the one in~\cite{TobiskaVerfurth2015} in order to handle the very broad class of abstract stabilizations considered above.
Note that, for any $(u_\T,m_\T)\in \left[V(\T)\right]^2$,
\begin{equation}
\begin{split}
 \sup_{v\in H^1_0(\Omega)\setminus\{0\}}\frac{S_i(u_\T,m_\T;\IT v)}{\norm{\nabla v}_\Omega} & = \sup_{v\in H^1_0(\Omega)\setminus\{0\}} \frac{\norm{\nabla \IT v}_{\Omega}}{\norm{\nabla v }_{\Omega}} \frac{S_i(u_\T,m_\T;\IT v)}{\norm{\nabla \IT v}_\Omega } 
 \\ &\lesssim \sup_{v_\T\in V(\T)\setminus\{0\}}\frac{S_i(u_\T,m_\T;v_\T)}{\norm{\nabla v_\T}_\Omega},
 \end{split}
\end{equation}
where in passing to the final inequality above, we have used~\eqref{eq:quasi_approx_bound_elements}, which entails the $H^1$-stability of $\IT$, and we have also used the Poincar\'e inequality.
This leads us to consider the \emph{stabilization estimators}
\begin{equation}\label{eq:stab_estimator}
\etastabi \coloneqq \sup_{v_\T\in V(\T)\setminus\{0\}}
\frac{S_i(u_\T,m_\T;v_\T)}{\norm{\nabla v_\T}_\Omega} , \quad i\in\{1,2\}.
\end{equation} 
As above, we leave implicit the dependence of $\etastabi$ on $(u_\T,m_\T)$ in order to alleviate the notation.
Observe that if $(u_\T,m_\T)$ solves~\eqref{eq:num_scheme}, then $S_i(u_\T,m_\T;v_\T)=\langle \Resi(u_\T,m_\T),v_\T\rangle_{H^{-1}\times H^1_0}$ for all $v_\T\in V(\T)$. 
Thus the stabilization estimators $\etastabi$ satisfy the global efficiency bound
\begin{equation}\label{eq:stabilization_efficiency}
    \etastabi  \lesssim \norm{\Resi(u_\T,m_\T)}_{H^{-1}(\Omega)} \lesssim \norm{u-u_\T}_{H^{1}(\Omega)}+\norm{m-m_\T}_{H^{1}(\Omega)},
\end{equation}
for each $i\in\{1,2\}$.
However, in the case of general stabilizations, it is not known how to localize $\etastabi$ whilst retaining local efficiency. 
This motivates the analysis in Section~\ref{sec:affine_preserving_stab} for the case of affine-preserving stabilizations.

\begin{remark}[Practical computation of stabilization estimators]\label{rem:discrete_computation}
Observe that the stabilization estimator is a dual norm of the stabilization acting over the space of discrete test functions, so it is a computable quantity.
In practice, it is often cheaper computationally to use well-known preconditioners for elliptic problems that lead to approximations $\widetilde{\eta}_{\mathrm{stab},i} \approx \etastabi$ satisfying the equivalence
\begin{equation}\label{eq:equivalence_stabilization}
\gamma\,\etastabi  \leq \widetilde{\eta}_{\mathrm{stab},i}\leq \Gamma \,\etastabi  \quad \forall i\in\{1,2\},
\end{equation}
for some fixed positive constants $\gamma$, $\Gamma$ that do not depend on the mesh-size of~$\T$.
Indeed, consider a positive definite symmetric preconditioner $P_\T\colon V(\T)\rightarrow V(\T)^*$ for the discrete Laplacian on $V(\T)$, such that the action of $P_\T^{-1}\colon V(\T)^*\rightarrow V(\T)$ is easily computable (e.g.\ a single multigrid V-cycle), and such that $P_\T$ satisfies the spectral equivalence property: for some fixed $\gamma$, $\Gamma >0$ independent of $\T$, $\gamma \norm{v_\T}_{P_\T}\leq \norm{\nabla v_\T}_\Omega \leq \Gamma\norm{v_\T}_{P_\T}$ for all $v_\T\in V(\T)$, where the norm $\norm{\cdot}_{P_\T}\coloneqq \sqrt{\langle P_\T \cdot, \cdot\rangle_{V(\T)^*\times V(\T)}}$.
Then, we compute 
\begin{equation}
\begin{aligned}
w_{\T,i}\coloneqq P_\T^{-1} S_i(u_\T,m_\T;\cdot)\in V(\T), &&& \widetilde{\eta}_{\mathrm{stab},i}\coloneqq \sqrt{S_i(u_\T,m_\T;w_{\T,i})}=\norm{w_{\T,i}}_{P_\T}.
\end{aligned}
\end{equation}
It follows from the spectral equivalence property of $P_\T$ that $\widetilde{\eta}_{\mathrm{stab},i}$ satisfies the equivalence~\eqref{eq:equivalence_stabilization}.
Thus, up to the constants in~\eqref{eq:equivalence_stabilization}, one can often approximate $\etastabi$ at a smaller computation cost. 
In the following analysis we shall focus on the discussion the stabilization estimators $\etastabi$ defined in~\eqref{eq:stab_estimator}, yet we note that these can be substituted for any equivalent approximation throughout this work with identical results up to the constants.
For an introduction to preconditioners in the context of FEM, including various possible choices for $P_\T$, we refer the reader to the textbook~\cite{ElmanSilvesterWathen2014}.
Note also that in Section~\ref{sec:affine_preserving_stab} below, we show that the computation (or approximation) of the stabilization estimators can be avoided in the case of affine-preserving stabilizations.
\end{remark}

\subsection{Error bounds for general stabilizations}

We now define the total estimator
\begin{equation} \label{eq:total-estimator}
\eta(u_\T,m_\T) \coloneqq \sum_{i=1}^2\left[\etaResi+\etastabi\right].
\end{equation}
In Theorem~\ref{thm:error_upper_bound} below we give the main upper bound on the error in the general case where~\ref{H:stabilization_main} holds.
This result shows that, provided that the total estimator $\eta(u_\T,m_\T)$ is sufficiently small, then the estimator gives an upper bound on the error between the exact solution and its numerical approximation.

\begin{theorem}[Upper bound on the error]\label{thm:error_upper_bound}
Assume~\ref{H:stabilization_main}, that $G\in L^2(\Omega)$, and that $F\colon \mathcal{H}\to L^2(\Omega)$.
Suppose that $(u_\T,m_\T)\in \left[V(\T)\right]^2$ is a solution of~\eqref{eq:num_scheme}.
There exists an $\widetilde{\mathcal{R}}_*>0$ such that if $\eta(u_\T,m_\T) \leq \widetilde{\mathcal{R}}_*$, then
\begin{equation}\label{eq:error_upper_bound}
\norm{u-u_\T}_{H^1(\Omega)}+\norm{m-m_\T}_{H^1(\Omega)}\lesssim \eta (u_\T,m_\T).
\end{equation}
The constant $\widetilde{\mathcal{R}}_*$ depends possibly on the problem data in~\eqref{sys} and on $\theta_\T$ the shape-regularity parameter of $\T$, but not on $(u_\T,m_\T)$ or the size of the mesh.
\end{theorem}
\begin{proof}
First, recall that~\ref{H:stabilization_main} states that if $(u_\T,m_\T)$ is a solution of~\eqref{eq:num_scheme}, then $m_\T\geq 0$ in $\Omega$.
Next, by combining the bound~\eqref{eq:residual_split_3}, Theorem~\ref{thm:residual_bound} and \eqref{eq:stab_estimator}, we see that the residual $\mathcal{R}(u_\T,m_\T)$ satisfies
\begin{equation}\label{eq:error_upper_bound_1}
\mathcal{R}(u_\T,m_\T)\lesssim \eta (u_\T,m_\T).
\end{equation}
Therefore, it follows that if $\eta(u_\T,m_\T)$ is sufficiently small, then the pair $(u_\T,m_\T)$ satisfies the condition~\eqref{eq:error_res_equiv_condition} from Theorem~\ref{thm:error-res-equiv}. 
Thus we obtain~\eqref{eq:error_upper_bound} from~\eqref{eq:error_res_equivalence} and from~\eqref{eq:error_upper_bound_1}.
\end{proof}

We now consider the efficiency of the estimators, i.e.\ we show that the estimators are bounded from above by the error, up to some data oscillation terms.
These data oscillation terms measure the approximability of the data and the nonlinear terms $F[m_\T]$, $H[\nabla u_\T]$ and $\frac{\partial H}{\partial p}[\nabla u_\T]$ by piecewise polynomial functions.
In order to define the data oscillation terms, we consider a fixed but arbitrary polynomial degree $\kappa\geq 0$, and we define the following approximation operators.
For each $K\in\T$ and each nonnegative integer $\kappa\geq 0$, let $\Pi_{K,\kappa}\colon L^2(K)\rightarrow \mathcal{P}_\kappa(K)$ denote the $L^2$-projection operator onto polynomials of degree at most $\kappa$ defined on $F$.
For each face $F\in\calFi$, let $\Pi_{F,\kappa}\coloneq L^2(F)\rightarrow \mathcal{P}_\kappa(F)$ denote the $L^2$-projection operator onto the space of polynomials of degree at most $\kappa$ defined over $F$.
For each element $K\in\T$, $i\in\{1,2\}$ we define the data oscillation terms
\begin{equation}\label{eq:osc}
[\oscKqi]^2 \coloneqq \sum_{K^\prime \in \tTK} h_{K^\prime}^2 \norm{r_{K^\prime,i}-\Pi_{K^\prime,\kappa}r_{K^\prime,i}}_{K^\prime}^2 + \sum_{F\in\calFKi} h_F\norm{\ji-\Pi_{F,\kappa}\ji}^2_F.
\end{equation}
Observe that
\begin{equation}
[\oscKqH]^2= \sum_{K^\prime \in \tTK} h_{K^\prime}^2\norm{\left(\mathrm{I}-\Pi_{K^\prime,\kappa}\right)\left(F[m_\T]-H[\nabla u_\T]\right)}_{K^\prime}^2.
\end{equation}
Additionally, observe that
\begin{multline}
 [\oscKqF]^2 = \sum_{K^\prime\in \TK} h_{K^\prime}^2\norm{\left(\mathrm{I}-\Pi_{K^\prime,\kappa}\right)\left(G+\Div\left(m_\T\frac{\partial H}{\partial p}[\nabla u_T] \right)\right)}_{K^\prime}^2
 \\ + \sum_{F\in \calFKi} h_F\norm{(\mathrm{I}-\Pi_{F,\kappa})\left(m_\T\jump{\frac{\partial H}{\partial p}[\nabla u_T]\cdot n_F}\right) }_F^2.
\end{multline}

The following theorem shows the local efficiency of the residual estimators \eqref{eq:residual-estimator}. 
\begin{theorem}[Local efficiency of residual estimators]\label{theorem:eta-lower-bound}
Let $\kappa\geq 0$ be a nonnegative integer.
Assume~\ref{H:stabilization_main}, that $G\in L^2(\Omega)$, and that $F\colon \mathcal{H}\rightarrow L^2(\Omega)$.
For any $(u_\T,m_\T)\in \left[V(\T)\right]^2$, and any $K \in \mathcal{T}$, it holds that
\begin{subequations}\label{eq:eta-lower-bounds}
\begin{align}
\etaKH &\lesssim \norm{F[m]-F[m_\T]}_{H^{-1}(\omK)}+\norm{\nabla(u-u_\T)}_{\omK} + \oscKqH,\label{eq:eta-lower-bound_1}
\\
\etaKF &\lesssim \norm{m-m_\T}_{H^1(\omK)}+\norm{\nabla(u-u_\T)}_{
\omK} + \oscKqF,\label{eq:eta-lower-bound_2}
\end{align}
\end{subequations}
where the hidden constants depend on the shape-regularity of $\mathcal{T}$, the constants $M_\infty$, $L_H$, and $L_{H_p}$, $\nu$, and~$\kappa$.
\end{theorem}
We postpone the proof of Theorem~\ref{theorem:eta-lower-bound} to Section~\ref
{sec:proof_of_theorem_ref_theorem_eta_lower_bound} below.
The next result expresses the global efficiency for the total estimator, which comprises both the residual and stabilization estimators. 
This shows that, up to data oscillation and constants, the total estimator is a lower bound for the total error.

\begin{corollary}[Global efficiency]\label{corollary:eta-lower-bound}
Assume~\ref{H:stabilization_main}, that $G\in L^2(\Omega)$, and that $F\colon \mathcal{H}\rightarrow L^2(\Omega)$.
Suppose that $(u_\mathcal{T},m_\mathcal{T}) \in V(\mathcal{T})\times V(\T)$ is a solution of~\eqref{eq:num_scheme}.
Then,
\begin{equation}
      \eta(u_\T,m_\T)  \lesssim \norm{m -m_\mathcal{T}}_{H^1(\Omega)} + \norm{u -u_\T}_{H^1(\Omega)}   + \left[\sum_{i=1}^2 \sum_{K\in\T} [\oscKqi]^2 \right]^{\frac{1}{2}}, 
\end{equation}
where the hidden constant depends on the shape-regularity of $\mathcal{T}$, the constants $M_\infty$, $L_F$, $L_H$, and $L_{H_p}$, the diffusion parameter $\nu$, and the Poincar\'{e} constant of $\Omega$.
\end{corollary}
\begin{proof}
The result is essentially immediate after summing up the bounds~\eqref{eq:eta-lower-bounds} over all elements of the mesh and combining with the efficiency of the stabilization estimators in~\eqref{eq:stabilization_efficiency}.
The only nonobvious step that is required is to use the inequality 
$$\sum_{K\in\T}\norm{F[m]-F[m_\T]}_{H^{-1}(\omK)}^2\leq(d+2)\norm{F[m]-F[m_\T]}_{H^{-1}(\Omega)}^2,$$
which is deduced from the fact that each element of $\T$ is contained in at most $d+2$ of the patches $\{\omK,\; K\in\T\}$. We then use the Lipschitz continuity of $F$ in~\eqref{F2}, and the hypothesis that $H^1_0(\Omega)$ is continuously embedded into $\mathcal{H}$.
\end{proof}

\begin{remark}
In practical computations, the constant $\widetilde{\mathcal{R}}_*$ appearing in Theorem~\ref{thm:error_upper_bound} may be hard to compute \emph{a priori}, since it depends on the problem data. 
Nevertheless, Theorem~\ref{thm:error_upper_bound} still guarantees that convergence to zero of the estimators implies convergence of the numerical approximations to the exact solution.
Note also that there is no smallness condition for the efficiency bounds, c.f.\ Corollary~\ref{corollary:eta-lower-bound}.
\end{remark}

\begin{remark}[Data oscillation]
We stress that the polynomial degree $\kappa$ in Theorem~\ref{theorem:eta-lower-bound} is arbitrary, although the constants in the bounds do depend on $\kappa$. Therefore, in many cases the oscillation terms $\oscKqi$ are higher order than the errors, and might even vanish entirely in some cases. For instance, if the Hamiltonian $H$ has no spatial dependence, i.e.\ $H(x,p)=H(p)$ for all $x\in \overline{\Omega}$, then $H[\nabla u_\T]$ and $\frac{\partial H}{\partial p}[\nabla u_\T]$ are piecewise constant over the mesh $\T$, and thus these are reproduced exactly by the local projection operators for any $\kappa\geq 0$.
Regarding the possibly nonlocal operator $F$, we can consider the following situation for the sake of illustration: suppose that we are given a shape-regular sequence of conforming simplicial meshes $\{\T_k\}_{k\in\N}$ such that the numerical scheme is stable with some uniform \emph{a priori} bound on $\norm{m_{\T_k}}_{H^1(\Omega)}$ (see for instance \cite[Eq.~(5.6a)]{osborne2024analysis}), and suppose also that $F$ is continuous from its domain $\mathcal{H}$ into $L^2(\Omega)$.
Then the compactness of the embedding of $H^1_0(\Omega)$ into $\mathcal{H}$ implies the precompactness in $L^2(\Omega)$ of the set $\{F[m_{\T_k}]\}_{k\in\N}$.
It is then easy to show that $\left(\sum_{K\in\T}h_K^2\norm{F[m_\T]-\Pi_{K,\kappa}F[m_\T]}_K^2\right)^{\frac{1}{2}}= o ( h_k )$ where $h_k\coloneqq \max_{K\in\T_k}h_K$ denotes the maximum mesh-size of $\T_k$. 
In other words, the contribution to data oscillation from the approximation of $F[m_\T]$ is then higher-order than the error in many cases.
Despite these examples, it is known however that the classical $L^2$-data oscillation term as considered does have some significant shortcomings for problems with rough data, see \cite{KreuzerVeeser2021} for a detailed discussion in the case of the Poisson equation.
\end{remark}

\subsection{Proof of Theorem~\ref{theorem:eta-lower-bound}}\label{sec:proof_of_theorem_ref_theorem_eta_lower_bound}
To prove the lower bounds, we start by recalling some well-known bounds for the terms in the residual bounds in Lemma~\ref{lem:bubble_functions} below.
\begin{lemma}\label{lem:bubble_functions}
For each $K\in\T$ and any $\kappa\geq 0$, and any $r\in L^2(K)$, we have
\begin{equation}\label{eq:bubble_function_bound_1}
h_K\norm{r}_K\lesssim \left(\sup_{w\in H^1_0(K)\setminus\{0\}}\frac{\int_K r w\dx }{\norm{\nabla w}_K} \right)+h_K\norm{r-\Pi_{K,\kappa} r}_K.
\end{equation}
For each interior face $F\in \mathcal{F}^I$, any $\kappa\geq 0$, and any $j\in L^2(F)$, we have
\begin{equation}\label{eq:bubble_function_bound_2}
h_F^{\frac{1}{2}}\norm{j}_F \lesssim \left(\sup_{w\in H^1_0(\omega_F)\setminus\{0\}}\frac{\int_F j w \ds}{\norm{\nabla w}_{\omega_F}}\right)+h_F^{\frac{1}{2}}\norm{j-\Pi_{F,\kappa}j}_F.
\end{equation}
The hidden constants in~\eqref{eq:bubble_function_bound_1} and~\eqref{eq:bubble_function_bound_2} depend only on the polynomial degree~$\kappa$, the dimension $\dim$, and the shape-regularity of $\T$.
\end{lemma}
\begin{proof}
This result is derived from the well-known bubble function argument, see e.g.~\cite[Section~3.8.2]{verfurth2013posteriori}, so we only outline here the main steps.
The bound~\eqref{eq:bubble_function_bound_1} is shown by first applying the triangle inequality $h_K\norm{r}_K\leq h_K\norm{\Pi_{K,\kappa}r}_K+h_K\norm{r-\Pi_{K,\kappa}r}_K$, followed by applying the inverse inequality $h_K\norm{\Pi_{K,\kappa}r}_{K}\lesssim \sup_{w\in H^1_0(K)\setminus\{0\}}\frac{\int_K \Pi_{K,\kappa}r w\dx }{\norm{\nabla w}_K}$, see e.g.~\cite[Theorem~3.59]{verfurth2013posteriori}, followed by a further triangle inequality and the Poincar\'e inequality for functions in $H^1_0(K)$.
The bound~\eqref{eq:bubble_function_bound_2} follows by an analoguous argument.
\end{proof}

\paragraph{Proof of Theorem~\ref{theorem:eta-lower-bound}}
We start by considering an arbitrary $K\in\T$ and arbitrary $w\in H^1_0(K)$. Then, integration-by-parts implies that
\begin{equation}
\int_K \rKi w\dx = \langle \Resi(u_\T,m_\T), w\rangle_{H^{-1}\times H^1_0}  \quad \forall i\in\{1,2\}.
\end{equation}
Therefore, we obtain
\begin{subequations}
\begin{align}
 \int_K \rKH w \dx &=\langle F[m_\T]-F[m], w\rangle_{H^{-1}\times H^1_0} + \int_\Omega \nu \nabla(u-u_\T)\cdot\nabla w + (H[\nabla u]-H[\nabla u_\T])w \dx,
 \\
 \int_K \rKF w \dx &=\int_K \left[ \nu \nabla(m-m_\T)\cdot\nabla w + \left(m\frac{\partial H}{\partial p}[\nabla u]-m_\T\frac{\partial H}{\partial p}[\nabla u_\T]\right)\cdot \nabla w \right]\dx. \label{eq:lower_bound_r2_elements_1}
\end{align}
\end{subequations}
Therefore, applying~\eqref{eq:bubble_function_bound_1}, we get
\begin{equation}\label{eq:lower_bound_r1_element}
h_K\norm{\rKH}_K \lesssim \norm{F[m]-F[m_\T]}_{H^{-1}(K)}+\norm{\nabla(u-u_\T)}_K+h_K\norm{\rKH-\Pi_{K,\kappa}\rKH}_K.
\end{equation}
The triangle inequality followed by the bounds~\eqref{h-deriv-linf-bound} and~\eqref{eq:Hp_bound} imply that
\begin{equation}\label{eq:lower_bound_r2_elements_2}
\begin{split}
\norm{m\frac{\partial H}{\partial p}[\nabla u]-m_\T\frac{\partial H}{\partial p}[\nabla u_\T]}_K &\leq \norm{m \left(\frac{\partial H}{\partial p}[\nabla u] - \frac{\partial H}{\partial p}[\nabla u_\T]\right) }_K+ \norm{(m-m_\T)\frac{\partial H}{\partial p}[\nabla u_\T]}_K
\\
& \leq \norm{m}_{L^{\infty}({\Omega})}L_{H_p}\norm{\nabla(u-u_\T)}_K + L_H\norm{m-m_\T}_K.
\end{split}
\end{equation}
Recall that \eqref{m-ess-bound} provides a maximum norm bound for the density $m$.
Therefore, combining \eqref{eq:bubble_function_bound_1} with \eqref{eq:lower_bound_r2_elements_1} and~\eqref{eq:lower_bound_r2_elements_2}, we find that
\begin{equation}\label{eq:lower_bound_r2_elements}
h_K \norm{\rKF}_K\lesssim \norm{m-m_\T}_{H^1(K)}+\norm{\nabla(u-u_\T)}_K+h_K\norm{\rKF-\Pi_{K,\kappa}\rKF }_K,
\end{equation}
where the hidden constant depends on the constants $M_\infty$ from~\eqref{m-ess-bound}, $L_{H_p}$ and $L_H$ as well as $\nu$, $\dim$, $\kappa$, and the shape-regularity of $\T$.
Next, for each interior face $F\in \calFi$, and any $w\in H^1_0(\omega_F)$, it is straightforward to check that integration-by-parts gives
\begin{equation}
\begin{aligned}
\int_F \ji w\ds = \sum_{K\in\T_F} \int_K \rKi w \dx - \langle \Resi(u_\T,m_\T),w \rangle_{H^{-1}\times H^1_0} &&& \forall i\in\{1,2\},
\end{aligned}
\end{equation}
where it is recalled that $\TF$ denotes the set of elements containing $F$.
Therefore, using~\eqref{eq:bubble_function_bound_2} (where we note that $\jH=\nu \jump{\nabla u_\T\cdot n_F}$ is a constant function on $F$), the Poincar\'e inequality on the patch $\omega_F$,  and the bound~\eqref{eq:lower_bound_r1_element}, we find that, for each $F\in\calFi$,
\begin{equation}
h_F^{\frac{1}{2}} \norm{\jH}_F \lesssim \norm{F[m]-F[m_\T]}_{H^{-1}(\omega_F)}+\norm{\nabla(u-u_\T)}_{\omega_F}+\sum_{K\in \T_F}h_K\norm{\rKH-\Pi_{K,\kappa}\rKH}_K.
\label{eq:lower_bound_j1_faces}
\end{equation}
Similarly,
\begin{multline}
h_F^{\frac{1}{2}}\norm{\jF}_F \lesssim \norm{\nabla(m-m_\T)}_{H^1(\omega_F)}+\norm{\nabla(u-u_\T)}_{\omega_F}
\\ + h_F^{\frac{1}{2}}\norm{\jF-\Pi_{F,\kappa}\jF}_F +\sum_{K\in\T_F}h_K\norm{\rKF-\Pi_{K,\kappa}\rKF}_K.\label{eq:lower_bound_j2_faces_1}
\end{multline}
Note that if $K\in \T$ and $F\in \calFKi$, then $\omega_F\subset \omK$, and thus after summing~\eqref{eq:lower_bound_j1_faces} over all faces of $K$, we get, for each $K\in\T$,
\begin{equation}\label{eq:lower_bound_j1_faces_2}
\sum_{F\in\calFKi}h_F^{\frac{1}{2}} \norm{\jH}_F \lesssim \norm{F[m]-F[m_\T]}_{H^{-1}(\omK)}+\norm{\nabla(u-u_\T)}_{\omK}+\oscKqH.
\end{equation}
Similarly, using~\eqref{eq:lower_bound_r2_elements} and~\eqref{eq:lower_bound_j2_faces_1}, we find that, for each $K\in\T$,
\begin{equation}\label{eq:lower_bound_j2_faces}
\sum_{F\in\calFKi}h_F^{\frac{1}{2}}\norm{\jF}_F \lesssim \norm{m-m_\T}_{H^1(\omK)}+\norm{\nabla(u-u_\T)}_{\omK}+\oscKqF.
\end{equation}
Thus we obtain~\eqref{eq:eta-lower-bound_1} by combining~\eqref{eq:lower_bound_r1_element} with \eqref{eq:lower_bound_j1_faces_2}. 
Similarly, we obtain~\eqref{eq:eta-lower-bound_2} by combining~\eqref{eq:lower_bound_r2_elements} with \eqref{eq:lower_bound_j2_faces}.
\hfill\qedsymbol

\section{A posteriori error bounds for affine-preserving stabilizations}\label{sec:affine_preserving_stab}
In this section, we show that the analysis can be refined if, in addition to the hypothesis~\ref{H:stabilization_main}, the stabilizations satisfies some further structural assumptions, namely Lipschitz continuity, maximum-norm stability of the density approximation, and patchwise affine-preservation.
We present each of these assumptions in turn.

\subsection{Lipschitz continuity}
We assume the following local Lipschitz continuity property on the stabilizations.
\begin{enumerate}[label={(H\arabic*)},resume]
\item \emph{Lipschitz continuity.} For each $i\in\{1,2\}$, for each interior vertex $z\in \calVi$, 
\begin{equation}
\abs{
 S_i(v_{\T},w_{\T};\psi_{z})-S_i(\widetilde{v}_{\T},\widetilde{w}_\T; \psi_{z} ) } \lesssim \abs{\omz}_d^{\frac{1}{2}}\left(\norm{\nabla(v_{\T}-\widetilde{v}_{\T})}_{\omega_{z}}+\norm{ \nabla(w_{\T}-\widetilde{w}_{\T})}_{\omega_{z}}\right),
\end{equation}
for all $(v_{\T},w_{\T},\widetilde{v}_{\T},\widetilde{w}_{\T})\in \left[V(\T)\right]^4$, where the hidden constant is independent of the mesh-size and of the vertex $z\in \mathcal{V}$.
\label{H:stabilization_local_Lipschitz}
\end{enumerate}
Recall that $|\omz|_\dim$ denotes the Lebesgue measure of $\omz$.
The assumption~\ref{H:stabilization_local_Lipschitz} contains several ingredients. 
Observe that in addition to Lipschitz continuity with respect to the first two arguments, it essentially requires that the stabilizations should be acting locally, i.e.\ when testing with a basis function $\psi_{z}$ in the third argument, the stabilizations may only depend on the values of the first two arguments over the corresponding vertex patch $\omega_{z}$. 
It also further requires, in essence, that the stabilization terms should scale as the first-order (or lower) terms of the differential operator.

\begin{example}
Consider the stabilizations $S_1$ and $S_2$ defined by~\eqref{eq:linear_stabilization}, and suppose that the matrix-valued function $D_\T\colon \Omega \rightarrow \Rddsymp$ appearing in~\eqref{eq:linear_stabilization} satisfies $ \norm{D_\mathcal{T}}_{L^{\infty}\left(\omz;{\Rdd}\right)}\lesssim h_{\omz}$ for each vertex $z\in \calVi$.
For instance, this is satisfied when $D_\T$ is given by~\eqref{edge-tensor-formula} and the weights $\edgeweight \eqsim \diam E$, as is used in~\cite{osborne2023finite,osborne2024near}.
By applying the Cauchy-Schwarz inequality, then applying inverse and H\"{o}lder inequalities, we obtain
\begin{equation}
\begin{split}
\abs{S_1(v_{\T},w_{\T};\psi_{z})-S_1(\widetilde{v}_{\T},\widetilde{w}_\T; \psi_{z} ) } 
&\leq \norm{D_\mathcal{T} \nabla (v_\mathcal{T} - \widetilde{v}_\T)}_{\omz} \norm{\nabla \psi_z}_{\omz} \\
&\lesssim \norm{D_\mathcal{T}}_{L^{\infty}(\omz;\Rdd)} \norm{\nabla (v_\mathcal{T} - \widetilde{v}_\T)}_{\omz} h_{\omz}^{-1} \abs{\omz}_{d}^{\frac{1}{2}} \\
&\lesssim  \norm{\nabla (v_\mathcal{T} - \widetilde{v}_\T)}_{\omz} \abs{\omz}_{d}^{\frac{1}{2}}.
\end{split}
\end{equation}
A similar bound for $S_2(\cdot,\cdot ; \cdot)$ follows from the same arguments. This verifies~\ref{H:stabilization_local_Lipschitz} for this example.
\end{example}

\subsection{Maximum-norm bound}
We also assume the following bound for the maximum-norm of $m_\T$.
\begin{enumerate}[label={(H\arabic*)},resume]
\item \label{H:maxnorm_bound} \emph{Maximum norm bound.} 
We assume that the datum $G$ and the numerical scheme are such that there is a uniform maximum norm bound for the computed density, i.e.\ any solution $(u_\T,m_\T)$ of~\eqref{eq:num_scheme} satisfies
\begin{equation}\label{eq:max-norm}
\norm{m_\T}_{C(\overline{\Omega})} \leq \widetilde{M}_\infty, 
\end{equation}
for some constant $\widetilde{M}_\infty$, which may depend on the problem data and the shape-regularity parameter $\theta_\T$ but not on $m_\T$ or the mesh-size of $\T$.
\end{enumerate}
The maximum norm bound~\eqref{eq:max-norm} is the discrete analogue of~\eqref{m-ess-bound}.
We show in Proposition~\ref{prop:discrete_maxnorm_bound} below how to verify~\ref{H:maxnorm_bound} for problems in two and three space-dimensions, in the case of the stabilizations from~\cite{osborne2023finite,osborne2024near} on meshes satisfying the Xu--Zikatanov condition.
It can also be verified for a wider range of stabilizations satisfying a DMP.

\begin{proposition}\label{prop:discrete_maxnorm_bound}
Let $\dim\in \{2,3\}$.
Suppose that $G\in H^{-1}(\Omega)$ satisfies $G=g_0 - \nabla {\cdot} g_1$ with $g_0\in L^{q/2}(\Omega)$ and $g_1\in L^q(\Omega;\R^\dim)$ for some $q>\dim$.
Let $\T$ satisfy the Xu--Zikatanov condition.
Let $S_2$ be defined by~\eqref{eq:linear_stabilization} and~\eqref{edge-tensor-formula}, with weights $\edgeweight$ verifying the hypotheses of~\cite[Theorem~4.2]{osborne2023finite}. Suppose that $(u_\T,m_\T)\in \left[V(\T)\right]^2$ solves~\eqref{eq:num_scheme}.
Then,
    \begin{equation}\label{eq:maxnorm_bound_example}
        \norm{m_\T}_{C(\overline{\Omega})} \lesssim \norm{G}_{H^{-1}(\Omega)}+\norm{g_0}_{L^{q/2}(\Omega)}+\norm{g_1}_{L^q(\Omega)}.
    \end{equation}
\end{proposition}
Proposition~\ref{prop:discrete_maxnorm_bound} therefore verifies~\ref{H:maxnorm_bound} for the stabilizations of~\cite{osborne2023finite,osborne2024near}, for some constant $\widetilde{M}_\infty \lesssim \norm{G}_{H^{-1}(\Omega)}+\norm{g_0}_{L^{q/2}(\Omega)}+\norm{g_1}_{L^q(\Omega;\R^\dim)}$.

\begin{proof}
The proof is essentially an application of \cite[Theorem~1]{ciarlet1973maximum}, yet we make a small adjustment of their proof in order to reduce the integrability condition on $g_0$.
Note that in this setting, it follows from~\cite[Lemma~3.1]{osborne2024near} that
\begin{equation}\label{eq:uniform_m_H1_bound}
\norm{m_\T}_{H^1(\Omega)}\lesssim \norm{G}_{H^{-1}(\Omega)},
\end{equation}
see~\cite[Remark~3.2]{osborne2024near} for further details.
Then, we can write~\eqref{eq:num_scheme_fp} equivalently as
\begin{equation}
a(m_\T,w_T) = (f_0, w_\T)_\Omega + (f_1, \nabla w_\T)_\Omega \quad \forall w_\T \in V(\T),
\end{equation}
where the bilinear form $a(m_\T,w_\T) \coloneqq \int_{\Omega} (\nu \mathbb{I}_d + D_\T) \nabla m_\T \cdot \nabla w_\T\ \dx$, and with right-hand side terms $f_0 \coloneqq g_0 \in L^{q/2}(\Omega)$ and $f_1 \coloneqq g_1 - m_\T \frac{\partial H}{\partial p}[\nabla u_\T] \in L^{q_*}(\Omega;\R^\dim)$, $q_* \coloneqq \min\left(q, \frac{2d}{d-2}\right)$, where in particular the Sobolev embedding theorem and~\eqref{eq:uniform_m_H1_bound} above imply that
\begin{equation}\label{eq:maxnorm_rhs_data_bounds}
\norm{f_1}_{L^{q_*}(\Omega;\R^\dim)}\leq \norm{g_1}_{L^q(\Omega;\R^\dim)}+L_H \norm{m_\T}_{L^{q_*}(\Omega)}\lesssim \norm{g_1}_{L^{q}(\Omega;\R^\dim)}+\norm{G}_{H^{-1}(\Omega)}.
\end{equation}
Observe that $q\geq q_*>\dim$ for $\dim \in \{2,3\}$.
Since $D_\T$ is symmetric positive semi-definite a.e.\ in $\Omega$, we see that the bilinear form $a(\cdot,\cdot)$ is coercive, i.e. $a(v,v)\geq \nu \norm{\nabla v}_\Omega^2$ for all $v\in H^1_0(\Omega)$.
Note also that the bilinear form $a(\cdot,\cdot)$ is of \emph{nonnegative type} in the sense of~\cite[Eqs.~(2.25), (2.26)]{ciarlet1973maximum} as a result of the Xu--Zikatanov condition{\ifJ, c.f.~\cite{xu1999monotone},\else~\cite{xu1999monotone}\fi} and the definition of $D_\T$ in~\eqref{edge-tensor-formula}, see~\cite[Lemma~A.1]{osborne2023finite} for details.
Therefore, the hypotheses~\cite[Eqs.~(2.1), (2.2), (2.3), (2.25), (2.26)]{ciarlet1973maximum} are all verified.
The proof then essentially follows~\cite{ciarlet1973maximum}, however we can improve the inequality~\cite[Eq.~(3.12)]{ciarlet1973maximum} by using the sharper bound: for any function $v\in H^1_0(\Omega)$ that is supported on a measurable set $E\subset \Omega$, we have
\begin{equation}\label{eq:ciarlet_raviart_improved}
|(f_0,v)_\Omega + (f_1,\nabla v)_\Omega| \lesssim |E|_\dim^{\frac{1}{2}-\frac{1}{q_*}} \left( \norm{f_0}_{L^{q_*/2}(\Omega)} +\norm{f_1}_{L^{q_*}(\Omega;\R^\dim)} \right)\norm{v}_{H^1(\Omega)},
\end{equation}
where $|E|_\dim$ is the Lebesgue measure of $E$. 
In other words, it is sufficient for the analysis to have $f_0=g_0\in L^{q/2}(\Omega)$ for some $q>\dim$ instead of the condition $f_0\in L^q(\Omega)$, $q>\dim$, that was assumed in~\cite{ciarlet1973maximum}.
Indeed, to show~\eqref{eq:ciarlet_raviart_improved}, we define the index $\sigma\coloneqq \frac{2q_*}{q_*-2}$ and note that $\sigma< \frac{2\dim}{\dim-2}$ since $q_*>\dim$. 
Then, the H\"older inequality and the Sobolev embedding theorem imply that 
\begin{multline}
\abs{(f_0,v)_\Omega}  \leq \norm{f_0}_{L^{q_*/2}(\Omega)} \norm{v}_{L^{\frac{q_*}{q_*-2}}(\Omega)}
\\ \leq |E|_\dim^{1-\frac{2}{q_*}-\frac{1}{\sigma}}\norm{f_0}_{L^{q_*/2}(\Omega)}\norm{v}_{L^\sigma(\Omega)}
   \lesssim |E|_\dim^{\frac{1}{2}-\frac{1}{q_*}}\norm{f_0}_{L^{q_*/2}(\Omega)}\norm{v}_{H^1(\Omega)},
\end{multline}
where we have used above the identity $1-\frac{2}{q_*}-\frac{1}{\sigma}=\frac{1}{2}-\frac{1}{q_*}$ as a result of the choice of $\sigma$ above.
The rest of the proof then follows~\cite{ciarlet1973maximum} leading to the bound $\norm{m_\T}_{C(\overline{\Omega})}=\norm{m_\T}_{L^\infty(\Omega)}\lesssim \norm{f_0}_{L^{q_*/2}(\Omega)} +\norm{f_1}_{L^{q_*}(\Omega;\R^\dim)}$, which, in combination with~\eqref{eq:maxnorm_rhs_data_bounds}, yields~\eqref{eq:maxnorm_bound_example}.
\end{proof}

\subsection{Patchwise affine preservation}

The last assumption is that the stabilizations are \emph{patchwise affine-preserving} in the following sense.
\begin{enumerate}[label={(H\arabic*)},resume]
\item \label{H:stabilization_lin_preserving} \emph{Patchwise affine preservation.} The stabilizations are \emph{patchwise affine-preserving}: if $v_{\mathcal{T}}$ and $w_{\T}$ in $V(\T)$ are such that their restrictions to an interior vertex patch $\omega_{z}$, $z\in\calVi$, are affine functions over $\omega_{z}$, i.e. $v_{\mathcal{T}}|_{\omega_{z}}$ and $w_{\T}|_{\omega_{z}}$ are both in $\mathcal{P}_1\left(\omega_{z}\right)$, then
\begin{equation}
S_i(v_{\T},w_{\T};\psi_{z}) =0  \qquad \forall i\in \{1,2\}.
\end{equation}
\end{enumerate}
The concept of local affine preservation in~\ref{H:stabilization_lin_preserving} (sometimes also called local linearity preservation in the literature) is an important one in the design and analysis of stabilization schemes, leading to improved error bounds~\cite[]{barrenechea2024finite}.

\begin{example}\label{ex:affinepreservation}
A very simple example of a stabilization satisfying~\ref{H:stabilization_lin_preserving} is to consider $S_1$ and $S_2$ defined by~\eqref{eq:linear_stabilization} with $D_\T= \rho_\T \mathbb{I}_\dim$, where $\rho_\T>0$ is a constant, and where $\mathbb{I}_\dim$ denotes the $\dim \times \dim$ identity matrix.
Then, for any vertex $z\in \calVi$ and any $v_\T \in V(\T)$ that is affine on the vertex patch $\omz$, we have $S_i(v_\T,w_\T;\psi_z)=\rho_\T \int_{\omz} \nabla v_\T \cdot \nabla \psi_z\mathrm{d}x = - \rho_\T \int_{\omz} \Delta v_\T \psi_z=0$ since $\psi_z\in H^1_0(\omz)$ and $v_\T|_{\omz}\in H^2(\omz)$ with $\Delta v_\T|_{\omz} =0$. 
If the mesh is in addition strictly acute, then such a choice of stabilization is affine-preserving as well as enforcing a discrete maximum principle for all $\rho_\T$ sufficiently large, c.f.~\cite{osborne2024analysis}.
\end{example}

For more general meshes satisfying the Xu--Zikatanov condition, it is known that one can obtain both affine preservation and a DMP with \emph{linear} stabilizations provided that the mesh is symmetric with respect to its interior vertices.
Many more examples are provided by nonlinear stabilizations, which can be designed to satisfying~\ref{H:stabilization_lin_preserving} and a DMP on more general meshes, we refer the reader to \cite{barrenechea2024finite} for a comprehensive survey with many further references.

\begin{remark}[Patchwise and elementwise affine preservation]
In the context of \emph{a posteriori} error analysis, an \emph{elementwise} affine preservation property was considered in~\cite{AllendesBarrenecheaRankin2017}, which requires a decomposition of the stabilization over the mesh elements \cite[Eq.~(6)]{AllendesBarrenecheaRankin2017}, each of which is vanishing for functions that are locally affine over a certain neighbourhood~\cite[Assumption~3.2]{AllendesBarrenecheaRankin2017}.
Note that the elementwise affine preservation condition is not equivalent to the patchwise one considered here since Example~\ref{ex:affinepreservation} evidences a stabilization that satisfies~\ref{H:stabilization_lin_preserving} but not the elementwise condition from~\cite[Assumption~3.2]{AllendesBarrenecheaRankin2017}.
\end{remark}

\subsection{A posteriori error bounds}

One of the original contributions in this work is to show that in the case where \ref{H:stabilization_main}, \ref{H:stabilization_local_Lipschitz},~\ref{H:maxnorm_bound} and~\ref{H:stabilization_lin_preserving} all hold, then the stabilization estimators can be bounded in terms of the jump part of the elementwise estimators. Recall the definition of the jump terms $\jH$ and $\jF$ from~\eqref{eq:residual_estimators_jumps}.

\begin{theorem}[Bound on stabilization estimators]\label{thm:stab_jump_bound}
Assume~\ref{H:stabilization_main}, \ref{H:stabilization_local_Lipschitz}, \ref{H:maxnorm_bound},~and~\ref{H:stabilization_lin_preserving}.
Suppose that $(u_\T,m_\T)\in \left[V(\T)\right]^2$ is a solution of~\eqref{eq:num_scheme}.
Then
\begin{equation}\label{eq:stabilization_jump_bound}
\sum_{i=1}^2 \etastabi \lesssim \left(\sum_{i=1}^2\sum_{F\in\calFi} h_F\norm{\ji}_F^2 \right)^{\frac{1}{2}}.
\end{equation}
\end{theorem}

Theorem~\ref{thm:stab_jump_bound} implies that the total estimator $\eta(u_\T,m_\T)$ is then bounded the residual estimators. 
Indeed, recalling~\eqref{eq:total-estimator}, it is then immediate from~\eqref{eq:stabilization_jump_bound} that  
\begin{equation}\label{eq:totalestimator_bound_by_res}
\eta(u_\T,m_\T) \lesssim \sum_{i=1}^2\etaResi.
\end{equation}
This gives a bound for the error that requires only the residual estimator, i.e.\ in this case, it is not necessary to compute the stabilization estimator.

\begin{theorem}[Upper bound on the error for affine-preserving stabilizations]\label{thm:upper-bound-res}
Assume that $F\colon \mathcal{H}\rightarrow L^2(\Omega)$.
Assume \ref{H:stabilization_main}, \ref{H:stabilization_local_Lipschitz},~\ref{H:maxnorm_bound} and \ref{H:stabilization_lin_preserving}. 
Suppose that $(u_\T,m_\T)\in \left[V(\T)\right]^2$ is a solution of~\eqref{eq:num_scheme}.
There exists an $\widetilde{\mathcal{R}_*}>0$ such that if $\sum_{i=1}^2\etaResi\leq \widetilde{\mathcal{R}_*}$, then
\begin{equation}\label{eq:error_upper_bound_special}
\norm{u-u_\T}_{H^1(\Omega)}+\norm{m-m_\T}_{H^1(\Omega)}\lesssim \sum_{i=1}^2 \etaResi.
\end{equation}
\end{theorem}
\begin{proof}
This result follows directly from~\eqref{eq:totalestimator_bound_by_res} and from Theorem~\ref{thm:error_upper_bound} above.
\end{proof}

\begin{remark}
From a theoretical perspective, the interest of~\eqref{eq:error_upper_bound_special}  arises from the fact that in general, numerical schemes with stabilization fail to satisfy the Galerkin orthogonality property on the residual, and thus the dual norm on the residual is then generally not localizable.
Furthermore, the residual estimators $\etaResi$ have the additional advantage of being locally efficient, as shown in Theorem~\ref{theorem:eta-lower-bound} above. 
For practical computation, this also has the benefit of allowing one to avoid the explicit computation of the stabilization estimators. 
This highlights in particular the role of the patchwise affine preservation condition on the stabilizations in the \emph{a posteriori} error analysis of stabilized finite element methods.
\end{remark}

\subsection{Proof of Theorem~\ref{thm:stab_jump_bound}}

We shall use the following well-known approximation result for piecewise polynomial functions.
\begin{lemma}\label{lem:affine_approximation}
For each $v\in V(\T)$ and each vertex $z\in \calVi$, there holds
\begin{equation}\label{eq:affine_approximation}
\inf_{ Q \in \mathcal{P}_1(\omz) }\left\{ \sum_{k=0}^1 h_{
\omz}^{2k-2}\norm{\nabla^k (v-Q)}^2_{\omz}  \right\} \lesssim \sum_{F\in\calFiz} h_F \norm{\jump{\nabla v\cdot n_F}}_F^2,
\end{equation}
where $\mathcal{P}_1(\omz)$ denotes the space of affine functions on $\omz$.
The hidden constant in~\eqref{eq:affine_approximation} depends only on the dimension $\dim$ and the shape-regularity constant $\theta_\T$.
\end{lemma}
\begin{proof}
The proof is well-known so we include it here only for completeness.
Let $z\in\calVi$ be a fixed but arbitrary interior vertex.
In \cite[Theorem~3.1]{BuffaOrtner2009}, it is shown that any function $w\colon \omz\tends \R$ that is piecewise polynomial with respect to the elements in $\Tz$ satisfies the generalized Poincar\'e inequality
\begin{equation}\label{eq:piecewise_constant_approx}
\norm{w - \overline{w}}_{\omz}\lesssim h_{\omz}\left(\sum_{K\in \Tz} \norm{\nabla w}_{K}^2+ \sum_{F\in\calFiz} h_F^{-1}\norm{\jump{w}}_F^2 \right)^{\frac{1}{2}}, 
\end{equation}
where $\overline{w}\coloneqq \frac{1}{\abs{\omz}_{\dim}}\int_{\omz}w\dx$ denotes the mean-value of $w$ over $\omz$.

Given a $v\in V(\T)$, we then define $Q\in \mathcal{P}_1(\omz)$ by $Q(x)\coloneqq q_1{\cdot} x + q_0$ for all $x\in\omz$, where
\begin{equation}
q_1\coloneqq \frac{1}{\abs{\omz}_\dim}\int_{\omz}\nabla v\,\dx \in \R^\dim ,\quad q_0 \coloneqq \frac{1}{\abs{\omz}_\dim}\int_{\omz}(v-q_1{\cdot} x)\,\dx\in \R.
\end{equation}
We therefore deduce from~\eqref{eq:piecewise_constant_approx} applied to the components of $\nabla v$ that
\begin{equation}\label{eq:piecewise_affine_approx_1}
\norm{\nabla(v-Q)}_{\omz} = \norm{\nabla v - q_1}_{\omz} \lesssim \left(\sum_{F\in\calFiz} h_F \norm{\jump{\nabla v{\cdot} n_F}}_F^2\right)^{\frac{1}{2}},
\end{equation}
where we have used the bound $h_{\omz}\lesssim h_F$ for some constant depending only on the shape-regularity paramter, and where we have also used the facts that $\nabla v$ is piecewise constant and that only the normal component of $\nabla v$ is possibly discontinuous across mesh faces.
We then use~\eqref{eq:piecewise_constant_approx} again to conclude that
\begin{equation}\label{eq:piecewise_affine_approx_2}
\norm{v - Q}_{\omz} = \norm{v-q_1{\cdot} x - \overline{v-q_1{\cdot} x}}_{\omz} \lesssim h_{\omz} \norm{\nabla v - q_1}_{\omz},
\end{equation} 
where we have used the fact that $v-q_1{\cdot} x$ is continuous on $\omz$ so the jump-terms vanish.
We therefore obtain~\eqref{eq:affine_approximation} from \eqref{eq:piecewise_affine_approx_1} and \eqref{eq:piecewise_affine_approx_2}.
\end{proof}

\paragraph{Proof of Theorem~\ref{thm:stab_jump_bound}.}

For each function $w_\T\in V(\T)$ and each vertex $z\in\calVi$, let $\overline{w}_{\T}^z \in \mathcal{P}_1(\omz)$ denote the unique minimizer of the left-hand side of~\eqref{eq:affine_approximation}, where we note that a unique minimizer exists since $\mathcal{P}_1(\omz)$ is a finite dimensional space.
Now, consider a fixed but arbitrary $i\in\{1,2\}$ and a fixed but arbitrary $v_\T\in V(\T)$. Then, expanding $v_\T$ in the nodal basis of $V(\T)$, we may write $v_\T=\sum_{z\in\calVi}v_\T(z) \psi_z$.
Then, using the linearity of the stabilization in the last argument, we write
\begin{equation}\label{eq:jump_stab_bound_1}
\begin{aligned}
\abs{S_i(u_\T,m_\T;v_\T)} &= \left\lvert\sum_{z\in\calVi} v_\T(z) S_i(u_\T,m_\T;\psi_z)\right\rvert \\
&\underset{\ref{H:stabilization_lin_preserving}}{=} \left\lvert \sum_{z\in\calVi} v_\T(z) \left[S_i(u_\T,m_\T;\psi_z)-S_i(\overline{m}_T^z,\overline{u}_\T^z;\psi_z)\right]\right\rvert 
\\
&\underset{\ref{H:stabilization_local_Lipschitz}}{\lesssim} \sum_{z\in\calVi} \abs{v_\T(z)} \abs{\omz}_d^{\frac{1}{2}}\left[\norm{\nabla(m_\T-\overline{m}_\T^z)}_{\omz}+\norm{\nabla(u_\T-\overline{u}_{\T}^z)}_{\omz}\right]
\\& \lesssim \left(\sum_{z\in\calVi} \abs{v_\T(z)}^2\abs{\omz}\right)^{\frac{1}{2}}\left(\sum_{z\in\calVi}\left[\norm{\nabla(m_\T-\overline{m}_\T^z)}^2_{\omz}+\norm{\nabla(u_\T-\overline{u}_{\T}^z)}^2_{\omz}\right]\right)^{\frac{1}{2}}.
\end{aligned}
\end{equation}

It follows from inverse inequalities that $\sum_{z\in\calVi}\abs{v_\T(z)}^2\abs{\omz}_d \lesssim \norm{v_\T}_{\Omega}^2$ for all $v_\T\in V(\T)$.
The Poincar\'e inequality thus implies that
\begin{equation}\label{eq:jump_stab_bound_2}
\sup_{v_\T\in V(\T)\setminus\{0\}} \frac{S_i(u_\T,m_\T;v_\T)}{\norm{\nabla v_\T}_\Omega}\lesssim \left(\sum_{z\in\calVi}\left[\norm{\nabla(m_\T-\overline{m}_\T^z)}^2_{\omz}+\norm{\nabla(u_\T-\overline{u}_{\T}^z)}^2_{\omz}\right] \right)^{\frac{1}{2}}.
\end{equation} 
Using the definitions of $\overline{m}_\T^z$ and $\overline{u}_\T^z$ above, we now apply Lemma~\ref{lem:affine_approximation} to bound the right-hand side of~\eqref{eq:jump_stab_bound_2}, thus giving
\begin{multline}
\sum_{z\in\calVi}\left[\norm{\nabla(m_\T-\overline{m}_\T^z)}^2_{\omz}+\norm{\nabla(u_\T-\overline{u}_{\T}^z)}^2_{\omz}\right] 
\\ \lesssim \sum_{z\in\calVi}\sum_{F\in\calFiz} \left[h_F\norm{\jump{\nabla m_\T\cdot n_F}}_F^2+ h_F\norm{\jump{\nabla u_\T\cdot n_F}}_F^2\right]
\\ \lesssim \sum_{F\in\calFi} \left[h_F\norm{\jump{\nabla m_\T\cdot n_F}}_F^2+ h_F\norm{\jump{\nabla u_\T\cdot n_F}}_F^2\right],
\end{multline}
where, in passing to the final line above, we have used a counting argument based on the fact that each face $F\in\calFi$ is associated to at most $d$ vertices in $\mathcal{V}$.
Recalling the definition of the jumps in~\eqref{eq:residual_estimators_jumps}, we see that in order to obtain~\eqref{eq:stabilization_jump_bound} and thereby complete the proof, it remains only to show that, for each $F\in \calFi$, we have 
\begin{equation}\label{eq:jump_stab_5}
\norm{\jump{\nabla m_\T\cdot n_F}}_F \lesssim \norm{\jH}_F+\norm{\jF}_F.
\end{equation}
Indeed, applying the triangle inequality, we have
\begin{equation}\label{eq:jump_stab_bound_5}
\norm{\jump{\nabla m_\T\cdot n_F}}_F \lesssim\norm{\jF}_F+\norm{m_\T\jump{\frac{\partial H}{\partial p}[\nabla u_\T]\cdot n_F}}_F,
\end{equation}
with a hidden constant that depends on $\nu$.
We now use the maximum norm bound~\ref{H:maxnorm_bound} to obtain
\begin{equation}\label{eq:jump_stab_bound_4}
\norm{m_\T\jump{\frac{\partial H}{\partial p}[\nabla u_\T]\cdot n_F}}_F \leq
\widetilde{M}_\infty \norm{\jump{\frac{\partial H}{\partial p}[\nabla u_\T]\cdot n_F}}_F.
\end{equation}
Note that the Lipschitz continuity~\eqref{eq:Hp_bound} implies that, for each $F\in \calFi$, we have
\begin{equation}\label{eq:jump_stab_bound_3}
\begin{split}
\norm{\jump{\frac{\partial H}{\partial p}[\nabla u_\T]\cdot n_F} }_F 
\leq L_{H_p} \norm{\nabla u_\T|_K-\nabla u_\T|_{K^\prime} }_F
 = L_{H_p} \norm{\jump{\nabla u_\T\cdot n_F}}_F \lesssim \norm{\jH}_F,
\end{split}
\end{equation}
where $\{K,K^\prime\}\subset \T$ are the elements such that $F=\overline{K}\cap \overline{K^\prime}$. Note that in~\eqref{eq:jump_stab_bound_3} above, we have used the continuity across faces of the tangential component of $\nabla u_\T$, recalling that $u_\T\in V(\T)$ is continuous.
Therefore, we combine~\eqref{eq:jump_stab_bound_5}, \eqref{eq:jump_stab_bound_4} and \eqref{eq:jump_stab_bound_3} to deduce~\eqref{eq:jump_stab_5}, thereby completing the proof.
\hfill\qedsymbol

\section{Numerical Experiments}\label{sec:numerics}
In this section we present numerical experiments that show the computational performance of the \textit{a posteriori} error estimators of Sections \ref{sec:general-estimates} and \ref{sec:affine_preserving_stab}, and that demonstrate how the estimators lead to potential improvements in efficiency and accuracy in adaptive computations.
Our numerical experiments are performed within the open-source finite element software Firedrake \cite{FiredrakeUserManual}. The resulting system of equations are solved using the built in nonlinear solver of Firedrake, which uses a Newton line search algorithm. 
Mesh construction and refinement are handled using the NETGEN python bindings within Firedrake. 

\subsection{Numerical validation of theoretical error bounds}
The first experiment aims to validate the theoretical error bounds of Sections~\ref{sec:general-estimates} and~\ref{sec:affine_preserving_stab} above.

\paragraph{Experiment 1.}
The domain is defined as the unit square $\Omega = (0,1)^2$.
The PDE data for the MFG system \eqref{sys} is chosen as $\nu = 1/10$, and $H[\nabla u] = \sqrt{\abs{\nabla u}^2 +1}$.
The source term $G(x)$ and coupling term $F[m]$ are chosen \textit{a priori} such that 
\begin{subequations}\label{eq:Experiment-1-solution}
\begin{align}
    u(x,y) &= \left(50(x - y)^2-1 \right) \left( x (1 - x) \, y (1 - y) \right), \\
    m(x,y) &=  \exp\left( -50(x - y)^2 \right) \left( x (1 - x) \, y (1 - y) \right),
\end{align}
\end{subequations}
satisfies \eqref{sys} along with homogeneous Dirichlet boundary conditions on $\partial \Omega$.
To maintain the nonlinear coupling in the MFG system, the cost function is defined by $F[\tilde{m}] = \tilde{m} - m_0(x,y)$, where $m_0$ is the unique function ensuring that the solution is given by $(u,m)$ in \eqref{eq:Experiment-1-solution} above. 
Likewise, the datum~$G$ is obtained from~\eqref{eq:Experiment-1-solution} above.
The solution is smooth, although the density is concentrated along the diagonal of the domain.
We perform the computations on a sequence of uniform meshes with right-angle triangles, thus ensuring that each mesh in the sequence satisfies the Xu--Zikatanov condition.
We choose the stabilizations as defined in~\eqref{eq:linear_stabilization} and~\eqref{edge-tensor-formula} with the edge-weights chosen as $\gamma_{\T,E} = \diam E$.  
The stabilization estimator is approximated by $\widetilde{\eta}_{\mathrm{stab},i}$ as detailed in Remark~\ref{rem:discrete_computation}.
In Figure \ref{fig:MFG_experiment_1_estimators}, we compare the total $H^1$-norm of the error of the FEM solution, the total estimator $\eta(u_\T,m_\T)$, and the residual estimator.
We first observe, in agreement with the error bounds of \cite{osborne2024near}, that the $H^1$-norm of the error and the estimators converge to zero at a rate of $O(N^{-1/2})$, where $N$ denotes the number of degrees of freedom. Secondly, both the total estimator $\eta(u_\T,m_\T)$ and residual estimator bound the $H^1$ error up to constants. 
We conclude this experiment by noting that the observation of equivalence between the total and residual estimator is as expected from Theorem~\ref{thm:stab_jump_bound}.

\begin{figure}
    \centering
    \begin{tikzpicture}
    \begin{groupplot}[group style={
                      group name= myplot,
                      group size= 1 by 1,
                      horizontal sep=1.5cm,
                      vertical sep=1.5cm},
                      width=0.65\linewidth,
                      height=0.65\linewidth,
                      xmode=log,
                      ymode=log, xmax=1e6,
                      axis background/.style={fill=gray!0}, 
                      grid=both,
                      grid style={line width=.1pt, draw=gray!10},
                      major grid style={line width=.2pt,draw=gray!50}]
                        
        \nextgroupplot[xlabel={Degrees of freedom},legend pos=north east]
                \addplot+[mark=triangle, thick, black, mark options={black, solid}] table [x=num_dofs, y=H1_error, col sep=comma] {experiment_1_uniform.csv};
                \addplot+[mark=square, thick, blue, mark options={blue, solid}] table [x=num_dofs, y=eta, col sep=comma] 
                {experiment_1_uniform.csv};
                \addplot+[mark=diamond, thick, red, mark options={red, solid}] table [x=num_dofs, y=eta+stab, col sep=comma] 
                {experiment_1_uniform.csv};
                \legend{$H^1$ error, residual estimator, total estimator};

                \addplot[mark=none, solid, black] coordinates {(100000*0.3, 1.5/100) (100000*3, 1.5/100*10^-0.5)}; 
                \addplot[mark=none, solid, black] coordinates {(100000*3, 1.5/100*10^-0.5) (100000*0.3, 1.5/100 * 10^-0.5)}; 
                \addplot[mark=none, solid, black] coordinates {(100000*0.3, 1.5/100 * 10^-0.5) (100000*0.3, 1.5/100)}; 
                \addplot[mark=none] coordinates {(100000*0.3, 1.5/200+1.5/200*10^-0.5)} node[anchor=east] {$0.5$};        
    \end{groupplot}
\end{tikzpicture}
\caption{Experiment 1: Comparisons between the $H^1$-norm of the error, the total estimator $\eta(u_\T,m_\T)$, and residual estimator.} \label{fig:MFG_experiment_1_estimators}
\end{figure}
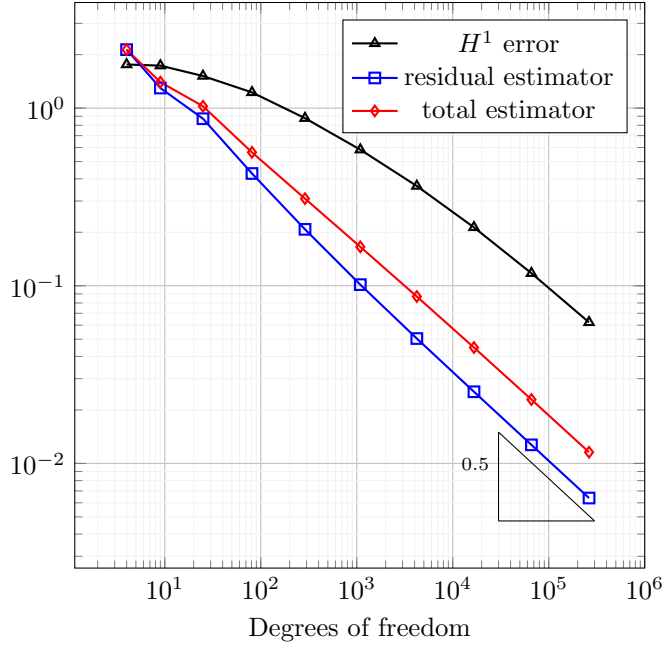

\subsection{Adaptive FEM}
Our remaining experiments focus on adaptive FEM for stationary MFG. 
In contrast to Experiment~1, we consider below problems for which the exact solution is not known analytically, further highlighting the benefits of \textit{a posteriori} error bounds for monitoring the error. 
As an extension of the theory, we consider models with mixed boundary conditions of Dirichlet and Neumann types, which are furthermore nonlinear for the case of the density function. 
In applications, Neumann boundary conditions model situations where there is a prescribed flux of players entering or exiting the game on parts of the boundary.
Note that handling mixed boundary conditions requires a small adjustment to the estimators, as we now detail.
Suppose that $\partial \Omega$ is partitioned as $\partial \Omega = \GammaD \cup \GammaN$ into a Dirichlet part~$ \GammaD$ and a Neumann part~$\GammaN$.
The mesh $\T$ is additionally required to satisfy the condition that every face on the boundary is wholly contained in either the Dirichlet or Neumann part of the boundary. 
In the examples below, we consider Neumann conditions of the form $\nu \nabla u {\cdot} n = g_2$ and $\nu \nabla m {\cdot} n  + m \frac{\partial H}{\partial p}[\nabla u]{\cdot} n =g_3$ on $\GammaN$, where $n$ denotes the unit outward normal vector to $\partial\Omega$, and $g_2$, $g_3  \in L^2(\GammaN)$ are prescribed functions. 
Notice that the Neumann condition for the density is nonlinear and depends also on $u$.
The estimators are then adjusted as follows.
For each $K\in\T$, let $\calFinK$ denote the set of faces of $K$ that are either in $\calFiK$ or subsets of $\GammaN$.
We then extend the jump terms~\eqref{eq:residual_estimators_jumps} to Neumann boundary faces by defining $j_{F,1} \coloneqq \nu \nabla u_\mathcal{T} {\cdot} n_F - g_2$, and $j_{F,2} \coloneqq \nu \nabla m_\T {\cdot} n_F + m_\T \frac{\partial H}{\partial p}[\nabla u_\mathcal{T}] {\cdot} n_F - g_3$ for each face $F\subset \GammaN$, where $n_F$ denotes the unit outward normal to $\Omega$ on $F$.
We then replace~\eqref{eq:element_estimators} by $[\etaKi]^2\coloneqq h_K^2 \norm{\rKi}^2_K+ \sum_{F\in \calFinK}h_F\norm{\ji}_F^2$, i.e.\ we extend the face terms to include the contributions from faces in $\GammaN$.

In the following experiments, the adaptive mesh refinement is performed using the D\"{o}rfler/bulk-chasing marking strategy for the residual estimators, i.e.\ in each iteration of the adaptive loop, we refine the subset of elements $\mathcal{M} \subset \T$ with minimal cardinality that satisfies
\begin{align}
    \sum_{i=1}^2 \left( \sum_{K \in \mathcal{M}} \eta_{K,i}^2 \right)^{\frac{1}{2}} \geq \theta \sum_{i=1}^2 \etaResi, \label{eq:marking-strategy}
\end{align}
where $\theta \in (0,1)$ is a chosen parameter. In the computations below we use $\theta = \frac{3}{10}$.
Marked elements are refined via the newest vertex bisection algorithm.
Regarding the choice of stabilization, we use the stabilizations defined in~\eqref{eq:linear_stabilization} and~\eqref{edge-tensor-formula} with the local edge-weights chosen as $\gamma_{\T,E} = \diam E$, where the sum in~\eqref{edge-tensor-formula} is additionally extended to include boundary edges that are contained in the Neumann part of the boundary.

\paragraph{Experiment 2.}
We consider the MFG system on a nonconvex L-shaped domain defined by $\Omega = (-1,1)^2 \setminus [0,1]^2$. 
The PDE data for this test is chosen to be $\nu =1$, $H[\nabla u] = \sqrt{\abs{\nabla u}^2 +1}$, $F[m]=m$, and $G=0$.
Let $\GammaD \coloneqq \{ (x,y) \in \partial \Omega\ :\ x=0  \text{  or  } y=0 \}$ denote the Dirichlet part of the boundary, on which we impose the conditions $u(x,y) = \abs{x}+\abs{y}-1$ and $m(x,y) = 0$ for all $(x,y)\in\GammaD$.
The Neumann conditions on $\GammaN = \partial \Omega \setminus \GammaD$ are given by $\nu \nabla u \cdot n =0$ on $\GammaN$, and $  \nu \nabla m \cdot n + m \frac{\partial H}{\partial p}[\nabla u] \cdot n  = \chi_{\GammaIn}$ on $\GammaN$, where $\chi_{\GammaIn}$ denotes the indicator function of $\GammaIn=\{ (x,y) \in \partial \Omega\ :\ x=-1 \text{  or  } y=-1 \}$. Observe that the Neumann conditions are thus nonlinear.
Such boundary conditions represent a situation where players can exit the game along the edges forming the re-entrant corner, with the minimal exit cost attained on the corner exactly. 
The Neumann condition prevents players from exiting the game on $\GammaN\setminus \GammaIn$ while the positive flux condition on $\GammaIn$ represents a constant inflow of players into the game. 
Note that despite taking $G=0$, the density is nonvanishing thanks to the inhomogenous boundary condition on $\GammaN$.
The initial computational mesh is composed of right-angled triangular elements, therefore each mesh produced by the newest vertex bisection algorithm will satisfy the Xu--Zikatanov condition.
An example of one of the adaptively refined meshes is depicted in Figure~\ref{fig:MFG_experiment_2_meshes}, in which we see that the mesh refinements are concentrated around the re-entrant corner.
In Figure~\ref{fig:MFG_experiment_2_estimators}, we directly compare the performance between adaptive and uniform mesh refinements: whereas adaptive refinement maintains the $O(N^{-1/2})$ rate of convergence with respect to the number of degrees of freedom, uniform mesh refinement converges at a lower rate of $O(N^{-1/3})$, where $N$ denotes the number of degrees of freedom.

\begin{figure}
    \centering
    \subfloat[$0$ adaptive mesh refinements]{
        \includegraphics[width=0.5\linewidth,trim=300 50 300 50, clip]{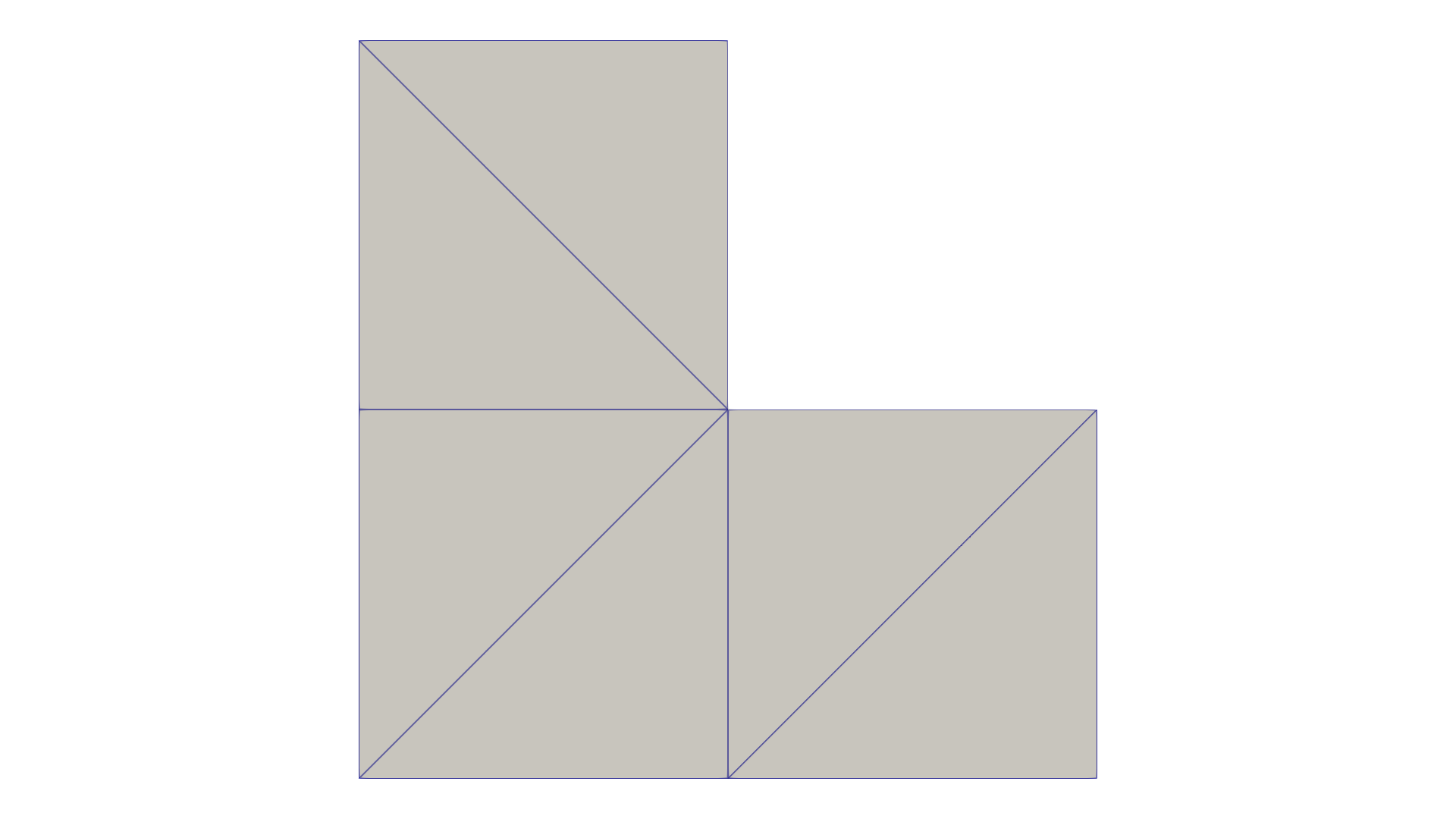}}
    \subfloat[$10$ adaptive mesh refinements]{
        \includegraphics[width=0.5\textwidth,trim=300 50 300 50, clip]{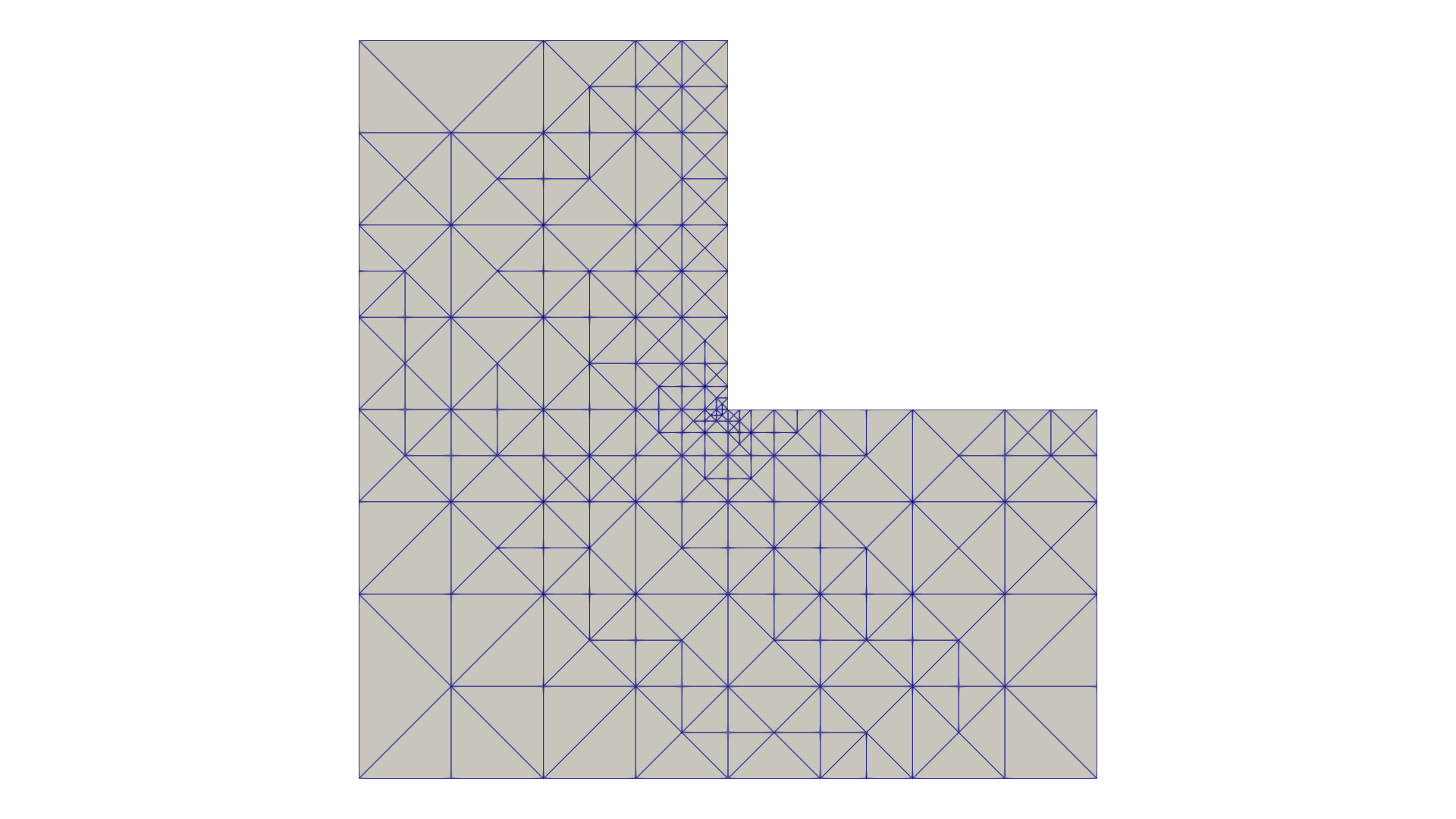}}
        \\
    \caption{Experiment 2: A sample of the adaptively refined meshes.}
    \label{fig:MFG_experiment_2_meshes}
\end{figure}

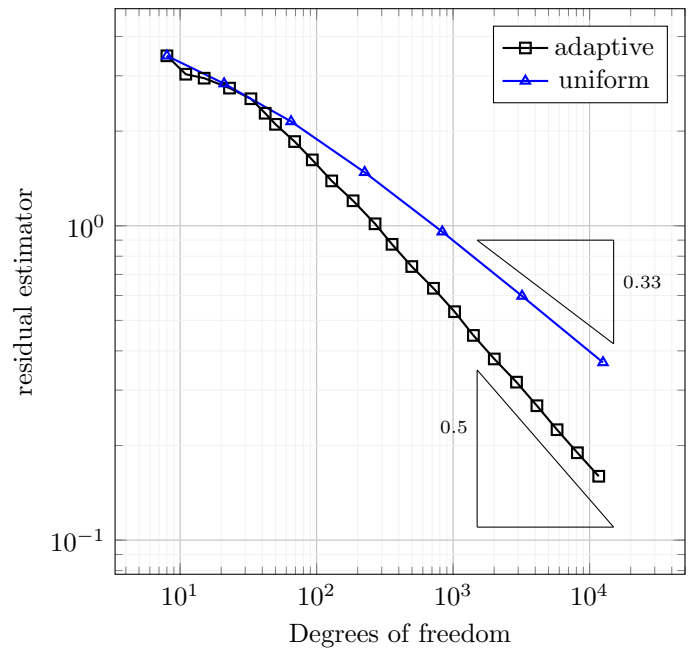
\begin{figure}
    \centering
    \begin{tikzpicture}
    \begin{groupplot}[group style={
                      group name= myplot,
                      group size= 2 by 1,
                      horizontal sep=1.5cm,
                      vertical sep=1.5cm},
                      width=0.65\linewidth,
                      height=0.65\linewidth,
                      xmode=log, xmax=5e4,
                      ymode=log,
                      axis background/.style={fill=gray!0}, 
                      legend pos=south east,
                      grid=both,
                      grid style={line width=.1pt, draw=gray!10},
                      major grid style={line width=.2pt,draw=gray!50}]
                        
        \nextgroupplot[xlabel={Degrees of freedom},ylabel={residual estimator},legend pos=north east]
                \addplot+[mark=square, thick, black, mark options={black, solid}] table [x=num_dofs, y=eta, col sep=comma] {experiment_3_adaptive.csv};
                \addplot+[mark=triangle, thick, blue, mark options={blue, solid}] table [x=num_dofs, y=eta, col sep=comma] {experiment_3_uniform.csv};
                \legend{adaptive, uniform};

\addplot[mark=none, solid, black] coordinates {(1500, 1.1*1e-1) (15000, 1.1*1e-1)}; 
\addplot[mark=none, solid, black] coordinates {(15000, 1.1*1e-1) (1500, 1.1*1e-1 * 10^0.5)}; 
\addplot[mark=none, solid, black] coordinates {(1500, 1.1*1e-1 * 10^0.5) (1500, 1.1*1e-1)}; 
\addplot[mark=none] coordinates {(1500, 0.5*1.1*1e-1+0.5*1.1*1e-1*10^0.5)} node[anchor=east] {$0.5$};

\addplot[mark=none, solid, black] coordinates {(15000, 1.5*0.6) (1500, 1.5*0.6)}; 
\addplot[mark=none, solid, black] coordinates {(15000, 1.5*0.6) (15000, 1.5*0.6 * 10^-0.33)}; 
\addplot[mark=none, solid, black] coordinates {(15000, 1.5*0.6 * 10^-0.33) (1500, 1.5*0.6)}; 
\addplot[mark=none] coordinates {(15000, 1.5*0.3+1.5*0.3*10^-0.33)} node[anchor=west] {$0.33$};
         \end{groupplot}
\end{tikzpicture}
\caption{Experiment 2: Comparison between adaptive and uniform mesh refinement.} \label{fig:MFG_experiment_2_estimators}
\end{figure}

\paragraph{Experiment 3.}
For our final experiment, we demonstrate the benefits of adaptively refined meshes for MFG with complex geometries and large solution gradients.
The domain $\Omega$ is chosen to be a subset of $(0,1)^2$ that is constructed by removing shapes of various sizes and smoothness from the interior, see Figure~\ref{fig:MFG_experiment_3_meshes} for an illustration. 
The PDE data for this experiment is chosen to be $\nu = 1/4$, $H[\nabla u] = \sqrt{\abs{\nabla u}^2 + 1}$,  $F[m] = m^3$, $G(x,y) = 1-x$.
Note that $F$ is nonlinear and does not satisfy the Lipschitz condition~\eqref{F2} in this example.
We impose homogenous Dirichlet boundary conditions for both $u$ and $m$ on $\GammaD = \{(x,y) \in \partial \Omega\ :\ x = 1 \}$. 
Similar to Experiment 2, the Neumann conditions on $\GammaN = \partial \Omega \setminus \GammaD$ are given by $\nu \nabla u \cdot n  =0$ on $\GammaN$, and $ \nu \nabla m \cdot n + m \frac{\partial H}{\partial p}[\nabla u] \cdot n  = \chi_{\GammaIn}$ on $\GammaN$, where $\chi_{\GammaIn}$ is the indicator function for the set $\GammaIn=\{ (x,y) \in \partial \Omega\ :\ x=0 \}$. As before, the Neumann condition for the KFP equation is nonlinear.
Computations are performed on sequences of adaptive and uniformly refined meshes from an initial non-uniform triangulation $\T$ of $\Omega$. 
The resulting adaptively refined meshes are presented in Figure~\ref{fig:MFG_experiment_3_meshes}, along with the solution in Figure~\ref{fig:MFG_experiment_3_snapshot}. 
These demonstrate that the adaptive algorithm is successfully concentrating mesh refinements around regions of high player congestion, and also at some of the re-entrant corners.
This is further supported by the approximated flux term $b_\mathcal{T} \coloneqq -\nu \nabla m_\mathcal{T} - m_\mathcal{T} \frac{\partial H}{\partial p}[\nabla u_\mathcal{T}]$, which we present in Figure~\ref{fig:MFG_experiment_3_drift}.
\begin{figure}
    \centering
    \subfloat[$0$ adaptive mesh refinements]{
        \includegraphics[width=0.5\linewidth,trim=300 50 300 50, clip]{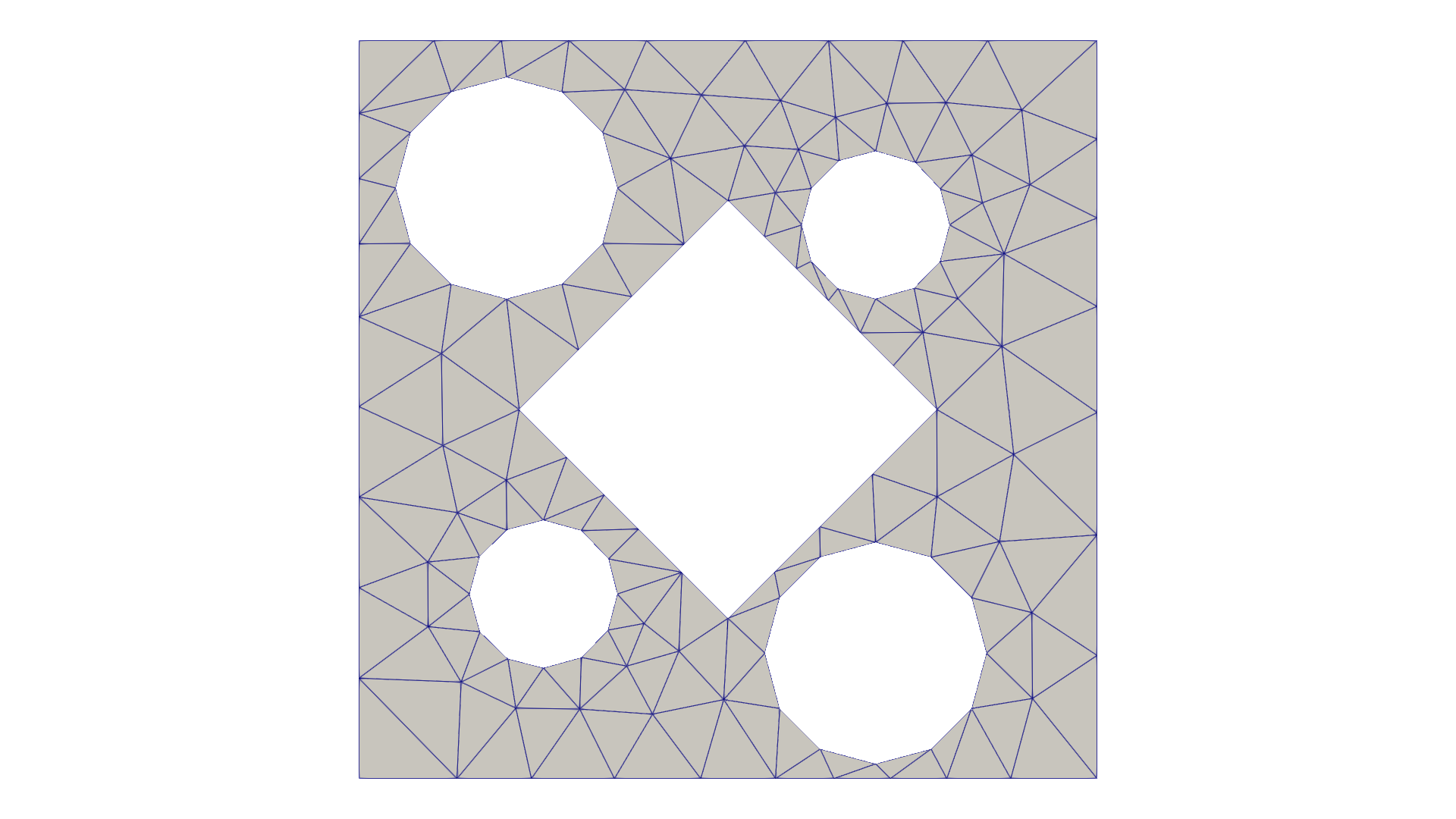}}
    \subfloat[$10$ adaptive mesh refinements]{
        \includegraphics[width=0.5\textwidth,trim=300 50 300 50, clip]{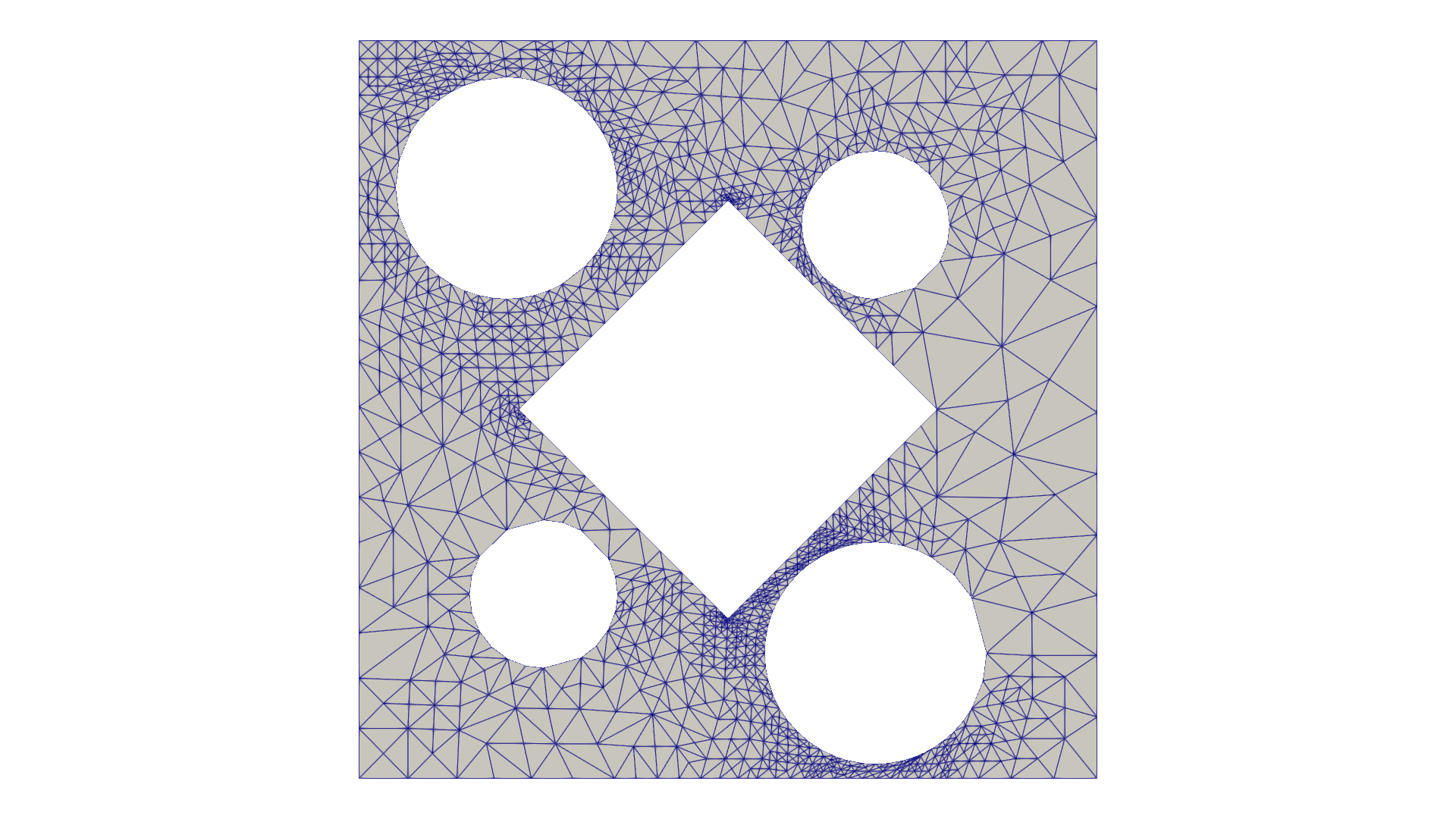}}
        \\
    \caption{Experiment 3: A sample of the adaptively refined mesh.}
    \label{fig:MFG_experiment_3_meshes}
\end{figure}
\begin{figure}
    \centering
    \subfloat[value function $u_\mathcal{T}$]{\includegraphics[width=0.5\linewidth,trim=400 0 0 0, clip]{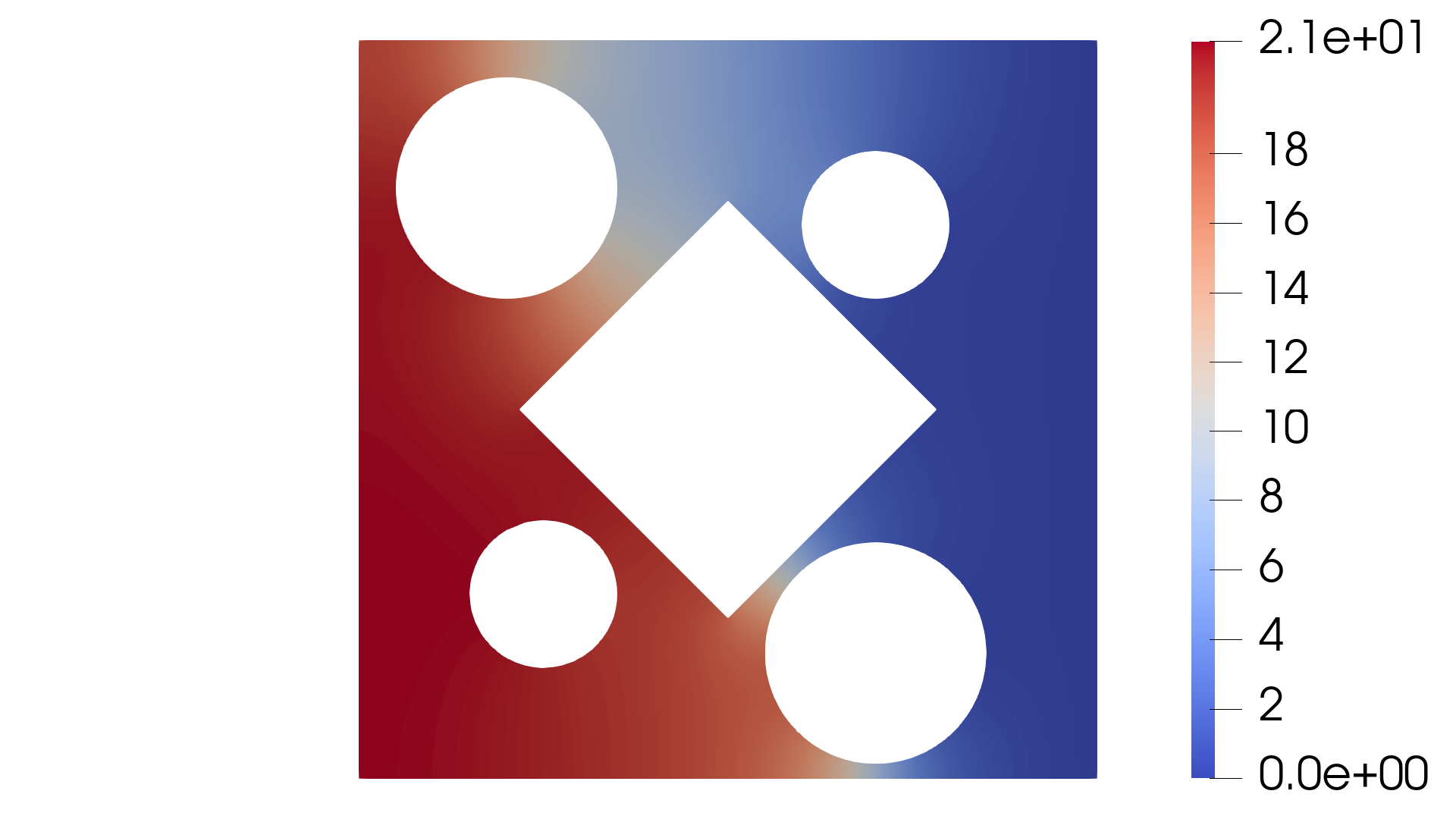}}~\hspace{0.1cm}~
    \subfloat[density function $m_\mathcal{T}$]{\includegraphics[width=0.5\linewidth,trim=400 0 0 0, clip]{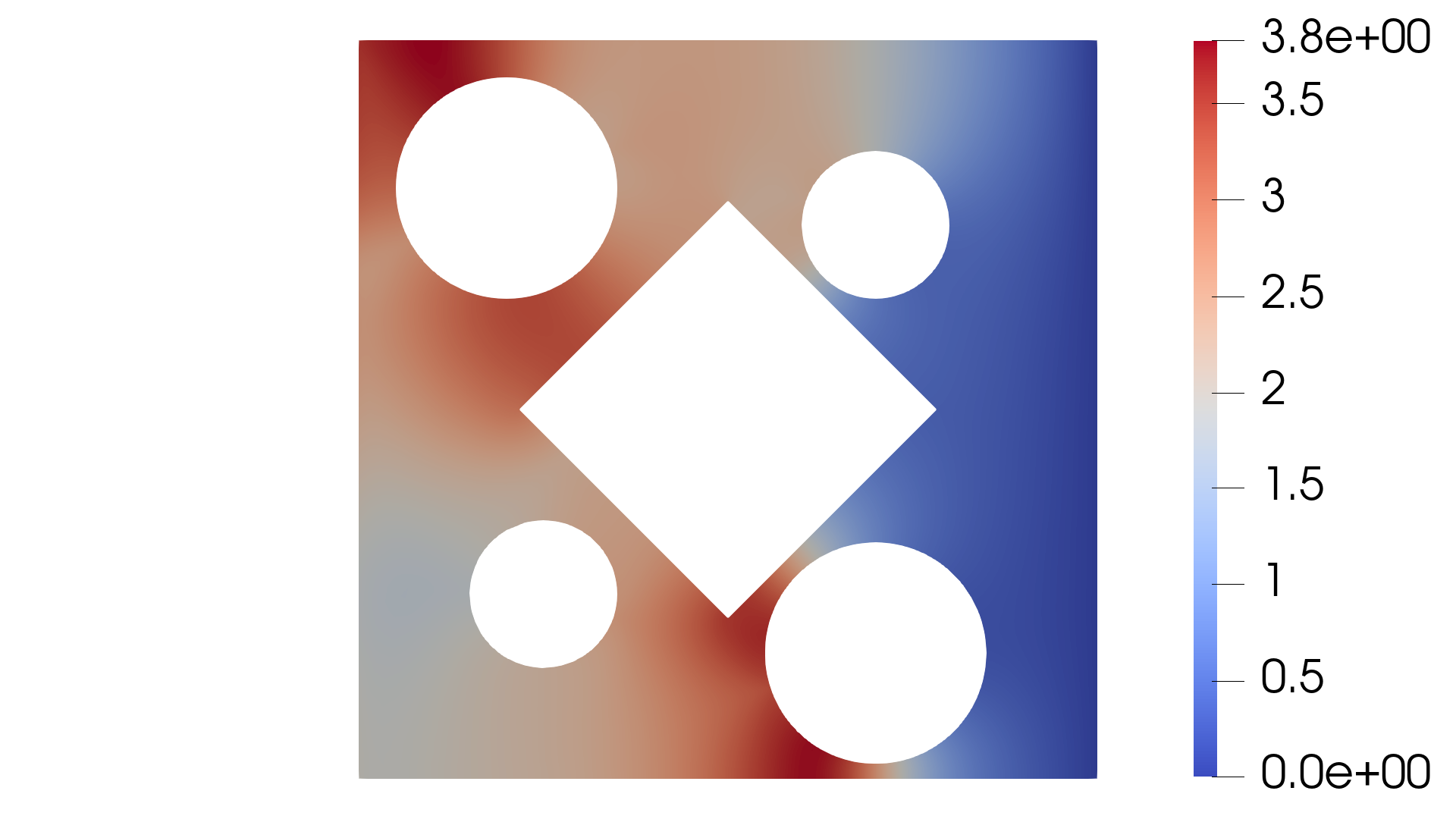}}
 \caption{Experiment 3: Numerical solutions computed on the adaptively refined mesh with 20 refinements.}
    \label{fig:MFG_experiment_3_snapshot}
\end{figure}
\begin{figure}
    \centering
    \subfloat{\includegraphics[width=0.5\linewidth,trim=400 00 0 0, clip]{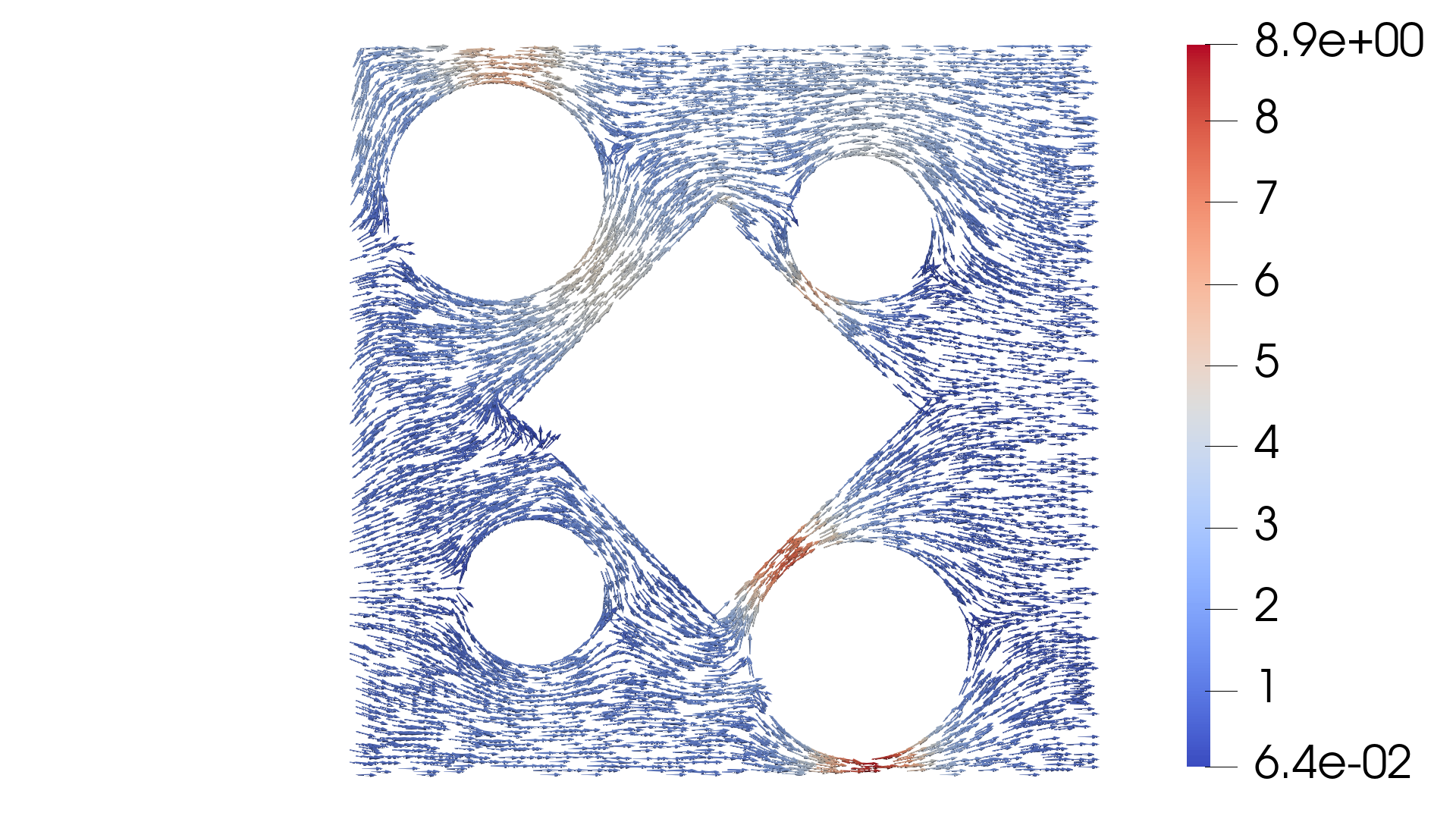}}
    \caption{Experiment 3: Quiver plot of the flux term $b_\mathcal{T} \coloneqq -\nu \nabla m_\mathcal{T} - m_\mathcal{T} \frac{\partial H}{\partial p}[\nabla u_\mathcal{T}]$ computed on the adaptively refined mesh with 20 refinements.}
    \label{fig:MFG_experiment_3_drift}
\end{figure}
In Figure~\ref{fig:MFG_experiment_3_comparison}, we compare the convergence of the residual estimator between the adaptive and uniform mesh refinement algorithms. 
We observe that uniform mesh refinement attains a lower rate of convergence of $O(N^{-r})$ for $r \approx \frac{4}{10}$ whereas adaptive mesh refinement maintains the $O(N^{-1/2})$ rate of convergence, where $N$ denotes the number of degrees of freedom. 
Although this example falls outside the scope of the analysis of Section~\ref{sec:affine_preserving_stab}, we see from Figure~\ref{fig:MFG_experiment_3_stab_jump} that the stabilization estimator, approximated by $\widetilde{\eta}_{\mathrm{stab},i}$ from Remark~\ref{rem:discrete_computation}, remains bounded by the jump component of the residual estimator for both adaptive and uniform mesh refinement.

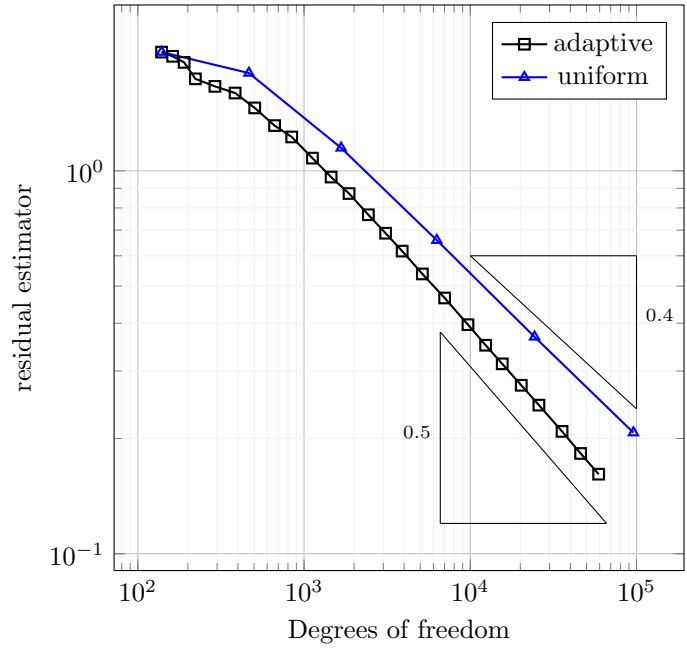
\begin{figure}
    \centering
    \begin{tikzpicture}
    \begin{groupplot}[group style={
                      group name= myplot,
                      horizontal sep=1.5cm,
                      vertical sep=1.5cm},width=0.65\linewidth,
                      height=0.65\linewidth,
                      xmode=log,
                      ymode=log,
                      axis background/.style={fill=gray!0}, 
                      legend pos=south east,
                      grid=both,
                      grid style={line width=.1pt, draw=gray!10},
                      major grid style={line width=.2pt,draw=gray!50}]
                        
        \nextgroupplot[xlabel={Degrees of freedom},ylabel={residual estimator},legend pos=north east]
                \addplot+[mark=square, thick, black, mark options={black, solid}] table [x=num_dofs, y=eta, col sep=comma] {experiment_2.csv};
                \addplot+[mark=triangle, thick, blue, mark options={blue, solid}] table [x=num_dofs, y=eta, col sep=comma] {experiment_2_uniform.csv};
                \legend{adaptive, uniform};

\addplot[mark=none, solid, black] coordinates {(100000/10, 1.5*0.4) (100000, 1.5*0.4)}; 
\addplot[mark=none, solid, black] coordinates {(100000, 1.5*0.4) (100000, 1.5*0.4 * 10^-0.4)}; 
\addplot[mark=none, solid, black] coordinates {(100000, 1.5*0.4 * 10^-0.4) (100000/10, 1.5*0.4)}; 
\addplot[mark=none] coordinates {(100000, 1.5*0.2+1.5*0.2*10^-0.4)} node[anchor=west] {$0.4$};

\addplot[mark=none, solid, black] coordinates {(66000, 0.12) (66000/10, 0.12)}; 
\addplot[mark=none, solid, black] coordinates {(66000/10, 0.12) (66000/10, 0.12 * 10^0.5)}; 
\addplot[mark=none, solid, black] coordinates {(66000/10, 0.12 * 10^0.5) (66000, 0.12)}; 
\addplot[mark=none] coordinates {(66000/10, 0.05*10^0.5+0.05)} node[anchor=east] {$0.5$};
    \end{groupplot}
\end{tikzpicture}
\caption{Experiment 3: Comparison between adaptive and uniform mesh refinements.} \label{fig:MFG_experiment_3_comparison}
\end{figure}

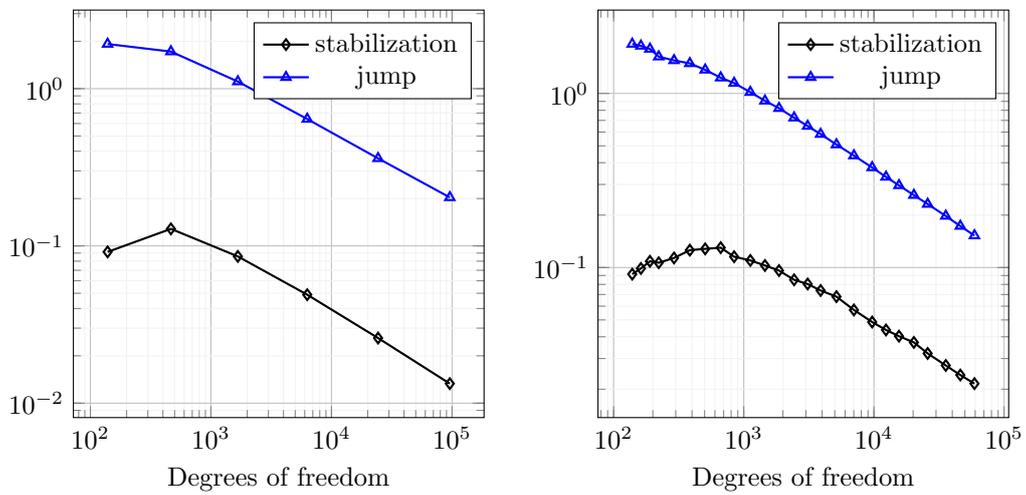
\begin{figure}
    \centering
    \begin{tikzpicture}
    \begin{groupplot}[group style={
                      group name= myplot,
                      group size= 2 by 2,
                      horizontal sep=1.5cm,
                      vertical sep=1.5cm},
                      width=0.5\linewidth,
                      height=0.5\linewidth,
                      xmode=log,
                      ymode=log,
                      axis background/.style={fill=gray!0}, 
                      legend pos=south east,
                      grid=both,
                      grid style={line width=.1pt, draw=gray!10},
                      major grid style={line width=.2pt,draw=gray!50}]
                        
        \nextgroupplot[xlabel={Degrees of freedom},legend pos=north east]
                \addplot+[mark=diamond, thick, black, mark options={black, solid}] table [x=num_dofs, y=stab, col sep=comma] {experiment_2_uniform.csv};
                \addplot+[mark=triangle, thick, blue, mark options={blue, solid}] table [x=num_dofs, y=jump, col sep=comma] {experiment_2_uniform.csv};
                \legend{stabilization, jump};
                
        \nextgroupplot[xlabel={Degrees of freedom},legend pos=north east]
                \addplot+[mark=diamond, thick, black, mark options={black, solid}] table [x=num_dofs, y=stab, col sep=comma] {experiment_2.csv};
                \addplot+[mark=triangle, thick, blue, mark options={blue, solid}] table [x=num_dofs, y=jump, col sep=comma] {experiment_2.csv};
                \legend{stabilization, jump}
    \end{groupplot}
\end{tikzpicture}
\caption{Experiment 3: Comparison between the stabilization and jump components of the residual estimator for the sequence of uniform (left) and adaptive (right) mesh refinement algorithms.} \label{fig:MFG_experiment_3_stab_jump}
\end{figure}

\section{Conclusion}\label{sec:conclusion}
This paper presented computable a posteriori error bounds for steady state MFG. We proved the reliability and efficiency of residual a posteriori error estimators for general stabilized FEM. For linearity preserving stabilizations, we further showed that the stabilization estimator is bounded by the jump terms of the residual estimator, resulting in a fully localizable estimator. The theoretical results were further validated through numerical experiments. Additionally, experiments with adaptive FEM demonstrate that the a posteriori error estimators yield more computationally efficient approximations compared to standard refinement techniques.

\section*{Acknowledgements}
Yohance A.~P.~Osborne was supported by The Royal Society Career Development Fellowship.
Iain~Smears and Harry~Wells were supported by the Engineering and Physical Sciences Research Council [grant number EP/Y008758/1].

\ifJ
\else

\fi


\begin{thebibliography}{10}

\bibitem{achdou2013mean}
{\sc Y.~Achdou, F.~Camilli, and I.~Capuzzo-Dolcetta}, {\em Mean field games:
  convergence of a finite difference method}, SIAM J. Numer. Anal., 51 (2013),
  pp.~2585--2612, \url{https://doi.org/10.1137/120882421}.

\bibitem{achdou2010mean}
{\sc Y.~Achdou and I.~Capuzzo-Dolcetta}, {\em Mean field games: numerical
  methods}, SIAM J. Numer. Anal., 48 (2010), pp.~1136--1162,
  \url{https://doi.org/10.1137/090758477}.

\bibitem{achdou2021mean}
{\sc Y.~Achdou, P.~Cardaliaguet, F.~Delarue, A.~Porretta, and F.~Santambrogio},
  {\em Mean Field Games: Cetraro, Italy 2019}, vol.~2281, Springer Nature,
  2021, \url{https://doi.org/10.1007/978-3-030-59837-2}.

\bibitem{AllendesBarrenecheaRankin2017}
{\sc A.~Allendes, G.~R. Barrenechea, and R.~Rankin}, {\em Fully computable
  error estimation of a nonlinear, positivity-preserving discretization of the
  convection-diffusion-reaction equation}, SIAM J. Sci. Comput., 39 (2017),
  pp.~A1903--A1927, \url{https://doi.org/10.1137/16M1092763}.

\bibitem{baba1981conservation}
{\sc K.~Baba and M.~Tabata}, {\em On a conservative upwind finite element
  scheme for convective diffusion equations}, ESAIM: Mathematical Modelling and
  Numerical Analysis - Mod\'elisation Math\'ematique et Analyse Num\'erique, 15
  (1981), pp.~3--25.

\bibitem{barrenechea2017edge}
{\sc G.~R. Barrenechea, E.~Burman, and F.~Karakatsani}, {\em Edge-based
  nonlinear diffusion for finite element approximations of convection-diffusion
  equations and its relation to algebraic flux-correction schemes}, Numer.
  Math., 135 (2017), pp.~521--545,
  \url{https://doi.org/10.1007/s00211-016-0808-z}.

\bibitem{barrenechea2024finite}
{\sc G.~R. Barrenechea, V.~John, and P.~Knobloch}, {\em Finite element methods
  respecting the discrete maximum principle for convection-diffusion
  equations}, SIAM Rev., 66 (2024), pp.~3--88,
  \url{https://doi.org/10.1137/22M1488934}.

\bibitem{barrenechea2018unified}
{\sc G.~R. Barrenechea, V.~John, P.~Knobloch, and R.~Rankin}, {\em A unified
  analysis of algebraic flux correction schemes for convection-diffusion
  equations}, SeMA J., 75 (2018), pp.~655--685,
  \url{https://doi.org/10.1007/s40324-018-0160-6}.

\bibitem{berry2024approximation}
{\sc J.~Berry, O.~Ley, and F.~J. Silva}, {\em Approximation and perturbations
  of stable solutions to a stationary mean field game system}, J. Math. Pures Appl., 194 (2025), p.~103666,
  \url{https://doi.org/https://doi.org/10.1016/j.matpur.2025.103666}.

\bibitem{BuffaOrtner2009}
{\sc A.~Buffa and C.~Ortner}, {\em Compact embeddings of broken {S}obolev
  spaces and applications}, IMA J. Numer. Anal., 29 (2009), pp.~827--855,
  \url{https://doi.org/10.1093/imanum/drn038}.

\bibitem{burman2002nonlinear}
{\sc E.~Burman and A.~Ern}, {\em Nonlinear diffusion and discrete maximum
  principle for stabilized {G}alerkin approximations of the
  convection--diffusion-reaction equation}, Comput. Methods Appl. Mech. Engrg., 191 (2002), pp.~3833--3855,
  \url{https://doi.org/10.1016/S0045-7825(02)00318-3}.

\bibitem{CarliniSilva14}
{\sc E.~Carlini and F.~J. Silva}, {\em A fully discrete semi-{L}agrangian
  scheme for a first order mean field game problem}, SIAM J. Numer. Anal., 52
  (2014), pp.~45--67, \url{https://doi.org/10.1137/120902987}.

\bibitem{CarliniSilve15}
{\sc E.~Carlini and F.~J. Silva}, {\em A semi-{L}agrangian scheme for a
  degenerate second order mean field game system}, Discrete Contin. Dyn. Syst.,
  35 (2015), pp.~4269--4292, \url{https://doi.org/10.3934/dcds.2015.35.4269}.

\bibitem{ciarlet1973maximum}
{\sc P.~G. Ciarlet and P.-A. Raviart}, {\em Maximum principle and uniform
  convergence for the finite element method}, Comput. Methods Appl. Mech.
  Engrg., 2 (1973), pp.~17--31,
  \url{https://doi.org/10.1016/0045-7825(73)90019-4}.

\bibitem{ElmanSilvesterWathen2014}
{\sc H.~C. Elman, D.~J. Silvester, and A.~J. Wathen}, {\em Finite elements and
  fast iterative solvers: with applications in incompressible fluid dynamics},
  Numerical Mathematics and Scientific Computation, Oxford University Press,
  Oxford, second~ed., 2014,
  \url{https://doi.org/10.1093/acprof:oso/9780199678792.001.0001}.

\bibitem{ErnSmearsVohralik2017}
{\sc A.~Ern, I.~Smears, and M.~Vohral\'ik}, {\em Discrete {$p$}-robust {$H({\rm
  div})$}-liftings and a posteriori estimates for elliptic problems with
  {$H^{-1}$} source terms}, Calcolo, 54 (2017), pp.~1009--1025,
  \url{https://doi.org/10.1007/s10092-017-0217-4}.

\bibitem{ErnVohralik2013}
{\sc A.~Ern and M.~Vohral\'ik}, {\em Adaptive inexact {N}ewton methods with a
  posteriori stopping criteria for nonlinear diffusion {PDE}s}, SIAM J. Sci.
  Comput., 35 (2013), pp.~A1761--A1791,
  \url{https://doi.org/10.1137/120896918}.

\bibitem{gilbarg2015elliptic}
{\sc D.~Gilbarg and N.~S. Trudinger}, {\em Elliptic partial differential
  equations of second order}, Classics in Mathematics, Springer-Verlag, Berlin,
  2001.
\newblock Reprint of the 1998 edition.

\bibitem{GomesSaude2014}
{\sc D.~A. Gomes and J.~a. Sa\'{u}de}, {\em Mean field games models---a brief
  survey}, Dyn. Games Appl., 4 (2014), pp.~110--154,
  \url{https://doi.org/10.1007/s13235-013-0099-2}.

\bibitem{GueantLasryLions2003}
{\sc O.~Gu\'{e}ant, J.-M. Lasry, and P.-L. Lions}, {\em Mean field games and
  applications}, in Paris-{P}rinceton {L}ectures on {M}athematical {F}inance
  2010, vol.~2003 of Lecture Notes in Math., Springer, Berlin, 2011,
  pp.~205--266, \url{https://doi.org/10.1007/978-3-642-14660-2\_3}.

\bibitem{FiredrakeUserManual}
{\sc D.~A. Ham, P.~H.~J. Kelly, L.~Mitchell, C.~J. Cotter, R.~C. Kirby,
  K.~Sagiyama, N.~Bouziani, S.~Vorderwuelbecke, T.~J. Gregory, J.~Betteridge,
  D.~R. Shapero, R.~W. Nixon-Hill, C.~J. Ward, P.~E. Farrell, P.~D. Brubeck,
  I.~Marsden, T.~H. Gibson, M.~Homolya, T.~Sun, A.~T.~T. McRae, F.~Luporini,
  A.~Gregory, M.~Lange, S.~W. Funke, F.~Rathgeber, G.-T. Bercea, and G.~R.
  Markall}, {\em Firedrake User Manual}, Imperial College London and University
  of Oxford and Baylor University and University of Washington, first
  edition~ed., 5 2023, \url{https://doi.org/10.25561/104839}.

\bibitem{huang2006large}
{\sc M.~Huang, R.~P. Malham\'{e}, and P.~E. Caines}, {\em Large population
  stochastic dynamic games: closed-loop {M}c{K}ean-{V}lasov systems and the
  {N}ash certainty equivalence principle}, Commun. Inf. Syst., 6 (2006),
  pp.~221--251, \url{https://doi.org/10.4310/cis.2006.v6.n3.a5}.

\bibitem{KreuzerVeeser2021}
{\sc C.~Kreuzer and A.~Veeser}, {\em Oscillation in a posteriori error
  estimation}, Numer. Math., 148 (2021), pp.~43--78,
  \url{https://doi.org/10.1007/s00211-021-01194-8}.

\bibitem{lasry2006jeuxI}
{\sc J.-M. Lasry and P.-L. Lions}, {\em Jeux \`a champ moyen. {I}. {L}e cas
  stationnaire}, C. R. Math. Acad. Sci. Paris, 343 (2006), pp.~619--625,
  \url{https://doi.org/10.1016/j.crma.2006.09.019}.

\bibitem{lasry2006jeuxII}
{\sc J.-M. Lasry and P.-L. Lions}, {\em Jeux \`a champ moyen. {II}. {H}orizon
  fini et contr\^{o}le optimal}, C. R. Math. Acad. Sci. Paris, 343 (2006),
  pp.~679--684, \url{https://doi.org/10.1016/j.crma.2006.09.018}.

\bibitem{lasry2007mean}
{\sc J.-M. Lasry and P.-L. Lions}, {\em Mean field games}, Jpn. J. Math., 2
  (2007), pp.~229--260, \url{https://doi.org/10.1007/s11537-007-0657-8}.

\bibitem{MonkSuli1999}
{\sc P.~Monk and E.~S\"uli}, {\em The adaptive computation of far-field
  patterns by a posteriori error estimation of linear functionals}, SIAM J.
  Numer. Anal., 36 (1999), pp.~251--274,
  \url{https://doi.org/10.1137/S0036142997315172}.

\bibitem{osborne2024analysis}
{\sc Y.~A.~P. Osborne and I.~Smears}, {\em Analysis and numerical approximation
  of stationary second-order mean field game partial differential inclusions},
  SIAM J. Numer. Anal., 62 (2024), pp.~138--166,
  \url{https://doi.org/10.1137/22M1519274}.

\bibitem{osborne2024erratum}
{\sc Y.~A.~P. Osborne and I.~Smears}, {\em Erratum: {A}nalysis and numerical
  approximation of stationary second-order mean field game partial differential
  inclusions}, SIAM J. Numer. Anal., 62 (2024), pp.~2415--2417,
  \url{https://doi.org/10.1137/24M165123X}.

\bibitem{osborne2023finite}
{\sc Y.~A.~P. Osborne and I.~Smears}, {\em Finite element approximation of
  time-dependent mean field games with nondifferentiable {H}amiltonians},
  Numer. Math., 157 (2025), pp.~165--211,
  \url{https://doi.org/10.1007/s00211-024-01447-2}.

\bibitem{osborne2024near}
{\sc Y.~A.~P. Osborne and I.~Smears}, {\em Near and full quasi-optimality of
  finite element approximations of stationary second-order mean field games},
  Math. Comp., in press (2025), \url{https://doi.org/10.1090/mcom/4080}.
\newblock (in press).

\bibitem{osborne2025regularization}
{\sc Y.~A.~P. Osborne and I.~Smears}, {\em Regularization of stationary
  second-order mean field game partial differential inclusions}, Preprint
  arXiv:2408.10810, 2025, \url{https://doi.org/10.48550/arXiv.2408.10810},
  \url{https://arxiv.org/abs/2408.10810}.

\bibitem{ScottZhang1990}
{\sc L.~R. Scott and S.~Zhang}, {\em Finite element interpolation of nonsmooth
  functions satisfying boundary conditions}, Math. Comp., 54 (1990),
  pp.~483--493, \url{https://doi.org/10.2307/2008497}.

\bibitem{SmearsSuli2014}
{\sc I.~Smears and E.~S\"{u}li}, {\em Discontinuous {G}alerkin finite element
  approximation of {H}amilton-{J}acobi-{B}ellman equations with {C}ordes
  coefficients}, SIAM J. Numer. Anal., 52 (2014), pp.~993--1016,
  \url{https://doi.org/10.1137/130909536}.

\bibitem{tabata1977}
{\sc M.~Tabata}, {\em A finite element approximation corresponding to the
  upwind finite differencing}, Mem. Numer. Math., 4 (1977), pp.~47--63.

\bibitem{TobiskaVerfurth2015}
{\sc L.~Tobiska and R.~Verf\"urth}, {\em Robust {\it a posteriori} error
  estimates for stabilized finite element methods}, IMA J. Numer. Anal., 35
  (2015), pp.~1652--1671, \url{https://doi.org/10.1093/imanum/dru060}.

\bibitem{verfurth2013posteriori}
{\sc R.~Verf\"urth}, {\em A posteriori error estimation techniques for finite
  element methods}, Numerical Mathematics and Scientific Computation, Oxford
  University Press, Oxford, 2013,
  \url{https://doi.org/10.1093/acprof:oso/9780199679423.001.0001}.

\bibitem{xu1999monotone}
{\sc J.~Xu and L.~Zikatanov}, {\em A monotone finite element scheme for
  convection-diffusion equations}, Math. Comp., 68 (1999), pp.~1429--1446,
  \url{https://doi.org/10.1090/S0025-5718-99-01148-5}.

\end{thebibliography}
\end{document}